\documentclass[12pt]{amsart}
\usepackage{amsfonts}
\usepackage{latexsym,amsmath,amsthm,amssymb,amsfonts}

\numberwithin{equation}{section}
\allowdisplaybreaks

\def\s{\,\,\,\,}

\numberwithin{equation}{section}
\newtheorem{theorem}{Theorem}[section]
\newtheorem{lem}[theorem]{Lemma}
\newtheorem{thm}[theorem]{Theorem}
\newtheorem{pro}[theorem]{Proposition}

\newtheorem{rem}[theorem]{Remark}

\newcounter{Cnumber}

\title[ ]
{\bf Higher order geometric flow of hypersurfaces in a Riemannian manifold}
\author[ ]
{Zonglin Jia\quad\quad Youde Wang}

\begin{document}
\maketitle

\begin{abstract}
In this paper, we consider the high order geometric flows of a submanifolds $M$ in a complete Riemannian manifold $N$ with $\dim(N)=\dim(M)+1=n+1$, which were introduced by Mantegazza in the case the ambient space is an Euclidean space, and extend some results due to Mantegazza to the present situation under some assumptions on $N$. Precisely, we show that if $m\in\mathbb{N}$ is strictly larger than the integer part of $n/2$ and $\varphi(t)$ is a immersion for all $t\in[0,T)$ and if $\mathfrak{F}_m(\varphi_0)$ is bounded by a constant which relies on the injectivity radius $\bar{R}>0$ and sectional curvature $\bar{K}_{\pi}(\bar{K}_{\pi}\leqslant1)$ of $N$ , then $T$ must be $\infty$.
\end{abstract}

\section{Introduction}
Let $N$ is a complete Riemannian manifold with metric $\langle\cdot|\cdot\rangle_N$. We represent orientable sub-manifolds $M$ as immersions $\Phi:M\longrightarrow N$ with   $\dim(M)=\dim(N)-1=n$. For a compact, orientable, $n$-dimensional manifold $M$ without boundary and an immersion $\Phi:M\longrightarrow(N,\langle\cdot|\cdot\rangle_N)$, we define
\[
\mathfrak{F}_m(\Phi):=\int_M(1+|\dot{\nabla}^m\Upsilon|^2)\,d\mu,
\]
where $\Phi:M\longrightarrow(N,\langle\cdot|\cdot\rangle_N)$ is an immersion with codimension 1; $\Upsilon$ is the unit normal vector field of the submanifolds $\Phi(M)$ and $\dot{\nabla}$ is the connection induced by $\Phi$. Here, $\mu$ and $\dot{\nabla}$ are the canonical measure and Riemannian connection of $(M,g)$, where the metric $g$ is obtained by pulling back on $M$ the metric $\langle\cdot|\cdot\rangle_N$ of $N$ with $\Phi$. The symbol $\dot{\nabla}^m$ denotes the $m^{th}$ iterated covariant derivative.

In local charts we have $g_{ij}:=\langle\dot{\nabla}_i\Phi|\dot{\nabla}_j\Phi\rangle$, $(g^{ij}):=(g_{ij})^{-1}$,
\[d\mu:=\sqrt{det(g_{ij})}dx^1\wedge\cdots\wedge dx^n\]
and
\[|\dot{\nabla}^m\Upsilon|^2:=g^{i_1j_1}\cdots g^{i_mj_m}\langle(\dot{\nabla}^m\Upsilon)(\frac{\partial}{\partial x^{i_1}},\cdots,\frac{\partial}{\partial x^{i_m}})|(\dot{\nabla}^m\Upsilon)(\frac{\partial}{\partial x^{j_1}},\cdots,\frac{\partial}{\partial x^{j_m}})\rangle.\]

It is easy to see that $\mathfrak{F}_1(\gamma)=\int_{\mathbb{S}^1}(1+\kappa^2)\,d\mathbb{S}^1$ as $M=\mathbb{S}^1$ (the unit circle) and $N=\mathbb{R}^2$, since the curvature $\kappa$ of a curve $\gamma:\mathbb{S}^1\longrightarrow\mathbb{R}^2$ satisfies $\kappa^2=|\nabla\nu|^2$. If $\gamma(t)$ is an closed immersed curve in $N$ then it has velocity vector $V = \nu T$ and squared geodesic curvature
$$\kappa^2=\|\nabla_T T\|^2.$$
Hence, the functional can be written as
 $$\mathfrak{F}_1(\gamma)=\int_{\mathbb{S}^1}(1+\kappa^2)\,d\mathbb{S}^1=\int_{\mathbb{S}^1}(1+\|\nabla_T T\|^2)\,d\mathbb{S}^1.$$

It is well-known that the functional $$F_1(\gamma)=\int_{\mathbb{S}^1}\kappa^2\,d\mathbb{S}^1$$ is just the total squared curvature functional of a elastic rod on which there is a long research history, for more details we refer to \cite{Sin} and references therein. When the length of curve flow is fixed, one usually calls the flow corresponding to the total squared curvature functional as curve-straightening flow. The negative gradient flow for the total squared curvature defined on curves has been widely studied in the literature. By virtue of a smoothing effect of the functional, there are various results such that the flow has a smooth solution for all times and subconverges to a (possibly nonunique) stationary solution(see \cite{NO, W} and references therein).

Simonett \cite{Si} also discussed the gradient flow of the Willmore functional (see \cite{Wi}) defined on surfaces immersed in $\mathbb{R}^3$. In \cite{KS1} and \cite{KS2}, Kuwert and Sch\"{a}tzle studied the global existence and regularity of the negative gradient flow of the Willmore functional for general initial data, where Willmore functional
\begin{eqnarray*}
\mathcal{W}(\varphi)=\int_M|A|^2\,d\mu
\end{eqnarray*}
is defined on surfaces immersed in $\mathbb{R}^3$ and $|A|=|\nabla\nu|$.

Similarly, the negative gradient flow of $$\mathfrak{F}_1(\gamma)=\int_{\mathbb{S}^1}(1+\kappa^2)\,d\mathbb{S}^1$$ is also fourth order curve flows. Now, the length of curve flow can not be fixed. For this case, the global regularity of the flow was shown by Polden in \cite{Po1} and \cite{Po2}.

Later, Mantegazza considered the negative gradient flow corresponding to $\mathfrak{F}_m(\Phi)$ where $m>1$, and proved that if the order of derivation $m \in \mathbb{N}$ is strictly larger than the integer part of $n/2$ then the singularities of this flow in finite time cannot occur during the evolution.

Actually, there have been many important works on fourth order flows of a slightly different character, from Willmore flow of surfaces to Calabi flow, a fourth order flow of metrics. Besides the above cited works of Polden, Wen and Mantegazza, we would like to quote some related work with geometric flows of high order.

Escher, Mayer and Simonett \cite{EMS} studied the surface diffusion flow (see also the references therein)
$$\frac{\partial\phi}{\partial t} = (\Delta_t H)\nu.$$
Recently, Wheeler in \cite{W} also considered closed immersed hypersurfaces evolving by surface diffusion flow, and perform an analysis based on local and global integral estimates. He showed that a properly immersed stationary ($H \equiv 0$) hypersurface in $\mathbb{R}^3$ or $\mathbb{R}^4$ with restricted growth of the curvature at infinity and small total trace-free curvature must be an embedded union of umbilic hypersurfaces. Furthermore, he show that if a singularity develops the curvature must concentrate in a definite manner, and prove that a blowup under suitable conditions converges to a nonumbilic embedded stationary surface. As a consequence, the surface diffusion flow of a surface initially close to a sphere in $L^2$ is a family of embeddings, exists for all time, and exponentially converges to a round sphere.

In the article \cite{Ch}, Chru\'{s}ciel has ever applied the global existence of a fourth order flow of metrics on a two-dimensional Riemannian manifold to the construction of solutions of Einstein vacuum equations representing an isolated gravitational system.
\medskip

In this paper we consider the negative gradient flow $\varphi:M\times[0,T)\longrightarrow N$ associated to the above functional $\mathfrak{F}_m(\varphi)$. Under some geometric assumptions on $N$, we will adopt the so-called geometric energy method to establish some similar results on the flow with that in \cite{M}. Precisely, we will show the following:
\begin{thm}\label{theorem1.1}
Let $M$ be a closed, $n$-dimensional orientable manifold and $(N,h)$ be a $(n+1)$-dimensional complete Riemannian manifold with $h=\langle\cdot|\cdot\rangle_N$. Denote the volume of the unit ball of $\mathbb{R}^n$ by $\omega_n$ and the injectivity radius of $N$ by $\bar{R}$. Suppose that $N$ is of bounded geometry, $\bar{R}>0$ and its sectional curvature $\bar{K}_{\pi}$ is smaller than or equals to $b^2$, where $b$ is a positive real number or a pure imaginary one and $b^2\leqslant1$. If the solution $\varphi$ of the negative gradient flow of $\mathfrak{F}_m$ with an initial immersion hypersurface $\varphi_0:M\longrightarrow N$ satisfies that for each $t$ in the maximal existence interval $[0,T_{\max})$, $\varphi(t)$ is an immersion, then the solution $\varphi$ is global, provided $m\geqslant[n/2]+1$ where $[n/2]$ is the integer part of $n/2$, and
\begin{eqnarray}\label{G:1}
\left\{
\begin{array}{llll}
\aligned
&\mathfrak{F}_m(\varphi_0)\leqslant\min\left\{\frac{\omega_n}{|b|^n(n+1)},\,\Big(\frac{b\bar{R}}{\pi}\Big)^n\frac{\omega_n}{n+1}\right\},\s \mbox{for}\s b\s real,\\
&\mathfrak{F}_m(\varphi_0)\leqslant\min\left\{\frac{\omega_n}{|b|^n(n+1)}, \, \frac{\bar{R}^n}{(n+1)2^n}\omega_n \right\}, \s\mbox{for}\s b \s imaginary.
\endaligned
\end{array}
\right.
\end{eqnarray}
\end{thm}

\begin{rem}\label{remark1.2}
$N$ is of bounded geometry if and only if, for any $\alpha\in\mathbb{N}$, there exists a constant $\bar{K}_{\alpha}\in\mathbb{R}^+$ such that for any tangent vector fields $X_1,\cdots,X_{\alpha}$, $Y$, $Z$, $W$ and $V$ on $N$, we have
\begin{align}\label{G:2}
|\langle(\nabla^{\alpha}R^N)(X_1,\cdots,X_{\alpha})(Y,Z)W|V\rangle_N|\leqslant\bar{K}_{\alpha}|X_1|\cdots|X_{\alpha}|\cdot|Y|\cdot|Z|\cdot|W|\cdot|V|.
\end{align}
\end{rem}

\begin{rem}\label{remark1.3}
\textbf{(1).} As $N$ is Euclidean space $\mathbb{R}^{n+1}$, it follows that $\bar{K}_{\pi}=0\leqslant\varepsilon^2$ for any $\varepsilon>0$ and $\bar{R}=\infty$. For any smooth enough $\varphi_0$, we may always choose $b=\varepsilon$ is small enough such that $\mathfrak{F}_m(\varphi_0)$ satisfies (\ref{G:1}) since $\bar{R}=\infty$. That is to say, when $N=\mathbb{R}^{n+1}$, we do not need to think about how big $\mathfrak{F}_m(\varphi_0)$ is. This property coincides with the result of \cite{M} and tells us that our conclusion is indeed stronger than that of Mantegazza.

\textbf{(2).} As $N$ is a Hadamard manifold of bounded geometry, we have $\bar{K}_{\pi}\leqslant0\leqslant\varepsilon^2$ and $\bar{R}=\infty$. For any smooth enough $\varphi_0$, (\ref{G:1}) is always satisfied.
\end{rem}

Denote the initial immersed manifold by $\varphi_0:M\longrightarrow N$. It is known that the analysis of the first variation of the functionals $\mathfrak{F}_m$ gives rise to a quasilinear system of partial differential equations on the manifold $M$ (see \cite{M}). The small time existence and uniqueness of a smooth negative gradient flow of $\mathfrak{F}_m(\varphi)$ is a particular case of a very general result of Polden proven in \cite{Po2}(also see \cite{HuP} and \cite{M1}).

We assume that such a negative gradient flow of the functional $\mathfrak{F}_m$ admits a unique solution defined on some interval $[0, T_{\dot{m}})$ and taking $\varphi_0$ as its initial value submanifold such that map $\varphi(t)=\varphi(\cdot,t)$ is an immersion for every $t\in [0, T_{\dot{m}})$. In the present situation, as in \cite{M} we need to establish suitable a priori estimates on the flow in order to obtain the long time existence.
\medskip

We will follow the route of \cite{M} to approach Theorem \ref{theorem1.1}, and take a contradiction argument. If the maximum extinction time $T_{\max}$ of the flow is not infinite, we try to apply Cauchy convergence rule to deduce contradiction. It requires that, for any $0\leqslant t_1<t_2<T_{\max}$, the derivatives of all orders of $\varphi(t_1)$ and ones of $\varphi(t_2)$ can be subtracted. Therefore, firstly we need to embed $N$ into $\mathbb{R}^{n+1+L}$. Denote the isometric embedding by $\Xi$. Let $\nabla$ be the connection induced by $\varphi(t)$(sometimes we also use $\nabla$ to denote the connection induced by $\varphi$) and $D$ denote the connection induced by $\Xi\circ\varphi(t)$(sometimes we also use $D$ to denote the connection induced by $\Xi\circ\varphi$).

Secondly, in order to show that
\begin{eqnarray*}
\max\limits_M|D_{i_1}D_{i_2}\cdots D_{i_k}D_t^s\varphi(t_1)-D_{i_1}D_{i_2}\cdots D_{i_k}D_t^s\varphi(t_2)|<\varepsilon
\end{eqnarray*}
as $|t_1-t_2|$ is small enough, we need to prove that
\begin{eqnarray*}
\max\limits_M|D_{i_1}D_{i_2}\cdots D_{i_k}D_t^{s+1}\varphi(t)|\leqslant C_{s+1,k}
\end{eqnarray*}
where $C_{s+1,k}$ is a universal constant. Since there holds true (\ref{6.92}), to show the above inequality we only need to show that
\begin{eqnarray*}
\max\limits_M|\nabla_{i_1}\nabla_{i_2}\cdots\nabla_{i_k}\nabla_t^s\varphi(t)|\leqslant C_{s,k}.
\end{eqnarray*}
We will provide a stronger result which is proved in Theorem \ref{theorem6.5}. The key ingredient of the proof is to estimate $||\nabla^pA||_{L^{\infty}(M_t)}$($A$ is the second fundamental form of $\varphi(t)$), where $M_t$ denotes $(M,g_t)$  and its metric $g_t$ is obtained by pulling back the metric $\langle\cdot|\cdot\rangle_N$ of $N$ via $\varphi(t)$.

Thirdly, because of $(\ref{5.67})$, we only need to show that $||\nabla^pA||_{L^2(M_t)}$ has a universal upper bound $C_p$. The method which we use is to consider
\begin{eqnarray*}
\frac{d}{dt}||\nabla^pA||^2_{L^2(M_t)}.
\end{eqnarray*}
By a complicated computation, we get that
\begin{eqnarray*}
\frac{d}{dt}||\nabla^pA||^2_{L^2(M_t)}\leqslant C(-||\nabla^pA||^2_{L^2(M_t)}+1).
\end{eqnarray*}
Using Gronwall inequality, we obtain the required result.

Throughout the process, once and again we need to use some universal interpolation inequalities on the Sobolev norms of $A$ which are shown in Lemma \ref{lemma5.9}, Lemma \ref{lemma5.10} and Lemma \ref{lemma5.11}. Although some similar interpolation inequalities were established in \cite{M} (see also \cite{DW1, DW2, H}), for the present situation the target manifold is a general Riemannian manifold instead of Euclidean space, it seems to be necessary to provide the detailed proofs of these lemmas.

\section{Preliminaries and Notations}\label{section2}
We devote this section to the introduction of some basic notations and facts about differentiable and Riemannian manifolds which are be used in this paper. \textbf{Throughout the paper the convention to sum over repeated indices will be adopted.} For two quantities $Q_1$ and $Q_2$, the symbols $Q_1\lesssim Q_2$ and $Q_1\thickapprox Q_2$ mean there is a universal constant $C$ such that $Q_1\leqslant C\cdot Q_2$ and $Q_1=C\cdot Q_2$ respectively.

Let $M$ be a compact orientable $n$-dimensional smooth manifold without boundary. $(N,\langle\cdot|\cdot\rangle_N)$ is a $(n+1)$-dimensional complete Riemannian manifold. Let $\Phi:M\longrightarrow N$ be an immersion, and the metric $g$ of $M$ is induced by $\Phi$. $\dot{\nabla}$ is the connection induced by $\Phi$ and $\dot{\nabla}^M$ denotes the Riemannian connection of $(M,g)$. $R^M$ and $R^N$ denote the Riemannian curvature tensors of $(M,g)$ and $(N,\langle\cdot|\cdot\rangle_N)$ respectively.

By Nash's embedding theorem, $(N,\,\langle\cdot|\cdot\rangle_N)$ is isometrically embedded into an Euclidean space $(\mathbb{R}^{n+1+L},\,\langle\cdot|\cdot\rangle_{\mathbb{R}^{n+1+L}})$ and the isometric embedding is denoted by $\Xi$.

From now on, we denote $\langle\cdot|\cdot\rangle_N$ and $\langle\cdot|\cdot\rangle_{\mathbb{R}^{n+1+L}}$ by $\langle\cdot|\cdot\rangle$, and denote the connection induced by $\Xi\circ\Phi$ by $\dot{\nabla}^{\mathbb{R}^{n+1+L}}$. Let $(x^i)$ be a local coordinate chart of $M$ and $(y^{\alpha})$ be a local coordinate chart of $N$. Let $g_{ij}:=g(\frac{\partial}{\partial x^i},\frac{\partial}{\partial x^j})$ and $(g^{ij}):=(g_{ij})^{-1}$. We also denote $\langle\frac{\partial}{\partial y^{\alpha}}|\frac{\partial}{\partial y^{\beta}}\rangle$ by $\gamma_{\alpha\beta}$. $\dot{\nabla}_i$ means $\dot{\nabla}_{\frac{\partial}{\partial x^i}}$. $\Upsilon$ is the unit normal vector field to the hypersurface $\Phi$. $\zeta$ is any vector field along $\Phi$. That is to say, for all $p\in M$, $\zeta(p)\in T_{\Phi(p)}N$.\medskip

We assume that $\varphi:M\times[0,T)\longrightarrow N$ is smooth and, for every $t\in[0,T)$, $\varphi(t)$ is an immersion. $g_t$ is the metric on $M$ obtained by pulling back the metric $\langle\cdot|\cdot\rangle$ of $N$ via $\varphi(t)$ and $\nabla$ is the connection induced by $\varphi(t)$ ($D$ is the connection induced by $\Xi\circ\varphi(t)$). Sometimes, we also use $\nabla$ to denote the connection induced by $\varphi:\,M\times[0,T)\longrightarrow N$ and use $D$ to denote the connection induced by $\Xi\circ\varphi$.

For simplicity, we denote $(M,g_t)$ by $M_t$. $\mu_t$ denotes the canonical measure of $M_t$. Sometimes, we use $||\cdot||_{L^p(\mu_t)}$ or $||\cdot||_p$ to denote $||\cdot||_{L^p(M_t)}$ and use $||\cdot||_{W^{k,p}(\mu_t)}$ to denote $||\cdot||_{W^{k,p}(M_t)}$. $\nu'(t)$ is any vector field along $\varphi(t)$. In other words, for any $p\in M$, $\nu'(p,t)\in T_{\varphi(p,t)}N$. $\nu(t)$ denotes the unit normal vector field to the hypersurface $\varphi(t)$.\medskip

For a tensor field on $M$
\[W(x,t):=W^{j_1\cdots j_p}_{i_1\cdots i_l}(x,t)dx^{i_1}\otimes\cdots\otimes dx^{i_l}\otimes\frac{\partial}{\partial x^{j_1}}\otimes\cdots\otimes\frac{\partial}{\partial x^{j_p}},
\]
which depends on time $t$, we define
\[
||\partial^s_t\partial^kW||_{L^{\infty}(M_t)}:=\max\limits_{\substack{i_1,\cdots,i_l\\j_1,\cdots,j_p\\a_1,\cdots,a_k}}\Big|\Big|\frac{\partial^s\partial^kW^{j_1\cdots j_p}_{i_1\cdots i_l}}{\partial t^s\partial x^{a_1}\cdots\partial x^{a_k}}(\cdot,t)\Big|\Big|_{L^{\infty}(M)}.
\]

Let
\[
S:=S^{\alpha_1\cdots\alpha_{\theta}}_{i_1\cdots i_l}dx^{i_1}\otimes\cdots\otimes dx^{i_l}\otimes\frac{\partial}{\partial y^{\alpha_1}}\otimes\cdots\otimes\frac{\partial}{\partial y^{\alpha_{\theta}}}
\]
and
\[
T:=T^{\beta_1\cdots\beta_{\delta}}_{j_1\cdots j_k}dx^{j_1}\otimes\cdots\otimes dx^{j_k}\otimes\frac{\partial}{\partial y^{\beta_1}}\otimes\cdots\otimes\frac{\partial}{\partial y^{\beta_{\delta}}}.
\]
We write $\langle S|T\rangle$ and $S\ast T$ to denote a tensor or a number formed by contraction on some indices of $S\otimes T$ using the coefficients $g^{ij}$ or $\gamma_{\alpha\beta}$. Since we do not specifically illustrate which indices we contract, usually
\begin{eqnarray*}
S\ast T-S\ast T\not=0
\end{eqnarray*}
and
\begin{eqnarray*}
\langle S|T\rangle-\langle S|T\rangle\not=0.
\end{eqnarray*}
So we appoint that
\begin{eqnarray*}
n(S\ast T)+S\ast T:=(n+1)S\ast T,
\end{eqnarray*}
\begin{eqnarray*}
n(S\ast T)-S\ast T:=(n+1)S\ast T,
\end{eqnarray*}
\begin{eqnarray*}
n\langle S|T\rangle+\langle S|T\rangle:=(n+1)\langle S|T\rangle,
\end{eqnarray*}
\begin{eqnarray*}
n\langle S|T\rangle-\langle S|T\rangle:=(n+1)\langle S|T\rangle.
\end{eqnarray*}
The reason why we make such appointment and the metric property of $S\ast T$ and $\langle S|T\rangle$ will be stated at the end of this section.

If $T_1,\ldots,T_l$ is a finite family of tensors, the symbol
\[\circledast_{i=1}^lT_i\]
will mean $T_1\ast T_2\ast\cdots\ast T_l$.

We will use the symbol $\mathfrak{p}_s(T_1,\cdots,T_l)$ for a polynomial in the tensors $T_1,\ldots,T_l$ and their iterated covariant derivatives with the $\ast$ product like
\[
\mathfrak{p}_s(T_1,\cdots,T_l)=\sum\limits_{i_1+\cdots+i_l=s}c_{i_1\cdots i_l}\dot{\nabla}^{i_1}T_1\ast\cdots\ast\dot{\nabla}^{i_l}T_l
\]
where the $c_{i_1\cdots i_l}$ are some real universal constants. Notice that every tensor $T_i$ must be present in every additive term of $\mathfrak{p}_s(T_1,\cdots,T_l)$ and there are no repetitions.\medskip

The second fundamental form $A:=h_{ij}dx^i\otimes dx^j$ of $\Phi$ is a 2-tensor where $h_{ij}$ is defined as follow
\[h_{ij}:=-\langle\dot{\nabla}_i\dot{\nabla}_j\Phi|\Upsilon\rangle,\]
and the mean curvature $H$ is the trace of $A$,
\[H:=g^{ij}h_{ij}.\]

Next, we will generalize the divergence theorem to the case that the ambient space is an abstract Riemannian manifold and the vector field does not have to be tangent to base manifold.

Because for all $X\in TM$, the divergence of $X$
\[div^M_{\Phi}X:=g^{ij}\langle\frac{\partial}{\partial x^i}|\dot{\nabla}^M_jX\rangle=g^{ij}\langle\dot{\nabla}_i\Phi|\dot{\nabla}_jX\rangle,\]
here we have used the fact that $\frac{\partial}{\partial x^i}=\dot{\nabla}_i\Phi$ and $\dot{\nabla}^M_jX=P(\Phi)\dot{\nabla}_jX$ where $P(\Phi(p))$ is the orthogonal projection operator from $T_{\Phi(p)}N$ to $d\Phi(T_pM)$ with $p\in M$. For a vector field $\tilde{X}$ with $\tilde{X}_p\in T_{\Phi(p)}N$ for any $p\in M$, we define its divergence
\[div^N_{\Phi}\tilde{X}:=g^{ij}\langle\dot{\nabla}_i\Phi|\dot{\nabla}_j\tilde{X}\rangle.\]
Since
\[\Xi\circ\Phi:M\longrightarrow \mathbb{R}^{n+1+L}\]
is an isometric immersion, $\tilde{X}$ can be regarded as a vector field in $T\mathbb{R}^{n+1+L}$ which is along $M$ and
\begin{equation}\label{2.1}
div^N_{\Phi}\tilde{X}=g^{ij}\langle\dot{\nabla}_i^{\mathbb{R}^{n+1+L}}(\Xi\circ\Phi)|\dot{\nabla}^{\mathbb{R}^{n+1+L}}_j\tilde{X}\rangle=div^{\mathbb{R}^{n+1+L}}_{\Xi\circ\Phi}\tilde{X},
\end{equation}
here we have used that
\begin{eqnarray*}
\dot{\nabla}_i\Phi=\dot{\nabla}_i^{\mathbb{R}^{n+1+L}}(\Xi\circ\Phi)
\end{eqnarray*}
and
\begin{eqnarray*}
\dot{\nabla}_j\tilde{X}=\mathbb{P}(\Xi\circ\Phi)(\dot{\nabla}^{\mathbb{R}^{n+1+L}}_j\tilde{X}),
\end{eqnarray*}
where $\mathbb{P}(y)$ is the orthogonal projection operator from $T_y\mathbb{R}^{n+1+L}$ to $T_yN$ with $y\in N$.

Let $\{\tau_1,\cdots,\tau_n\}$ be a local orthonormal frames of $TM$ and $\{\vec{n}_1,\cdots,\vec{n}_L\}$ be a local orthogonal basis of the normal bundle of $N$ in $\mathbb{R}^{n+1+L}$. Then $\{\tau_1,\cdots,\tau_n,\Upsilon,\vec{n}_1,\cdots,\vec{n}_L\}$ forms a orthonormal frames of $T\mathbb{R}^{n+1+L}$. The mean curvature vector $\overrightarrow{H}$ of $M$ in $\mathbb{R}^{n+1+L}$ can be written as
\[
\aligned
\overrightarrow{H}:&=\sum\limits_{i=1}^n(\dot{\nabla}_{\tau_i}^{\mathbb{R}^{n+1+L}}\tau_i)^{\bot}\\
&=\left[\sum\limits_{i=1}^n\langle\dot{\nabla}_{\tau_i}^{\mathbb{R}^{n+1+L}}\tau_i|\Upsilon\rangle\Upsilon+\sum\limits_{a=1}^k\langle\dot{\nabla}_{\tau_i}^{\mathbb{R}^{n+1+L}}
\tau_i|\vec{n}_a\rangle\vec{n}_a\right].
\endaligned
\]
Since $\tilde{X}$ is a map from $M$ to $TN$, we have $\langle\vec{n}_a|\tilde{X}\rangle=0$. It follows that
\begin{equation}\label{2.2}
\aligned
\langle\overrightarrow{H}|\tilde{X}\rangle
=&\sum\limits_{i=1}^n\langle\dot{\nabla}_{\tau_i}^{\mathbb{R}^{n+1+L}}\tau_i|\Upsilon\rangle\langle\Upsilon|\tilde{X}\rangle\\
=&\sum\limits_{i=1}^n\langle\mathbb{P}(\Xi\circ\Phi)(\dot{\nabla}_{\tau_i}^{\mathbb{R}^{n+1+L}}\tau_i)|\Upsilon\rangle\cdot\langle\Upsilon|\tilde{X}\rangle\\
=&\sum\limits_{i=1}^n\langle\dot{\nabla}_{\tau_i}\tau_i|\Upsilon\rangle\langle\Upsilon|\tilde{X}\rangle=-\sum\limits_{i=1}^n\langle\dot{\nabla}_{\tau_i}
\Upsilon|\tau_i\rangle\langle\Upsilon|\tilde{X}\rangle\\
=&-H\langle\Upsilon|\tilde{X}\rangle.
\endaligned
\end{equation}
Since $M$ is of empty boundary, by formula $(7.6)$ in the section 7 of the chapter 2 of \cite{S}, we have
\begin{equation}\label{2.3}
\int_M div^{\mathbb{R}^{n+1+L}}_{\Xi\circ\Phi}\tilde{X}\,dM=-\int_M\langle\tilde{X}|\overrightarrow{H}\rangle\,dM.
\end{equation}
Substituting $(\ref{2.1})$ and $(\ref{2.2})$ into $(\ref{2.3})$, we have
\begin{equation}\label{2.4}
\int_Mdiv^N_{\Phi}\tilde{X}\,dM=\int_MH\langle\Upsilon|\tilde{X}\rangle\,dM.
\end{equation}

For a fourth-order covariant tensor $\hat{T}:=\hat{T}_{ijkl}dx^i\otimes dx^j\otimes dx^k\otimes dx^l$ and a second-order covariant tensor $B:=B_{ij}dx^i\otimes dx^j$, we define following quantities:\medskip\\
\textbf{Quantity 1}.
\[
\mathfrak{q}^s(\underbrace{\hat{T},\cdots,\hat{T}}_{\alpha\,\,times},\underbrace{B,\cdots,B}_{\beta\,\,times}):=\sum C_{i_1\cdots i_lj_1\cdots j_k}\dot{\nabla}^{i_1}\hat{T}\ast\cdots\ast\dot{\nabla}^{i_l}\hat{T}\ast\dot{\nabla}^{j_1}B\ast\cdots\ast\dot{\nabla}^{j_k}B,
\]
where
\[s:=(i_1+2)+\cdots+(i_l+2)+(j_1+1)+\cdots+(j_k+1).\]
$\alpha$ can be arbitrary integer in $[1,(s-1)/2]$ while $\beta$ may be any integer in $[1,s-2]$.\medskip\\
\textbf{Quantity 2}.
\[
\aligned
&\mathfrak{q}^s(\underbrace{\hat{T},\cdots,\hat{T}}_{\alpha\,\,times},\underbrace{\dot{\nabla}\zeta,\cdots,\dot{\nabla}\zeta}_{\beta\,\,times},\underbrace{B,\cdots,B}_{\theta\,\,times})\\
:=&\sum C_{i_1\cdots i_lj_1\cdots j_kp_1\cdots p_q}\dot{\nabla}^{i_1}\hat{T}\ast\cdots\ast\dot{\nabla}^{i_l}\hat{T}\ast\dot{\nabla}^{j_1}B\ast\cdots\ast\dot{\nabla}^{j_k}B\ast\dot{\nabla}^{p_1}\zeta\ast\cdots\ast\dot{\nabla}^{p_q}\zeta,
\endaligned
\]
where
\[s:=(i_1+2)+\cdots+(i_l+2)+(j_1+1)+\cdots+(j_k+1)+p_1+\cdots+p_q,\]
and $p_1\geqslant1,\cdots,p_q\geqslant1$. $\alpha$ can be any integer in $[1,s/2-1]$ while $\beta$ and $\theta$ can be arbitrary integers in $[1,s-3]$.\medskip\\
\textbf{Quantity 3}.
\[R^s_1(\dot{\nabla}\Phi,\zeta):=\sum C_{a_1\cdots a_qbcde}\langle(\dot{\nabla}^qR^N)(\dot{\nabla}^{a_1+1}\Phi,\cdots,\dot{\nabla}^{a_q+1}\Phi)(\dot{\nabla}^b\zeta,\dot{\nabla}^{c+1}\Phi)\dot{\nabla}^d\zeta|\dot{\nabla}^e\zeta\rangle,\]
where
\[s:=(a_1+1)+\cdots+(a_q+1)+b+(c+1)+d+e.\]\medskip\\
\textbf{Quantity 4}.
\[R^s_2(\dot{\nabla}\Phi,\zeta):=\sum C_{a_1\cdots a_qbcde}\langle(\dot{\nabla}^qR^N)(\dot{\nabla}^{a_1+1}\Phi,\cdots,\dot{\nabla}^{a_q+1}\Phi)(\dot{\nabla}^{b+1}\Phi,\dot{\nabla}^{c+1}\Phi)\dot{\nabla}^d\zeta|\dot{\nabla}^{e+1}\Phi\rangle,\]
where
\[s:=(a_1+1)+\cdots+(a_q+1)+(b+1)+(c+1)+d+(e+1).\]\medskip\\
\textbf{Quantity 5}.
\[R^s_3(\dot{\nabla}\Phi,\zeta):=\sum C_{a_1\cdots a_qbcde}\langle(\dot{\nabla}^qR^N)(\dot{\nabla}^{a_1+1}\Phi,\cdots,\dot{\nabla}^{a_q+1}\Phi)(\dot{\nabla}^{b+1}\Phi,\dot{\nabla}^{c+1}\Phi)\dot{\nabla}^{d+1}\Phi|
\dot{\nabla}^e\zeta\rangle,\]
where
\[s:=(a_1+1)+\cdots+(a_q+1)+(b+1)+(c+1)+(d+1)+e.\]\medskip\\
\textbf{Quantity 6}.
\[R^s_4(\dot{\nabla}\Phi,\zeta):=\sum C_{a_1\cdots a_qbcde}\langle(\dot{\nabla}^qR^N)(\dot{\nabla}^{a_1+1}\Phi,\cdots,\dot{\nabla}^{a_q+1}\Phi)(\dot{\nabla}^b\zeta,\dot{\nabla}^{c+1}\Phi)
\dot{\nabla}^{d+1}\Phi|\dot{\nabla}^e\zeta\rangle,\]
where
\[s:=(a_1+1)+\cdots+(a_q+1)+b+(c+1)+(d+1)+e.\]\medskip\\
\textbf{Quantity 7}.
\[\mathfrak{q}^s(\underbrace{\dot{\nabla}\zeta,\cdots,\dot{\nabla}\zeta}_{\alpha\,\,times},\underbrace{B,\cdots,B}_{\beta\,\,times}):=\sum C_{j_1\cdots j_kp_1\cdots p_q}\dot{\nabla}^{j_1}B\ast\cdots\ast\dot{\nabla}^{j_k}B\ast\dot{\nabla}^{p_1}\zeta\ast\cdots\ast\dot{\nabla}^{p_q}\zeta,\]
where
\[s:=(j_1+1)+\cdots+(j_k+1)+p_1+\cdots+p_q,\]
and $p_1\geqslant1,\cdots,p_q\geqslant1$. $\alpha$ and $\beta$ can be arbitrary integers in $[1,\,s-1]$.
\medskip

\subsection{Some fundamental equations and differential relations}
Now we recall some formula or equations. Since those proofs are tedious and omitted, we list directly the facts as follows.

\noindent\textbf{Gauss equation:}
\begin{equation}\label{2.5}
\aligned
R^M_{lkij}:=&g\Big(R^M(\frac{\partial}{\partial x^i},\frac{\partial}{\partial x^j})\frac{\partial}{\partial x^k},\frac{\partial}{\partial x^l}\Big)\\
=&\langle R^N(\dot{\nabla}_i\Phi,\dot{\nabla}_j\Phi)\dot{\nabla}_k\Phi|\dot{\nabla}_l\Phi\rangle-h_{ik}h_{jl}+h_{il}h_{jk},
\endaligned
\end{equation}
which is also equivalent to
\begin{equation}\label{2.6}
(R^M)^a_{ljk}=g^{ai}\langle R^N(\dot{\nabla}_j\Phi,\dot{\nabla}_k\Phi)\dot{\nabla}_l\Phi|\dot{\nabla}_i\Phi\rangle-g^{ai}h_{ik}h_{jl}+g^{ai}h_{ij}h_{lk},
\end{equation}
where $(R^M)^a_{ljk}$ is defined as follow:
\begin{eqnarray*}
R^M(\frac{\partial}{\partial x^j},\frac{\partial}{\partial x^k})\frac{\partial}{\partial x^l}:=(R^M)^a_{ljk}\frac{\partial}{\partial x^a}.
\end{eqnarray*}
\medskip

\noindent\textbf{Codazzi equation:}
\begin{equation}\label{2.7}
h_{jk,i}-h_{ik,j}=(R^N)^{\Upsilon}_{ijk}:=\langle R^N(\dot{\nabla}_i\Phi,\dot{\nabla}_j\Phi)\dot{\nabla}_k\Phi|\Upsilon\rangle.
\end{equation}
\medskip

\noindent\textbf{Gauss-Weingarten relations:}
\begin{equation}\label{2.8}
\dot{\nabla}_i\Upsilon=h_{iq}g^{qp}\dot{\nabla}_p\Phi,
\end{equation}
\begin{equation}\label{2.9}
\dot{\nabla}_j\dot{\nabla}_i\Phi=\Gamma^k_{ij}\dot{\nabla}_k\Phi-h_{ij}\Upsilon
\end{equation}
where $\Gamma^k_{ij}$ is the connection coefficient of $\dot{\nabla}$,
\[
\Gamma^k_{ij}:=\frac{1}{2}g^{kl}\{\frac{\partial g_{il}}{\partial x^j}+\frac{\partial g_{jl}}{\partial x^i}-\frac{\partial g_{ij}}{\partial x^l}\}.
\]
\medskip

\noindent\textbf{Simon identity:}
By the definition we can prove that:
\begin{equation}\label{2.10}
(R^N)^{\Upsilon}_{ijk,l}=\langle(\dot{\nabla} R^N)(\dot{\nabla}_l\Phi)(\dot{\nabla}_i\Phi,\dot{\nabla}_j\Phi)\dot{\nabla}_k\Phi|\Upsilon\rangle+\langle R^N(\dot{\nabla}_i\Phi,\dot{\nabla}_j\Phi)\dot{\nabla}_k\Phi|\dot{\nabla}_q\Phi\rangle h_{lp}g^{pq}.
\end{equation}
Combining $(\ref{2.7})$ and $(\ref{2.10})$, we get the following Simon identity:
\begin{equation}\label{2.11}
\aligned
H,_{ij}
=&\Delta h_{ij}+|A|^2h_{ij}-h_{is}g^{sr}h_{rj}H\\
&+g^{pq}g^{sr}h_{sq}\langle R^N(\dot{\nabla}_r\Phi,\dot{\nabla}_i\Phi)\dot{\nabla}_p\Phi|\dot{\nabla}_j\Phi\rangle\\
&+g^{pq}g^{sr}h_{is}\langle R^N(\dot{\nabla}_r\Phi,\dot{\nabla}_q\Phi)\dot{\nabla}_p\Phi|\dot{\nabla}_j\Phi\rangle\\
&+g^{pq}g^{sr}h_{sj}\langle R^N(\dot{\nabla}_i\Phi,\dot{\nabla}_p\Phi)\dot{\nabla}_q\Phi|\dot{\nabla}_r\Phi\rangle\\
&+g^{pq}g^{sr}h_{pr}\langle R^N(\dot{\nabla}_j\Phi,\dot{\nabla}_q\Phi)\dot{\nabla}_i\Phi|\dot{\nabla}_s\Phi\rangle\\
&+g^{pq}\langle(\dot{\nabla} R^N)(\dot{\nabla}_j\Phi)(\dot{\nabla}_i\Phi,\dot{\nabla}_p\Phi)\dot{\nabla}_q\Phi|\Upsilon\rangle\\
&+g^{pq}\langle(\dot{\nabla} R^N)(\dot{\nabla}_p\Phi)(\dot{\nabla}_j\Phi,\dot{\nabla}_q\Phi)\dot{\nabla}_i\Phi|\Upsilon\rangle.
\endaligned
\end{equation}

Later, we will use the {\it ``Relations of changing the order of derivatives"} with respect to time variable and space variables:
\begin{equation}\label{2.12}
\nabla_t\nabla_i\nu'=\nabla_i\nabla_t\nu'+R^N(\nabla_t\varphi,\nabla_i\varphi)\nu'.
\end{equation}

For $s\geqslant2$, let $\vec{A}:=(i_1,\cdots,i_s)$,

\[
\nabla_{\vec{A}}\nu':=(\nabla^s\nu')(\frac{\partial}{\partial x^{i_1}},\cdots,\frac{\partial}{\partial x^{i_s}}),\s\s\s \tilde{\nabla}_{\vec{A}}\nu':=\nabla_{i_1}\nabla_{i_2}\cdots\nabla_{i_s}\nu',
\]

\[
\nabla_{\vec{A}}\varphi:=(\nabla^s\varphi)(\frac{\partial}{\partial x^{i_1}},\cdots,\frac{\partial}{\partial x^{i_s}}) \s\s\mbox{and}\s\s \tilde{\nabla}_{\vec{A}}\varphi:=\nabla_{i_1}\nabla_{i_2}\cdots\nabla_{i_s}\varphi.
\]
Then we obtain
\begin{equation}\label{2.13}
\aligned
\nabla_t\nabla_{\vec{A}}\nu'=&\nabla_{\vec{A}}\nabla_t\nu'+\sum(\nabla^qR^N)(\nabla_{\vec{A}_1}\varphi,\cdots,\nabla_{\vec{A}_q}\varphi)(\nabla_{\vec{B}}
\nabla_t\varphi,\nabla_{\vec{C}}\varphi)\nabla_{\vec{D}}\nu'\\
&-\mathfrak{p}_{s-2}(\nabla\nu',\nabla a(X)),
\endaligned
\end{equation}
where
\[\mathfrak{p}_{s-2}(\nabla\nu',\nabla a(X)):=\sum\limits_{k=2}^s\nabla_{i_{k+1}\cdots i_s}\Big(\sum\limits_{p=1}^{k-1}\frac{\partial\Gamma^d_{i_pi_k}}{\partial t}\nabla_{i_1\cdots i_{p-1}di_{p+1}\cdots i_{k-1}}\nu'\Big)\]
($\Gamma_{i_pi_k}^d$ is the connection coefficient of $\nabla$), and for $1\leqslant\theta\leqslant q$, $0\leqslant q\leqslant s-1$,
\[\vec{A}_{\theta}:=(i_{a_{\theta1}},\cdots,i_{a_{\theta n_{\theta}}}), \s a_{\theta1}<a_{\theta2}<\cdots<a_{\theta n_{\theta}},\s 1\leqslant n_{\theta}\leqslant s-1;\]
\[\vec{B}:=(i_{b_1},\cdots,i_{b_m}),\s b_1<b_2<\cdots<b_m,\s 0\leqslant m\leqslant s-1;\]
\[\vec{C}:=(i_{c_1},\cdots,i_{c_e}),\s c_1<c_2<\cdots<c_e, \s 1\leqslant e\leqslant s;\]
\[\vec{D}:=(i_{d_1},\cdots,i_{d_f}),\s d_1<d_2<\cdots<d_f,\s 0\leqslant f\leqslant s-1;\]
and
\[\bigcup\limits_{\theta=1}^q\{i_{a_{\theta1}},\cdots,i_{a_{\theta n_{\theta}}} \}\bigcup\{i_{b_1},\cdots,i_{b_m}\}\bigcup\{i_{c_1},\cdots,i_{c_e}\}\bigcup\{i_{d_1},\cdots,i_{d_f}\}=\{i_1,\cdots,i_s\}.\]
\medskip

\noindent\textbf{Covariant derivatives of curvature tensor of $N$ \uppercase\expandafter{\romannumeral1}:}
\begin{equation}\label{2.14}
\aligned
&[(\dot{\nabla}^mR^N)(\dot{\nabla}_{\vec{A}_1}\Phi,\cdots,\dot{\nabla}_{\vec{A}_m}\Phi)(\dot{\nabla}_{\vec{B}}\Phi,\dot{\nabla}_{\vec{C}}\Phi)\dot{\nabla}_{\vec{D}}\zeta],_e\\
=&(\dot{\nabla}^{m+1}R^N)(\dot{\nabla}_{\vec{A}_1}\Phi,\cdots,\dot{\nabla}_{\vec{A}_m}\Phi,\dot{\nabla}_e\Phi)(\dot{\nabla}_{\vec{B}}\Phi,\dot{\nabla}_{\vec{C}}\Phi)\dot{\nabla}_{\vec{D}}\zeta\\
&+\sum\limits_{q=1}^m(\dot{\nabla}^mR^N)(\dot{\nabla}_{\vec{A}_1}\Phi,\cdots,\dot{\nabla}_{(\vec{A}_q,e)}\Phi,\cdots,\dot{\nabla}_{\vec{A}_m}\Phi)(\dot{\nabla}_{\vec{B}}\Phi,\dot{\nabla}_{\vec{C}}\Phi)\dot{\nabla}_{\vec{D}}\zeta\\
&+(\dot{\nabla}^mR^N)(\dot{\nabla}_{\vec{A}_1}\Phi,\cdots,\dot{\nabla}_{\vec{A}_m}\Phi)(\dot{\nabla}_{(\vec{B},e)}\Phi,\dot{\nabla}_{\vec{C}}\Phi)\dot{\nabla}_{\vec{D}}\zeta\\
&+(\dot{\nabla}^mR^N)(\dot{\nabla}_{\vec{A}_1}\Phi,\cdots,\dot{\nabla}_{\vec{A}_m}\Phi)(\dot{\nabla}_{\vec{B}}\Phi,\dot{\nabla}_{(\vec{C},e)}\Phi)\dot{\nabla}_{\vec{D}}\zeta\\
&+(\dot{\nabla}^mR^N)(\dot{\nabla}_{\vec{A}_1}\Phi,\cdots,\dot{\nabla}_{\vec{A}_m}\Phi)(\dot{\nabla}_{\vec{B}}\Phi,\dot{\nabla}_{\vec{C}}\Phi)\dot{\nabla}_{(\vec{D},e)}\zeta
\endaligned
\end{equation}
where
\[\vec{D}=(d_1,\cdots,d_l)\]
and
\[(\vec{D},e)=(d_1,\cdots,d_l,e).\]
\textbf{Ricci identity:}
\begin{equation}\label{2.15}
\aligned
&\dot{\nabla}_{i_1\cdots i_mkl}\zeta-\dot{\nabla}_{i_1\cdots i_mlk}\zeta\\
=&R^N(\dot{\nabla}_l\Phi,\dot{\nabla}_k\Phi)\dot{\nabla}_{i_1\cdots i_m}\zeta+\sum\limits_{r=1}^m\dot{\nabla}_{i_1\cdots i_{r-1}si_{r+1}\cdots i_m}\zeta\cdot(R^M)^s_{i_rkl}.
\endaligned
\end{equation}
\textbf{Covariant derivative of curvature tensor of $M$:}
\begin{equation}\label{2.16}
(\dot{\nabla}^mR^M)(\frac{\partial}{\partial x^{p_1}},\cdots,\frac{\partial}{\partial x^{p_m}})(\frac{\partial}{\partial x^i},\frac{\partial}{\partial x^j})\frac{\partial}{\partial x^k}=(R^M)^a_{ijk,p_1\cdots p_m}\frac{\partial}{\partial x^a}.
\end{equation}
Here $(R^M)^a_{ijk,p_1\cdots p_m}$ is defined inductively as follows:
\medskip

Firstly, we define
\[(R^M)^a_{ijk,l}:=\frac{\partial(R^M)^a_{ijk}}{\partial x^l}-(R^M)^a_{sjk}\Gamma^s_{li}-(R^M)^a_{isk}\Gamma^s_{jl}-(R^M)^a_{ijs}\Gamma^s_{lk}+(R^M)^b_{ijk}\Gamma^a_{bl}.\]
If $(R^M)^a_{ijk,p_1\cdots p_m}$ has been defined, then we define
\[
\aligned
(R^M)^a_{ijk,p_1\cdots p_mp_{m+1}}:=&\frac{\partial(R^M)^a_{ijk,p_1\cdots p_m}}{\partial x^{p_{m+1}}}-(R^M)^a_{sjk,p_1\cdots p_m}\Gamma^s_{p_{m+1}i}-(R^M)^a_{isk,p_1\cdots p_m}\Gamma^s_{jp_{m+1}}\\
&-(R^M)^a_{ijs,p_1\cdots p_m}\Gamma^s_{p_{m+1}k}+(R^M)^b_{ijk,p_1\cdots p_m}\Gamma^a_{bp_{m+1}}\\
&-\sum\limits_{r=1}^m\Gamma^s_{p_rp_{m+1}}(R^M)^a_{ijk,p_1\cdots p_{r-1}sp_{r+1}\cdots p_m}.
\endaligned
\]

\noindent\textbf{Covariant derivative of curvature tensor of $N$ \uppercase\expandafter{\romannumeral2}:}\\
Let
\begin{equation}\label{2.17}
\mathfrak{R}_{ijkl}:=\langle R^N(\dot{\nabla}_j\Phi,\dot{\nabla}_k\Phi)\dot{\nabla}_l\Phi|\dot{\nabla}_i\Phi\rangle.
\end{equation}
Then
\begin{equation}\label{2.18}
\mathfrak{R}_{ijkl,p_1\cdots p_m}=\sum\langle(\dot{\nabla}^qR^N)(\dot{\nabla}_{\vec{A}_1}\Phi,\cdots,\dot{\nabla}_{\vec{A}_q}\Phi)
(\dot{\nabla}_{\vec{B}}\Phi,\dot{\nabla}_{\vec{C}}\Phi)\dot{\nabla}_{\vec{D}}\Phi|\dot{\nabla}_{\vec{E}}\Phi\rangle,
\end{equation}
where
\[(\vec{A}_1,\cdots,\vec{A}_q,\vec{B},\vec{C},\vec{D},\vec{E})=\sigma(i,j,k,l,p_1,\cdots,p_m)\]
and $\sigma$ is a permutation.

Using $(\ref{2.15})$, we get {\it "the relation of changing order of derivatives about space"}
\begin{equation}\label{2.19}
\aligned
\dot{\nabla}_{j_1\cdots j_mi_m\cdots i_1}\zeta
=&\dot{\nabla}_{j_1i_1j_2i_2\cdots j_mi_m}\zeta+\mathfrak{p}_{2m-3}(\dot{\nabla}\zeta,\mathfrak{R}+A\otimes A)\\
&+\sum(\dot{\nabla}^qR^N)(\dot{\nabla}_{\vec{A}_1}\Phi,\cdots,\dot{\nabla}_{\vec{A}_q}\Phi)(\dot{\nabla}_{\vec{B}}\Phi,\dot{\nabla}_{\vec{C}}\Phi)\dot{\nabla}_{\vec{D}}\zeta\\
\endaligned
\end{equation}
where
\begin{equation}\label{2.20}
\mathfrak{p}_s(\dot{\nabla}\zeta,\mathfrak{R}+A\otimes A):=\sum\limits_{i+j=s}C_{ij}\dot{\nabla}^{i+1}\zeta\ast\dot{\nabla}^j(\mathfrak{R}+A\otimes A),
\end{equation}
\[\mathfrak{p}_{-1}(\dot{\nabla}\zeta,\mathfrak{R}+A\otimes A):=0,\]
and
\[(\vec{A}_1,\cdots,\vec{A}_q,\vec{B},\vec{C},\vec{D})=\sigma(j_1,\cdots,j_m,i_m,\cdots,i_1)\]
where $\sigma$ is a permutation and $C_{ij}$ are some universal constants.
\medskip

\noindent\textbf{The Divergences of some vector fields along $\Phi$:}\medskip\\
I. From $(\ref{2.19})$, we get
\begin{equation}\label{2.21}
\aligned
&div^N_{\Phi}(g^{i_1j_1}\cdots g^{i_mj_m}\dot{\nabla}_{j_1\cdots j_mi_m\cdots i_1}\zeta)\\
=&g^{kl}\langle\dot{\nabla}_k\Phi|(\Delta^m\zeta),_l\rangle + g^{kl}g^{i_1j_1}\cdots g^{i_mj_m}\langle\dot{\nabla}_k\Phi|\mathfrak{p}_{2m-2}(\dot{\nabla}\zeta,\mathfrak{R}+A\otimes A)\rangle\\
&+\sum g^{kl}g^{i_1j_1}\cdots g^{i_mj_m}\langle\dot{\nabla}_k \Phi|(\dot{\nabla}^{q'}R^N)(\dot{\nabla}_{\vec{A}'_1}\Phi,\cdots,\dot{\nabla}_{\vec{A}'_{q'}}\Phi) (\dot{\nabla}_{\vec{B}'}\Phi,\dot{\nabla}_{\vec{C}'}\Phi)\dot{\nabla}_{\vec{D}'}\zeta\rangle\\
=&g^{kl}\langle\dot{\nabla}_k\Phi|(\Delta^m\zeta),_l\rangle +\sum\limits_{i+j=2m-2}C_{ij}\langle\dot{\nabla}^{i+1}\zeta|\dot{\nabla}\Phi\rangle\ast\dot{\nabla}^j(\mathfrak{R}
+A\otimes A)\\
&+\sum g^{kl}g^{i_1j_1}\cdots g^{i_mj_m}\langle\dot{\nabla}_k\Phi|(\dot{\nabla}^{q'}R^N)(\dot{\nabla}_{\vec{A}'_1}\Phi,\cdots,\dot{\nabla}_{\vec{A}'_{q'}}\Phi)
(\dot{\nabla}_{\vec{B}'}\Phi,\dot{\nabla}_{\vec{C}'}\Phi)\dot{\nabla}_{\vec{D}'}\zeta\rangle,
\endaligned
\end{equation}
where
\[(\vec{A}'_1,\cdots,\vec{A}'_{q'},\vec{B}',\vec{C}',\vec{D}')=\sigma(j_1,\cdots,j_m,i_m,\cdots,i_1,l)\]
and $\sigma$ is a permutation.\medskip

\noindent II. By a direct calculation, one can easily get
\begin{equation}\label{2.22}
div^N_{\Phi}\Upsilon=H.
\end{equation}

\noindent III. Using $(\ref{2.7})$, we obtain
\begin{equation}\label{2.23}
\aligned
div^N_{\Phi}(\Delta\Upsilon)
=&\Delta(div^N_{\Phi}\Upsilon)-H\cdot|A|^2+g^{kl}g^{ij}g^{sa}h_{sk}\langle R^N(\dot{\nabla}_i\Phi,\dot{\nabla}_j\Phi)\dot{\nabla}_l\Phi|\dot{\nabla}_a\Phi\rangle\\
&+g^{kl}g^{ij}\sum\langle\dot{\nabla}_k\Phi|(\dot{\nabla}^qR^N)(\dot{\nabla}_{\vec{A}_1}\Phi,\cdots,\dot{\nabla}_{\vec{A}_q}\Phi)(\dot{\nabla}_{(l,\vec{B})}\Phi,
\dot{\nabla}_{\vec{C}}\Phi)\dot{\nabla}_{\vec{D}}\Upsilon\rangle
\endaligned
\end{equation}
where
\[(\vec{A}_1,\cdots,\vec{A}_q,\vec{B},\vec{C},\vec{D})=\sigma(i,j)\]
and $\sigma$ is a permutation.
\medskip

\noindent IV.
\begin{equation}\label{2.24}
\aligned
div^N_{\Phi}(\Delta\zeta)
=&\Delta(div^N_{\Phi}\zeta)+g^{kl}g^{ij}h_{ki}\langle\Upsilon|\dot{\nabla}_{lj}\zeta\rangle\\
&+g^{kl}g^{ij}h_{kj}\langle\Upsilon|\dot{\nabla}_{li}\zeta\rangle+g^{kl}g^{ij}\langle h_{ki,j}\Upsilon+h_{ki}\dot{\nabla}_j\Upsilon|\dot{\nabla}_l\zeta\rangle\\
&+g^{kl}g^{ij}\langle\dot{\nabla}_k\Phi|(\dot{\nabla} R^N)(\dot{\nabla}_j\Phi)(\dot{\nabla}_l\Phi,\dot{\nabla}_i\Phi)\zeta+R^N(\dot{\nabla}_{lj}\Phi,\dot{\nabla}_i\Phi)\zeta\\
&+R^N(\dot{\nabla}_l\Phi,\dot{\nabla}_{ij}\Phi)\zeta+R^N(\dot{\nabla}_l\Phi,\dot{\nabla}_i\Phi)\dot{\nabla}_j\zeta+R^N(\dot{\nabla}_l\Phi,\dot{\nabla}_j\Phi)\dot{\nabla}_i\zeta\rangle\\
&+g^{kl}g^{ij}g^{sa}\langle\dot{\nabla}_k\Phi|\dot{\nabla}_s\zeta\rangle\langle R^N(\dot{\nabla}_i\Phi,\dot{\nabla}_j\Phi)\dot{\nabla}_l\Phi|\dot{\nabla}_a\Phi\rangle\\
&+g^{ij}g^{sa}h_{aj}\langle\dot{\nabla}_i\Upsilon|\dot{\nabla}_s\zeta\rangle-H\cdot g^{sa}\langle\dot{\nabla}_a\Upsilon|\dot{\nabla}_s\zeta\rangle.
\endaligned
\end{equation}

\noindent V. From $(\ref{2.23})$ and $(\ref{2.24})$ we can get inductively
\begin{equation}\label{2.25}
\aligned
div^N_{\Phi}(\Delta^m\Upsilon)
=&\Delta^m(div^N_{\Phi}\Upsilon)+\sum\limits_{i+j+k=2m-2}C_{ijk}\dot{\nabla}^iA\ast\langle\dot{\nabla}^{j+1}\Upsilon|\dot{\nabla}^{k+1}\Upsilon\rangle\\
&+\sum\limits_{i+j+k=2m-2}C_{ijk}\langle\dot{\nabla}^{i+1}\Phi|\dot{\nabla}^{j+1}\Upsilon\rangle\ast\dot{\nabla}^k\mathfrak{R}\\
&+\sum\limits_{\substack{a_1+\cdots+a_q\\+q+b+c+d+e\\
=2m-2}}C_{a_1\cdots a_qbcde}\langle\dot{\nabla}^{e+1}\Phi\\
&|(\dot{\nabla}^qR^N)(\dot{\nabla}^{a_1+1}\Phi,\cdots,\dot{\nabla}^{a_q+1}\Phi)(\dot{\nabla}^{b+1}\Phi,\dot{\nabla}^{c+1}\Phi)\dot{\nabla}^{d+1}\Upsilon\rangle.
\endaligned
\end{equation}

\begin{rem}\label{remark2.2}
Later, we usually use the following inequalities
\begin{eqnarray}\label{2.31}
|S\ast T|\leqslant|S|\cdot|T|
\end{eqnarray}
and
\begin{eqnarray}\label{2.32}
|\langle S|T\rangle|\leqslant|S|\cdot|T|.
\end{eqnarray}
This can be easily seen by choosing respectively an orthonormal basis in the tangent space at a point of $M$ and the tangent space at the corresponding point of $N$. In such two coordinate charts we have
\begin{eqnarray*}
\aligned
|S\ast T|^2&=\sum\limits_{\substack{free\\indices}}\Big(\sum\limits_{\substack{contracted\\indices}}S^{\alpha_1\cdots\alpha_{\theta}}_{i_1\cdots i_l}\cdot T^{\beta_1\cdots\beta_{\delta}}_{j_1\cdots j_k}\Big)^2\\
&\leqslant\sum\limits_{\substack{free\\indices}}\Big[\sum\limits_{\substack{contracted\\indices}}(S^{\alpha_1\cdots\alpha_{\theta}}_{i_1\cdots i_l})^2\Big]\cdot\Big[\sum\limits_{\substack{contracted\\indices}}(T^{\beta_1\cdots\beta_{\delta}}_{j_1\cdots j_k})^2\Big]\\
&\leqslant\Big[\sum\limits_{\substack{free\\indices}}\sum\limits_{\substack{contracted\\indices}}(S^{\alpha_1\cdots\alpha_{\theta}}_{i_1\cdots i_l})^2\Big]\cdot\Big[\sum\limits_{\substack{free\\indices}}\sum\limits_{\substack{contracted\\indices}}(T^{\beta_1\cdots\beta_{\delta}}_{j_1\cdots j_k})^2\Big]\\
&=|S|^2\cdot|T|^2.
\endaligned
\end{eqnarray*}
Using the same approach, we can easily get $(\ref{2.32})$. Comparing the following inequality
\begin{eqnarray*}
\aligned
|n(S\ast T)-(S\ast T)|
\leqslant |n(S\ast T)|+|S\ast T|\leqslant & n|S|\cdot|T|+|S|\cdot|T|\\
=&(n+1)|S|\cdot|T|
\endaligned
\end{eqnarray*}
with
\begin{eqnarray*}
|(n+1)S\ast T|\leqslant(n+1)|S|\cdot|T|,
\end{eqnarray*}
one will see why the result of $n(S\ast T)-(S\ast T)$ is denoted by $(n+1)(S\ast T)$ instead of $(n-1)(S\ast T)$. As for $n\langle S|T\rangle-\langle S|T\rangle$, the reason is similar.
\end{rem}

\subsection{Computation on the first variation of $\mathfrak{F}_m$}
In the previous we have defined that $\varphi_t:M\longrightarrow N$ is a one-parameter family of immersions. From now on, we use $div^N$ to denote $div^N_{\varphi_t}$ for short and write $g_t$ as $g$. $H$ and $A$ are the mean curvature and the second fundamental form of $\varphi_t$ respectively.

Setting $X_p:=\frac{\partial}{\partial t}\varphi_t(p)\Big|_{t=0}$, one can obtain a vector field along $M$ as a submanifold of $N$ via $\varphi_0$.
Then
\begin{equation}\label{3.1}
\aligned
\frac{d}{dt}\sqrt{det(g_{ij})}\Big|_{t=0}
=&\frac{1}{2\sqrt{det(g_{ij})}}\cdot\frac{d}{dt}\Big[det(g_{ij})\Big]\Big|_{t=0}\\
=&\frac{1}{2}\frac{dg_{ij}}{dt}\Big|_{t=0}\cdot g^{ij}\sqrt{det(g_{ij})}\\
=&g^{ij}\langle\nabla_t\nabla_i\varphi_t|\nabla_j\varphi_t\rangle\Big|_{t=0}\sqrt{det(g_{ij})}\\
=&g^{ij}\langle\nabla_i\nabla_t\varphi_t\Big|_{t=0}|\nabla_j\varphi_0\rangle\sqrt{det(g_{ij})}\\
=&g^{ij}\langle\nabla_iX|\nabla_j\varphi_0\rangle\sqrt{det(g_{ij})}\\
=&div^NX\sqrt{det(g_{ij})}.
\endaligned
\end{equation}
So it follows that
\begin{equation}\label{3.2}
\aligned
&\frac{d}{dt}\mathfrak{F}_m(\varphi_t)\Big|_{t=0}\\
=&\int_M(1+|\nabla^m\nu|^2)div^NX\,d\mu_0+\int_M\frac{\partial}{\partial t}|\nabla^m\nu|^2\,d\mu_0\\
=&\int_M\Big\{div^N[(1+|\nabla^m\nu|^2)X]-\langle \mbox{grad}^M(|\nabla^m\nu|^2)|X\rangle\Big\}\,d\mu_0+\int_M\frac{\partial}{\partial t}|\nabla^m\nu|^2\,d\mu_0.
\endaligned
\end{equation}
Here, for a function $u\in C^1(M)$, the gradient of $u$ with respect to the metric on $M$ is defined as
\begin{eqnarray*}
\mbox{grad}^Mu:=\frac{\partial u}{\partial x^i}g^{ij}\frac{\partial}{\partial x^j}=\frac{\partial u}{\partial x^i}g^{ij}\nabla_j\varphi.
\end{eqnarray*}
In $(\ref{2.4})$, taking $\tilde{X}=(1+|\nabla^m\nu|^2)X$, we have
\begin{equation}\label{3.3}
\aligned
\frac{d}{dt}\mathfrak{F}_m(\varphi_t)\Big|_{t=0}
=&-\int_M\langle\mbox{grad}^M(|\nabla^m\nu|^2)|X\rangle\,d\mu_0+\int_M(1+|\nabla^m\nu|^2)H\langle X|\nu\rangle\,d\mu_0\\
&+\int_M\frac{\partial}{\partial t}|\nabla^m\nu|^2\,d\mu_0.
\endaligned
\end{equation}
Now, we need to compute the derivatives in the last term on the right-hand side of $(\ref{3.3})$. For the metric tensor $g_{ij}$, let
\begin{equation}\label{3.4}
\aligned
a_{ij}(X):&=\frac{\partial g_{ij}}{\partial t}\Big|_{t=0}=\langle\nabla_t\nabla_i\varphi_t\Big|_{t=0}|\nabla_j\varphi_0\rangle+\langle\nabla_t\nabla_j\varphi_t\Big|_{t=0}|\nabla_i\varphi_0\rangle\\
&=\langle\nabla_iX|\nabla_j\varphi_0\rangle+\langle\nabla_jX|\nabla_i\varphi_0\rangle.
\endaligned
\end{equation}
Differentiating the formula $g_{is}g^{sj}=\delta_i^j$, we get
\begin{equation}\label{3.5}
\frac{\partial g^{ij}}{\partial t}\Big|_{t=0}=-g^{is}\cdot\frac{\partial}{\partial t}g_{sl}\Big|_{t=0}\cdot g^{lj}=-g^{is}a_{sl}(X)g^{lj}.
\end{equation}
It is well-known that the derivative of $\nu$ is tangent to $M$, so it is given by
\begin{equation}\label{3.6}
\aligned
b(X):&=\nabla_t\nu\Big|_{t=0}=\langle\nabla_t\nu\Big|_{t=0}|\nabla_i\varphi_0\rangle g^{ij}\nabla_j\varphi_0\\
&=-\langle\nu|\nabla_t\nabla_i\varphi_t\Big|_{t=0}\rangle g^{ij}\nabla_j\varphi_0=-\langle\nu|\nabla_iX\rangle g^{ij}\nabla_j\varphi_0\\
&=-\frac{\partial\langle\nu|X\rangle}{\partial x^i}g^{ij}\nabla_j\varphi_0+\langle\nabla_i\nu|X\rangle g^{ij}\nabla_j\varphi_0\\
&=-\mbox{grad}^M(\langle\nu|X\rangle)+\langle\nabla_p\varphi_0|X\rangle\nabla_j\varphi_0\cdot h_{iq}g^{qp}g^{ij}.
\endaligned
\end{equation}
Using the same method as in page 147 of \cite{M}, one can easily prove that
\begin{equation}\label{3.7}
\frac{\partial}{\partial t}\Gamma^i_{jk}=\frac{1}{2}g^{il}\Big\{a_{kl,j}(X)+a_{jl,k}(X)-a_{jk,l}(X)\Big\}.
\end{equation}
Hence
\begin{eqnarray*}
\aligned
\frac{\partial}{\partial t}|\nabla^m\nu|^2=&\frac{\partial}{\partial t}\Big(g^{i_1j_1}\cdots g^{i_mj_m}\langle\nabla_{i_1\cdots i_m}\nu|\nabla_{j_1\cdots j_m}\nu\rangle\Big)\\
=&\sum\limits_{l=1}^mg^{i_1j_1}\cdots\frac{\partial}{\partial t}g^{i_lj_l}\cdots g^{i_mj_m}\langle\nabla_{i_1\cdots i_m}\nu|\nabla_{j_1\cdots j_m}\nu\rangle\\
&+2g^{i_1j_1}\cdots g^{i_mj_m}\langle\nabla_t\nabla_{i_1\cdots i_m}\nu|\nabla_{j_1\cdots j_m}\nu\rangle.
\endaligned
\end{eqnarray*}
Noting $(\ref{3.5})$, immediately we have
\begin{eqnarray*}
\aligned
\frac{\partial}{\partial t}|\nabla^m\nu|^2=&-\sum\limits_{l=1}^mg^{i_1j_1}\cdots(g^{i_lp}a_{pq}(X)g^{qj_l})\cdots g^{i_mj_m}\langle\nabla_{i_1\cdots i_m}\nu|\nabla_{j_1\cdots j_m}\nu\rangle\\
&+2g^{i_1j_1}\cdots g^{i_mj_m}\langle\nabla_t\nabla_{i_1\cdots i_m}\nu|\nabla_{j_1\cdots j_m}\nu\rangle.
\endaligned
\end{eqnarray*}
Substituting $(\ref{2.13})$ into the right-hand side of above equation, we obtain
\begin{equation}\label{3.8}
\aligned
\frac{d}{dt}\mathfrak{F}_m(\varphi_t)\Big|_{t=0}
=&-\int_M\langle \mbox{grad}^M(|\nabla^m\nu|^2)|X\rangle\,d\mu_0+\int_M(1+|\nabla^m\nu|^2)H\langle X|\nu\rangle\,d\mu_0\\
&-\sum\limits_{k=1}^m\int_Mg^{i_1j_1}\cdots g^{i_ks}a_{sl}(X)g^{lj_k}\cdots g^{i_mj_m}\langle\nabla_{i_1\cdots i_m}\nu|\nabla_{j_1\cdots j_m}\nu\rangle\,d\mu_0\\
&+2\int_Mg^{i_1j_1}\cdots g^{i_mj_m}\langle\nabla_{i_1\cdots i_m}b(X)|\nabla_{j_1\cdots j_m}\nu\rangle\,d\mu_0\\
&-2\int_Mg^{i_1j_1}\cdots g^{i_mj_m}\langle\mathfrak{p}_{m-2}(\nabla\nu,\nabla a(X))|\nabla_{j_1\cdots j_m}\nu\rangle\,d\mu_0\\
&+2\sum\int_Mg^{i_1j_1}\cdots g^{i_mj_m}\langle\nabla_{j_1\cdots j_m}\nu|\\
&(\nabla^qR^N)(\nabla_{\vec{A}_1}\varphi,\cdots,\nabla_{\vec{A}_q}\varphi)(\nabla_{\vec{B}}
X,\nabla_{\vec{C}}\varphi)\nabla_{\vec{D}}\nu\rangle.
\endaligned
\end{equation}
It is easy to see that the above identity is linear with respect to $X$. By the same argument as in Proposition 3.4 of \cite{M} we know that $\frac{d}{dt}\mathfrak{F}_m(\varphi_t)\Big|_{t=0}$ depends only on $\langle X|\nu\rangle$. This means that, in the computation of $\frac{d}{dt}\mathfrak{F}_m(\varphi_t)\Big|_{t=0}$, we can assume that $X$ is just a normal field, i.e.
\begin{equation}\label{3.9}
X=\langle X|\nu\rangle\nu.
\end{equation}
Hence we can continue the previous computations in this situation and to get from $(\ref{3.4})$, $(\ref{3.5})$ and $(\ref{3.6})$
\begin{equation}\label{3.10}
a_{ij}(X)=2h_{ij}\langle X|\nu\rangle,
\end{equation}
\begin{equation}\label{3.11}
\frac{\partial g^{ij}}{\partial t}=-2g^{is}h_{sl}g^{lj}\langle X|\nu\rangle
\end{equation}
and
\begin{equation}\label{3.12}
b(X)=-\mbox{grad}^M\langle X|\nu\rangle.
\end{equation}

Note that, for the last term on the right-hand side of $(\ref{3.8})$, the unknown vector field $X$ hides in curvature tensor. To get Euler-Lagrange equation, we would like to take it out. From $(\ref{3.9})$ we derive that
\begin{equation}\label{3.13}
\nabla_{\vec{B}}X=\sum\nabla_{\vec{G}}(\langle X|\nu\rangle)\cdot\nabla_{\vec{H}}\nu
\end{equation}
where $$\vec{G}:=(i_{g_1},\cdots,i_{g_{|\vec{G}|}}), \s g_1<g_2<\cdots<g_{|\vec{G}|};$$
$$\vec{H}:=(i_{h_1},\cdots,i_{h_{|\vec{H}|}}), \s h_1<h_2<\cdots<h_{|\vec{H}|};$$
and $(\vec{G},\vec{H})=\sigma(\vec{B})$
where $\sigma$ is a permutation of $\{G, H\}$. $|\vec{G}|$ means the length of $\vec{G}$.

Substituting $(\ref{3.9}),(\ref{3.10}),(\ref{3.11}),(\ref{3.12})$ and $(\ref{3.13})$ into $(\ref{3.8})$, one can obtain
\begin{equation}\label{3.14}
\aligned
&\frac{d}{dt}\mathfrak{F}_m(\varphi_t)\Big|_{t=0}=\int_M(1+|\nabla^m\nu|^2)H\langle X|\nu\rangle\,d\mu_0\\
&-2\sum\limits_{k=1}^m\int_Mg^{i_1j_1}\cdots g^{i_ks}h_{sl}g^{lj_k}\cdots g^{i_mj_m}\langle X|\nu\rangle\langle\nabla_{i_1\cdots i_m}\nu|\nabla_{j_1\cdots j_m}\nu\rangle\,d\mu_0\\
&-2\int_Mg^{i_1j_1}\cdots g^{i_mj_m}\langle\nabla_{i_1\cdots i_m}(grad^M\langle X|\nu\rangle)|\nabla_{j_1\cdots j_m}\nu\rangle\,d\mu_0\\
&-2\int_Mg^{i_1j_1}\cdots g^{i_mj_m}\langle\mathfrak{p}_{m-2}(\nabla\nu,\nabla a(X))|\nabla_{j_1\cdots j_m}\nu\rangle\,d\mu_0\\
&+2\sum\int_Mg^{i_1j_1}\cdots g^{i_mj_m}\langle(\nabla^qR^N)(\nabla_{\vec{A}_1}\varphi,\cdots,\nabla_{\vec{A}_q}\varphi)(\nabla_{\vec{H}}
\nu,\nabla_{\vec{C}}\varphi)\nabla_{\vec{D}}\nu|\\
&\nabla_{j_1\cdots j_m}\nu\rangle\nabla_{\vec{G}}(\langle X|\nu\rangle)\,d\mu_0\\
:=&\int_M(1+|\nabla^m\nu|^2)H\langle X|\nu\rangle\,d\mu_0 -2J_1-2J_2+2\sum J_{\vec{A}_1\cdots\vec{A}_q\vec{H}\vec{C}\vec{D}\vec{G}}\\
&-2\sum\limits_{k=1}^m\int_Mg^{i_1j_1}\cdots g^{i_ks}h_{sl}g^{lj_k}\cdots g^{i_mj_m}\langle X|\nu\rangle\langle\nabla_{i_1\cdots i_m}\nu|\nabla_{j_1\cdots j_m}\nu\rangle\,d\mu_0\\
=&\int_M(1+|\nabla^m\nu|^2)H\langle X|\nu\rangle\,d\mu_0-2m\int_MA\ast\langle\nabla^m\nu|\nabla^m\nu\rangle\langle X|\nu\rangle\,d\mu_0\\
&-2J_1-2J_2+2\sum J_{\vec{A}_1\cdots\vec{A}_q\vec{H}\vec{C}\vec{D}\vec{G}},
\endaligned
\end{equation}
where
\[
\aligned
J_{\vec{A}_1\cdots\vec{A}_q\vec{H}\vec{C}\vec{D}\vec{G}}
:=&\int_Mg^{i_1j_1}\cdots g^{i_mj_m}\nabla_{\vec{G}}(\langle X|\nu\rangle)\\
&\langle(\nabla^qR^N)(\nabla_{\vec{A}_1}\varphi,\cdots,\nabla_{\vec{A}_q}\varphi)(\nabla_{\vec{H}}
\nu,\nabla_{\vec{C}}\varphi)\nabla_{\vec{D}}\nu|\nabla_{j_1\cdots j_m}\nu\rangle\,d\mu_0.
\endaligned\]
However, taking integration by parts we have
\begin{equation}\label{3.15}
\aligned
J_{\vec{A}_1\cdots\vec{A}_q\vec{H}\vec{C}\vec{D}\vec{G}}
=&(-1)^{|\vec{G}|}\sum\int_Mg^{i_1j_1}\cdots g^{i_mj_m}\langle X|\nu\rangle\langle \nabla_{j_1\cdots j_m\vec{K}''}\nu |\\ &(\nabla^{q'}R^N)(\nabla_{\vec{A}'_1}\varphi,\cdots,\nabla_{\vec{A}'_{q'}}\varphi)(\nabla_{\vec{H}'}
\nu,\nabla_{\vec{C}'}\varphi)\nabla_{\vec{D}'}\nu\rangle\,d\mu_0
\endaligned
\end{equation}
where $q'\geqslant q$, $\vec{A}'_r=(\vec{A}_r,\vec{A}''_r)$ for $1\leqslant r\leqslant q$, $\vec{H}'=(\vec{H},\vec{H}'')$, $\vec{C}'=(\vec{C},\vec{C}'')$, $\vec{D}'=(\vec{D},\vec{D}'')$ and, moreover, the following relation holds
\begin{eqnarray*}
(\vec{A}'_{q+1},\cdots,\vec{A}'_{q'},\vec{A}''_1,\cdots,\vec{A}''_q,\vec{H}'',\vec{C}'',\vec{D}'',\vec{K}'')=\sigma(\vec{G})
\end{eqnarray*}
where $\sigma$ is a permutation. Using integration by parts and $(\ref{2.4})$, we have
\begin{eqnarray}\label{3.16}
\aligned
J_1:=&\int_Mg^{i_1j_1}\cdots g^{i_mj_m}\langle\nabla_{i_1\cdots i_m}(grad^M\langle X|\nu\rangle)|\nabla_{j_1\cdots j_m}\nu\rangle\,d\mu_0\\
=&(-1)^m\int_Mg^{i_1j_1}\cdots g^{i_mj_m}\langle grad^M\langle X|\nu\rangle|\nabla_{j_1\cdots j_mi_m\cdots i_1}\nu\rangle\,d\mu_0\\
=&(-1)^m\int_Mdiv^N(\langle X|\nu\rangle g^{i_1j_1}\cdots g^{i_mj_m}\nabla_{j_1\cdots j_mi_m\cdots i_1}\nu)\,d\mu_0\\
&-(-1)^m\int_M\langle X|\nu\rangle div^N(g^{i_1j_1}\cdots g^{i_mj_m}\nabla_{j_1\cdots j_mi_m\cdots i_1}\nu)\,d\mu_0\\
=&(-1)^m\int_M\langle X|\nu\rangle g^{i_1j_1}\cdots g^{i_mj_m}\langle\nabla_{j_1\cdots j_mi_m\cdots i_1}\nu|\nu\rangle H\,d\mu_0\\
&+(-1)^{m+1}\int_M\langle X|\nu\rangle div^N(g^{i_1j_1}\cdots g^{i_mj_m}\nabla_{j_1\cdots j_mi_m\cdots i_1}\nu)\,d\mu_0\\
:=&(-1)^mJ_1^1+(-1)^{m+1}J_1^2.
\endaligned
\end{eqnarray}
Since
\begin{equation}\label{3.17}
0=\nabla_{j_1\cdots j_mi_m\cdots i_1}\langle\nu|\nu\rangle=2\langle\nabla_{j_1\cdots j_mi_m\cdots i_1}\nu|\nu\rangle+\sum\langle\nabla_{\vec{a}}\nu|\nabla_{\vec{b}}\nu\rangle
\end{equation}
where
\begin{equation}\label{3.18}
|\vec{a}|+|\vec{b}|=2m,\s\s |\vec{a}|\geqslant1,\s\s |\vec{b}|\geqslant1,
\end{equation}
substituting $(\ref{3.17})$ into the expression of $J_1^1$ we have
\begin{equation}\label{3.19}
\aligned
J_1=&\frac{(-1)^{m+1}}{2}\sum\limits_{\substack{|\vec{a}|+|\vec{b}|=2m\\|\vec{a}|\geqslant1,|\vec{b}|\geqslant1}}\int_M\langle X|\nu\rangle g^{i_1j_1}\cdots g^{i_mj_m}\langle\nabla_{\vec{a}}\nu|\nabla_{\vec{b}}\nu\rangle H\,d\mu_0+(-1)^{m+1}J_1^2\\
=&\frac{(-1)^{m+1}}{2}\sum\limits_{a=1}^{2m-1}\int_M\langle X|\nu\rangle A\ast\langle\nabla^a\nu|\nabla^{2m-a}\nu\rangle\,d\mu_0+(-1)^{m+1}J_1^2.
\endaligned
\end{equation}
Substituting $(\ref{3.7})$ and $(\ref{3.10})$ into $\mathfrak{p}_{m-2}(\nabla\nu,\nabla a(X))$ and taking integration by parts we have
\begin{equation}\label{3.20}
\aligned
J_2:=&\int_Mg^{i_1j_1}\cdots g^{i_mj_m}\langle\mathfrak{p}_{m-2}(\nabla\nu,\nabla a(X))|\nabla_{j_1\cdots j_m}\nu\rangle\,d\mu_0\\
=&\sum\limits_{k=2}^m(-1)^{m-k+1}\sum\limits_{p=1}^{k-1}\int_Mg^{i_1j_1}\cdots g^{i_mj_m}g^{dl}h_{i_kl}\langle X|\nu\rangle\langle\nabla_{i_1\cdots i_{p-1}di_{p+1}\cdots i_{k-1}i_p}\nu|\\
&\nabla_{j_1\cdots j_mi_m\cdots i_{k+1}}\nu\rangle\,d\mu_0\\
&+\sum\limits_{k=2}^m(-1)^{m-k+1}\sum\limits_{p=1}^{k-1}\int_Mg^{i_1j_1}\cdots g^{i_mj_m}g^{dl}h_{i_kl}\langle X|\nu\rangle\langle\nabla_{i_1\cdots i_{p-1}di_{p+1}\cdots i_{k-1}}\nu|\\
&\nabla_{j_1\cdots j_mi_m\cdots i_{k+1}i_p}\nu\rangle\,d\mu_0\\
&+\sum\limits_{k=2}^m(-1)^{m-k+1}\sum\limits_{p=1}^{k-1}\int_Mg^{i_1j_1}\cdots g^{i_mj_m}g^{dl}h_{i_pl}\langle X|\nu\rangle\langle\nabla_{i_1\cdots i_{p-1}di_{p+1}\cdots i_k}\nu|\\
&\nabla_{j_1\cdots j_mi_m\cdots i_{k+1}}\nu\rangle\,d\mu_0\\
&+\sum\limits_{k=2}^m(-1)^{m-k+1}\sum\limits_{p=1}^{k-1}\int_Mg^{i_1j_1}\cdots g^{i_mj_m}g^{dl}h_{i_pl}\langle X|\nu\rangle\langle\nabla_{i_1\cdots i_{p-1}di_{p+1}\cdots i_{k-1}}\nu|\\
&\nabla_{j_1\cdots j_mi_m\cdots i_k}\nu\rangle\,d\mu_0\\
&-\sum\limits_{k=2}^m(-1)^{m-k+1}\sum\limits_{p=1}^{k-1}\int_Mg^{i_1j_1}\cdots g^{i_mj_m}g^{dl}h_{i_pi_k}\langle X|\nu\rangle\langle\nabla_{i_1\cdots i_{p-1}di_{p+1}\cdots i_{k-1}l}\nu|\\
&\nabla_{j_1\cdots j_mi_m\cdots i_{k+1}}\nu\rangle\,d\mu_0\\
&-\sum\limits_{k=2}^m(-1)^{m-k+1}\sum\limits_{p=1}^{k-1}\int_Mg^{i_1j_1}\cdots g^{i_mj_m}g^{dl}h_{i_pi_k}\langle X|\nu\rangle\langle\nabla_{i_1\cdots i_{p-1}di_{p+1}\cdots i_{k-1}}\nu|\\
&\nabla_{j_1\cdots j_mi_m\cdots i_{k+1}l}\nu\rangle\,d\mu_0\\
=&-3\sum\limits_{k=2}^m(-1)^{m-k}\langle X|\nu\rangle A\ast(\langle\nabla^k\nu|\nabla^{2m-k}\nu\rangle+\langle\nabla^{k-1}\nu|\nabla^{2m-k+1}\nu\rangle).
\endaligned
\end{equation}

For $J_1^2$, we use $(\ref{2.21})$ to transform $div^N(g^{i_1j_1}\cdots g^{i_mj_m}\nabla_{j_1\cdots j_mi_m\cdots i_1}\nu)$ into $div^N(\Delta^m\nu)$ and some remainder terms. Then, using $(\ref{2.22})$ and $(\ref{2.25})$, we can transform $div^N(\Delta^m\nu)$ into $\Delta^m(div^N\nu)=\Delta^mH$ adding some remainder terms again. In conclusion,
\begin{equation}\label{3.21}
\frac{d}{dt}\mathfrak{F}_m(\varphi_t)\Big|_{t=0}=\int_ME_m(\varphi_0)\langle\nu|X\rangle\,d\mu_0
\end{equation}
where
\begin{eqnarray}\label{3.22}
\aligned
E_m(\varphi_0)
:=&(1+|\nabla^m\nu|^2)H-2mA\ast\langle\nabla^m\nu|\nabla^m\nu\rangle+(-1)^m\sum\limits_{a=1}^{2m-1}\langle\nabla^a\nu|\nabla^{2m-a}\nu\rangle\ast A\\
&+6\sum\limits_{k=2}^m(-1)^{m-k}A\ast(\langle\nabla^k\nu|\nabla^{2m-k}\nu\rangle+\langle\nabla^{k-1}\nu|\nabla^{2m-k+1}\nu\rangle)\\
&+2\sum\limits_{\substack{a_1+\cdots+a_q+q\\+h+c+d+k=m-1}}\langle(\nabla^qR^N)(\nabla^{a_1+1}\varphi,\cdots,\nabla^{a_q+1}\varphi)
(\nabla^h\nu,\nabla^{c+1}\varphi)\nabla^d\nu|\nabla^{m+k}\nu\rangle\\
&+2(-1)^m\Big[\Delta^mH+\sum\limits_{i+j+k=2m-2}C_{ijk}\nabla^iA\ast\langle\nabla^{j+1}\nu|\nabla^{k+1}\nu\rangle\\
&+\sum\limits_{\substack{a_1+\cdots+a_q\\+q+b+c+d+e\\=2m-2}}C_{a_1\cdots a_qbcde}\langle\nabla^{e+1}\varphi|\\
&(\nabla^qR^N)(\nabla^{a_1+1}\varphi,\cdots,\nabla^{a_q+1}\varphi)(\nabla^{b+1}\varphi,\nabla^{c+1}\varphi)\nabla^{d+1}\nu\rangle\\
&+\sum\limits_{i+j+k=2m-2}C_{ijk}\langle\nabla^{i+1}\varphi|\nabla^{j+1}\nu\rangle\ast\nabla^k\mathfrak{R}\\
&+\sum\limits_{\substack{a_1+\cdots+a_q+q\\+b+c+d=2m-1}}\langle\nabla\varphi|(\nabla^qR^N)(\nabla^{a_1+1}\varphi,\cdots,
\nabla^{a_q+1}\varphi)(\nabla^{b+1}\varphi,\nabla^{c+1}\varphi)\nabla^d\nu\rangle\\
&+\sum\limits_{i+j=2m-2}C_{ij}\langle\nabla^{i+1}\nu|\nabla\varphi\rangle\ast\nabla^j(\mathfrak{R}+A\otimes A)\Big].
\endaligned
\end{eqnarray}
Here the coefficients are universal constants.

\section{Sobolev inequalities on submanifolds}
Later, we need to establish some interpolation inequalities for some covariant derivative tensors with uniform coefficients with respect to the time variable $t$. In \cite{M} such estimates lie on the well-known Sobolev inequalities on a submanifold established by Michael and Simon in \cite{MS}. In the present situation, Theorem 2.1 of \cite{HS} is instead the foundation to prove interpolation inequalities for tensors of universal coefficients. We will see that, combining it with our Lemma \ref{lemma5.4} indicates why the sectional curvature and injectivity radius appear in our main result. For the sake of convenience and completeness, we would like to write it in the following.

Let $N$ be an $n$-dimensional Riemannian manifold and let $M\rightarrow N$ be an isometric immersion of an $m$-dimensional Riemannian manifold $M$ into $N$. We use the following quantities:

$\bar{K}_{\pi}=$ sectional curvature in $N$,

$\vec{H}=$ mean curvature vector field of the immersion,

$\bar{R}=$ injectivity radius of $N$,

$\omega_m=$ volume of the unit ball of $\mathbb{R}^m$,

$b=$ a positive real number or a pure imaginary one.

\begin{thm}(\cite{HS, HS1})\label{theorem2.1}
Let $M\longrightarrow N$ be an isometric immersion of Riemannian manifolds of dimension $m$ and $n$, respectively.
Assume that $\bar{K}_{\pi}\leqslant b^2\leqslant1$ and $\bar{R}>0$. Then, for a nonnegative $C^1$ function $h$ on $M$ vanishing on $\partial M$ there holds true
\begin{eqnarray}\label{2.26}
\Big(\int_Mh^{\frac{m}{m-1}}\,dM\Big)^{\frac{m-1}{m}}\leqslant c(m)\int_M(|\nabla h|+h|\vec{H}|)\,dM,
\end{eqnarray}
provided
\begin{eqnarray}\label{2.27}
b^2(1-\alpha)^{-2/m}[\omega_m^{-1}\mbox{vol}(\mbox{supp}\,h)]^{2/m}\leqslant1
\end{eqnarray}
and
\begin{eqnarray}\label{2.28}
2\rho_0\leqslant\bar{R},
\end{eqnarray}
where
\begin{eqnarray}\label{2.29}
\rho_0:=\left\{
\begin{array}{ll}
\aligned
&b^{-1}\arcsin(b)(1-\alpha)^{-1/m}[\omega_m^{-1}vol(supp\,h)]^{1/m}\s\s \mbox{for}\s b\,\,real,\\
&(1-\alpha)^{-1/m}[\omega_m^{-1}vol(supp\,h)]^{1/m}\s\s \mbox{for}\s b\,\,imaginary.
\endaligned
\end{array}
\right.
\end{eqnarray}
Here $\alpha$ is a free parameter, $0<\alpha<1$, and
\begin{eqnarray}\label{2.30}
c(m)=c(m,\alpha)=\frac{1}{2}\pi\cdot2^m\alpha^{-1}(1-\alpha)^{-1/m}\frac{m}{m-1}\omega_m^{-1/m}.
\end{eqnarray}
For $b$ imaginary we may omit the factor $\frac{1}{2}\pi$ in the definition of $c(m)$.
\end{thm}

We will also use the following Proposition \ref{lemma5.6}. The idea of its proof stems from Theorem 17.7 of \cite{S} directly.

\begin{pro}\label{lemma5.6}
Let $\varphi:M\longrightarrow N$ be an isometric immersion with $\dim M=n$ and $\dim N=n+1$. Let $\bar{R}$ be $N$'s injectivity radius and $\bar{K}_{\pi}$ be the sectional curvature of $N$. Assume that $\bar{R}$ is positive, $\bar{K}_{\pi}\leqslant b^2\leqslant1$ where $b$ is a positive real number or a pure imaginary one, and the mean curvature vector field of $\varphi$, denoted by $\vec{H}$, satisfies that $||\vec{H}||_{L^p(\mu)}\leqslant\Gamma$ for some $p\in(n,\infty)$. Then, for any $\xi\in N$, there hold that

 (1). when $b$ is a positive real number, for all $0<\sigma\leqslant\rho<\min\{\bar{R},\frac{\pi}{b}\}$,
\begin{equation}\label{5.15}
\Big[\frac{\mu(B_{\sigma}(\xi))}{(\sin b\sigma)^n}\Big]^{\frac{1}{p}}\leqslant\Big[\frac{\mu(B_{\rho}(\xi))}{(\sin b\rho)^n}\Big]^{\frac{1}{p}}
+\frac{\Gamma}{p}\int_{\sigma}^{\rho}(\sin b\tau)^{-\frac{n}{p}}\,d\tau,
\end{equation}

(2). when $b$ is a pure imaginary number, for all $0<\sigma\leqslant\rho<\bar{R}$,
\begin{equation}\label{5.16}
\Big[\frac{\mu(B_{\sigma}(\xi))}{\sigma^n}\Big]^{\frac{1}{p}}\leqslant\Big[\frac{\mu(B_{\rho}(\xi))}{\rho^n}\Big]^{\frac{1}{p}}
+\frac{\Gamma}{p-n}(\rho^{1-\frac{n}{p}}-\sigma^{1-\frac{n}{p}}),
\end{equation}
where
\begin{equation}\label{5.17}
\mu(B_{\sigma}(\xi)):=\int_{\{x\in M|d^N(\varphi(x),\xi)<\sigma\}}1\,dM.
\end{equation}
\end{pro}

\begin{proof} We follow Simon's idea in \cite{S} to prove the proposition. For any $x\in M$, let
\begin{equation}\label{5.18}
r(x):=d^N(\varphi(x),\xi).
\end{equation}
In the following context, we also use $\mbox{grad}^Nr$ to denote $\mbox{grad}^Nd^N$. Take $\gamma\in C^1(\mathbb{R}^1)$ which satisfies $\gamma'(t)\leqslant0$; for $t\leqslant\frac{\rho}{2}$, $\gamma(t)=1$; for $t\geqslant\rho$, $\gamma(t)=0$. Denote $-H\cdot\nu$ by $\vec{H}$.
Since
\begin{equation}\label{5.19}
\int_Mdiv^N(\gamma(r)\cdot r\cdot \mbox{grad}^Nr)\,dM = -\int_M\langle\gamma(r)\cdot r\cdot \mbox{grad}^Nr|\vec{H}\rangle\,dM,
\end{equation}
we have
\begin{equation}\label{5.20}
\aligned
&\int_M\gamma(r)div^N(r\cdot \mbox{grad}^Nr)\,dM + \int_M\gamma'(r)\langle \mbox{grad}^Mr|r\cdot \mbox{grad}^Nr\rangle\,dM\\
=&-\int_M\gamma(r)\cdot r\cdot\langle \mbox{grad}^Nr|\vec{H}\rangle\,dM.
\endaligned
\end{equation}
It is easy to know that, since $\varphi$ is an isometric immersion, the following identity is obvious
\[
\langle\varphi_*(\mbox{grad}^Mr)\Big|\varphi_*(\frac{\partial}{\partial x^i})\rangle=\langle \mbox{grad}^Mr\Big|\frac{\partial}{\partial x^i}\rangle=\langle \mbox{grad}^Nr|\nabla_i\varphi\rangle,
\]
where $\varphi_*$ is the tangent map of $\varphi$. So
\[
P(\varphi)\Big(\mbox{grad}^Nr-\varphi_*(\mbox{grad}^Mr)\Big)=0,
\]
where $P(\varphi(p))$ is the orthogonal projection operator from $T_{\varphi(p)}N$ to $\varphi_*(T_pM)$ with $p\in M$. Therefore, we get
\[\varphi_*(\mbox{grad}^Mr)=P(\varphi)(\mbox{grad}^Nr).\]
It means that
\begin{equation}\label{5.21}
|\mbox{grad}^Mr|\leqslant|\mbox{grad}^Nr|=1.
\end{equation}
Lemma 3.6 of \cite{HS} tells us that
\begin{equation}\label{5.22}
\int_M\gamma(r)\cdot n\cdot b\cdot r\cdot \cot(br)\,dM\leqslant-\int_M\gamma'(r)r\,dM-\int_M\gamma(r)\cdot r\langle \mbox{grad}^Nr|\vec{H}\rangle\,dM.
\end{equation}
Next, we need to consider the following two cases.

\textbf{Case 1:} $b$ is a positive real number. In this situation, for $r\leqslant\rho$, we have
\begin{equation}\label{5.23}
n\cdot b\cdot r\cdot \cot(br)\geqslant n\cdot b\cdot \rho\cdot \cot(b\rho),
\end{equation}

\textbf{Case 2:} $b$ is a pure imaginary number. For this case we have
\begin{equation}\label{5.24}
n\cdot b\cdot r\cdot \cot(br):=n|b|r\cdot \coth(|b|r)\geqslant n.
\end{equation}

For Case 1, substituting $(\ref{5.23})$ into $(\ref{5.22})$ we get
\begin{equation}\label{5.25}
n\cdot b\cdot \rho\cdot \cot(b\rho)\int_M\gamma(r)\,dM\leqslant-\int_M\gamma'(r)r\,dM-\int_M\gamma(r)\cdot r\langle \mbox{grad}^Nr|\vec{H}\rangle\,dM.
\end{equation}
Now, choose $\phi\in C^1(\mathbb{R}^1)$ which satisfies that $\phi(t)=0$ for $t\geqslant1$, $\phi(t)=1$ for $t\leqslant\frac{1}{2}$, and $\phi'(t)\leqslant0$ for all $t$.

Since $\gamma$ is arbitrary, we set
\begin{equation}\label{5.26}
\gamma(r):=\phi\Big(\frac{r}{\rho}\Big),
\end{equation}
\begin{equation}\label{5.27}
I(\rho):=\int_M\phi\Big(\frac{r}{\rho}\Big)\,dM,
\end{equation}
and
\begin{equation}\label{5.28}
L(\rho):=\int_M\phi\Big(\frac{r}{\rho}\Big)\langle r\cdot\mbox{grad}^Nr|\vec{H}\rangle\,dM.
\end{equation}
Note that
\begin{equation}\label{5.29}
\gamma'(r)\cdot r=-\rho\cdot\frac{\partial}{\partial\rho}\Big[\phi\Big(\frac{r}{\rho}\Big)\Big].
\end{equation}
Substituting $(\ref{5.26}),(\ref{5.27}),(\ref{5.28})$ and $(\ref{5.29})$ into $(\ref{5.25})$ gives rise to
\begin{equation}\label{5.30}
n\cdot b\cdot\rho\cdot \cot(b\rho)I(\rho)-\rho\cdot I'(\rho)\leqslant-L(\rho),
\end{equation}
which is equivalent to
\begin{equation}\label{5.31}
\frac{d}{d\rho}\Big[\frac{I(\rho)}{(\sin b\rho)^n}\Big]\geqslant\frac{L(\rho)}{(\sin b\rho)^n\rho}.
\end{equation}
Hence
\begin{equation}\label{5.32}
\aligned
\frac{d}{d\rho}\Big[\frac{I(\rho)}{(\sin b\rho)^n}\Big]
\geqslant&-\frac{\int_M\phi(r/\rho)\cdot\rho\cdot|\vec{H}|\,dM}{(\sin b\rho)^n\rho}\\
\geqslant&-\frac{(\int_M|\vec{H}|^p\,dM)^{\frac{1}{p}}(\int_M\phi(r/\rho)^{\frac{p}{p-1}}\,dM)^{\frac{p-1}{p}}}{(\sin b\rho)^n}.
\endaligned
\end{equation}
Since $0\leqslant\phi\leqslant1$, recalling the definition of $I(\rho)$ yields
\begin{equation}\label{5.33}
\frac{d}{d\rho}\Big[\frac{I(\rho)}{(\sin b\rho)^n}\Big]\geqslant-\Gamma\frac{I(\rho)^{\frac{p-1}{p}}}{(\sin b\rho)^n}.
\end{equation}
Furthermore
\begin{equation}\label{5.34}
\frac{d}{d\rho}\Big\{\Big[\frac{I(\rho)}{(\sin b\rho)^n}\Big]^{\frac{1}{p}}\Big\}\geqslant-\frac{\Gamma}{p}(\sin b\rho)^{-\frac{n}{p}}.
\end{equation}
Integrating on $[\sigma,\rho]$, we have
\begin{equation}\label{5.35}
\Big[\frac{I(\rho)}{(\sin b\rho)^n}\Big]^{\frac{1}{p}}-\Big[\frac{I(\sigma)}{(\sin b\sigma)^n}\Big]^{\frac{1}{p}}\geqslant-\frac{\Gamma}{p}\int_{\sigma}^{\rho}(\sin b\tau)^{-\frac{n}{p}}\,d\tau.
\end{equation}
Letting $\phi$ increasingly tends to the characteristic function of $(-\infty,1)$, we get $(\ref{5.15})$.
\medskip

For Case 2, substituting $(\ref{5.24})$ into $(\ref{5.22})$ gives
\begin{equation}\label{5.36}
n\int_M\gamma(r)\,dM\leqslant-\int_M\gamma'(r)r\,dM-\int_M\gamma(r)\cdot r\langle \mbox{grad}^Nr|\vec{H}\rangle\,dM.
\end{equation}
Let $\gamma$, $I$ and $L$ be the same as $(\ref{5.26})$, $(\ref{5.27})$ and $(\ref{5.28})$. Using similar method as we infer $(\ref{5.31})$, one can obtain
\begin{equation}\label{5.37}
\frac{d}{d\rho}\Big[\frac{I(\rho)}{\rho^n}\Big]\geqslant\frac{L(\rho)}{\rho^{n+1}}.
\end{equation}
Then we get $(\ref{5.16})$ by using H\"{o}lder inequality, integrating on $[\sigma,\rho]$ and letting $\phi$ increasingly tends to the characteristic function of $(-\infty,1)$. This completes the proof.
\end{proof}

Following directly the idea in Section 18 of \cite{S} we also have the following proposition which will be used in the sequel.
\begin{pro}\label{lemma5.7}
Let $\varphi:M\longrightarrow N$ be an isometric immersion with $\dim M=n$ and $\dim N=n+1$. Let $\bar{R}$ be $N$'s injectivity radius, $\bar{K}_{\pi}$ be the sectional curvature of $N$
and $\vec{H}$ be the mean curvature vector field of $\varphi$. Assume that $\bar{R}$ is positive, $\bar{K}_{\pi}\leqslant b^2\leqslant1$ where $b$ is a positive real number or a pure imaginary one. Then, for $h\in C^1(M)$ and $h\geqslant0$, there hold that

(1). when $b$ is real, for all $0<\sigma\leqslant\rho<\min\{\bar{R},\frac{\pi}{b}\}$, we have
\begin{equation}\label{5.38}
\frac{\int_{B_{\sigma}(\xi)}h\,dM}{(\sin b\sigma)^n}\leqslant\frac{\int_{B_{\rho}(\xi)}h\,dM}{(\sin b\rho)^n}+\int_{\sigma}^{\rho}d\tau\cdot\tau^{-1}(\sin b\tau)^{-n}\int_{B_{\tau}(\xi)}r(|\nabla h|+h|\vec{H}|)\,dM;
\end{equation}

(2). when $b$ is imaginary, for all $0<\sigma\leqslant\rho<\bar{R}$, we have
\begin{equation}\label{5.39}
\frac{\int_{B_{\sigma}(\xi)}h\,dM}{\sigma^n}\leqslant\frac{\int_{B_{\rho}(\xi)}h\,dM}{\rho^n}+\int_{\sigma}^{\rho}d\tau\cdot\tau^{-n-1}\int_{B_{\tau}(\xi)}r(|\nabla h|+h|\vec{H}|)\,dM,
\end{equation}
where $\xi\in N$ and $B_{\rho}(\xi):=\{x\in M|d^N(\varphi(x),\xi)<\rho\}$.
\end{pro}

\begin{proof}
Let $r$, $\gamma$, $\phi$ be the same as those in proof of Proposition \ref{lemma5.6}. Since
\begin{equation}
\int_Mdiv^N(h\cdot\gamma(r)\cdot r\cdot \mbox{grad}^Nr)\,dM=-\int_M\langle h\cdot\gamma(r)\cdot r\cdot \mbox{grad}^Nr|\vec{H}\rangle\,dM,
\end{equation}
simple calculation gives
\begin{equation}
\aligned
\int_M h\cdot\gamma(r)div^N(r\cdot \mbox{grad}^Nr)\,dM
=&-\int_M\langle \mbox{grad}^Mh\cdot\gamma(r)|r\cdot \mbox{grad}^Nr\rangle\,dM\\
&-\int_M h\cdot\gamma(r)\cdot r\cdot\langle \mbox{grad}^Nr|\vec{H}\rangle\,dM\\
&-\int_M\langle h\cdot\gamma'(r)\cdot \mbox{grad}^Mr|r\cdot \mbox{grad}^Nr\rangle\,dM.
\endaligned
\end{equation}
Because of Lemma 3.6 of \cite{HS} and $(\ref{5.21})$, readers can see
\begin{equation}\label{5.42'}
\aligned
&\int_M\gamma(r)\cdot n\cdot b\cdot r\cdot \cot(br)h\,dM\\
\leqslant&-\int_Mh\gamma'(r)r\,dM+\int_Mh\gamma(r)\cdot r|\vec{H}|\,dM+\int_M|\nabla h|\gamma(r)r\,dM.
\endaligned
\end{equation}
When $b$ is real, for $\rho<\min\{\frac{\pi}{b},\bar{R}\}$, from $(\ref{5.23})$ we have
\begin{equation}
\aligned
&n\cdot b\cdot \rho\cdot \cot(b\rho)\int_M\gamma(r)h\,dM\\
\leqslant&-\int_Mh\gamma'(r)r\,dM+\int_M\gamma(r)r(|\nabla h|+h|\vec{H}|)\,dM.
\endaligned
\end{equation}
Let
\begin{equation}
\tilde{I}(\rho):=\int_M\phi(r/\rho)h\,dM
\end{equation}
and
\begin{equation}
\tilde{L}(\rho):=\int_M\phi(r/\rho)r(|\nabla h|+h|\vec{H}|)\,dM.
\end{equation}
Because of $(\ref{5.29})$, we get
\begin{equation}
n\cdot b\cdot\rho\cdot \cot(b\rho)\tilde{I}(\rho)-\rho\cdot \tilde{I}'(\rho)\leqslant\tilde{L}(\rho)
\end{equation}
which is equivalent to
\begin{equation}
-\frac{d}{d\rho}\Big[\frac{\tilde{I}(\rho)}{(\sin b\rho)^n}\Big]\leqslant\frac{\tilde{L}(\rho)}{(\sin b\rho)^n\rho}.
\end{equation}
Integrating on $[\sigma,\rho]$, one can easily know
\begin{equation}
\frac{\tilde{I}(\sigma)}{(\sin b\sigma)^n}-\frac{\tilde{I}(\rho)}{(\sin b\rho)^n}\leqslant\int_{\sigma}^{\rho}\tau^{-1}(\sin b\tau)^{-n}\tilde{L}(\tau)\,d\tau.
\end{equation}
Letting $\phi$ increasingly tends to the characteristic function of $(-\infty,1]$, we obtain $(\ref{5.38})$.

When $b$ is imaginary, substituting $(\ref{5.24})$ into $(\ref{5.42'})$ and using a similar argument with the above, we derive $(\ref{5.39})$. This completes the proof.
\end{proof}

\begin{rem}
Proposition \ref{lemma5.6} and Proposition \ref{lemma5.7} do not require that $\varphi:M\longrightarrow N$ is an embedding. This is the difference between our proof of Lemma \ref{lemma5.8} and that of Proposition 6.2 in \cite{M}. Thanks to the difference, we do not have to construct an embedding as Mantegazza shows Proposition 6.2 in \cite{M}, when we prove Lemma \ref{lemma5.8}.
\end{rem}

We also need to recall the following Gagliardo-Nirenberg type inequality which established in \cite{H} (also see \cite{C, M}).
\begin{pro}\label{lemma5.1}
Assume that $(\tilde{M},g)$ is smooth, compact, without boundary, $\dim\tilde{M}=n$. Then there exists a constant $C=C(s,n)$ such that for any smooth section $T$ of a vector bundle $(E,\pi,\tilde{M})$ and for all $1\leqslant j\leqslant s$, there holds
\[
\int_{\tilde{M}}|\nabla^jT|^2\,d\tilde{M}\leqslant C\left(\int_{\tilde{M}}|\nabla^sT|^2\,d\tilde{M}\right)^{\frac{j}{s}}\left(\int_{\tilde{M}}|T|^2\,d\tilde{M}\right)^{1-\frac{j}{s}}.
\]
Here the constant $C$ depends not on the metric or the geometry of $\tilde{M}$.
\end{pro}
For the proof of this proposition we refer to Corollary 12.7 of \cite{H} and The following proposition has also been proved in Corollary 12.6 of \cite{H}.

\begin{pro}\label{lemma5.2}
Assume that $(\tilde{M},g)$ is smooth, compact, without boundary, $\dim\tilde{M}=n$. Then there exists a constant $C=C(s,n)$ such that for any smooth section $T$ of a vector bundle $(E,\pi,\tilde{M})$ and for all $1\leqslant j\leqslant s-1$, there holds
\[
\int_{\tilde{M}}|\nabla^jT|^{\frac{2s}{j}}\,d\tilde{M}\leqslant C||T||_{L^{\infty}(\tilde{M})}^{2(\frac{s}{j}-1)}\int_{\tilde{M}}|\nabla^sT|^2\,d\tilde{M}.
\]
\end{pro}


\section{Small time existence and Uniform A Priori Estimates}\label{section4}
Suppose that $\varphi_0:M\longrightarrow N$ is a smooth immersion. We want to look for a smooth map $\varphi:M\times[0,T)\longrightarrow N$ for some $T>0$ such that
\begin{equation}\label{4.1}
\left\{
\begin{array}{llll}
\aligned
&\frac{\partial\varphi}{\partial t}(p,t)=-E_m(\varphi_t)(p)\nu(p,t),\\
&\varphi(\cdot,0)=\varphi_0
\endaligned
\end{array}
\right.
\end{equation}
and for every $t\in[0,T)$, $\varphi_t$ is an immersion. If such a solution exists, we say that the submanifold $M_t:=(M,g_t)$, where $g_t$ is obtained by pulling back the metric $\langle\cdot|\cdot\rangle$ of $N$ via $\varphi_t$, evolves by the negative gradient flow of the functional $\mathfrak{F}_m$.

The following theorem is due to \cite{HuP}. For the details we also refer to Theorem 2.5.2 of Section 2 in \cite{Po2}.
\begin{thm}\label{theorem 4.1}
For any smooth submanifold immersion $\varphi_0:M^n\longrightarrow N^{n+1}$, there exists a unique solution to the flow problem
\[
\frac{\partial\varphi}{\partial t}=[(-1)^{s+1}\Delta^sH+\Phi(\varphi,\nabla\varphi,\nu,A,\nabla A,\cdots,\nabla^{2s-1}A)]\nu
\]
defined on some interval $0\leqslant t<T$ and taking $\varphi_0$ as its initial value.
\end{thm}

From $(\ref{2.8})$ and $(\ref{2.9})$, we have
\begin{equation}\label{4.2}
\nabla\nu=A\ast\nabla\varphi
\end{equation}
and
\begin{equation}\label{4.3}
\nabla^2\varphi=-A\otimes\nu.
\end{equation}
With an argument of induction, one can obtain following important relations: for $s\geqslant2$
\begin{equation}\label{4.4}
\nabla^s\nu=\mathfrak{q}^s(A)\ast\nabla\varphi+\mathfrak{q}^s(A)\otimes\nu,
\end{equation}
and for $s\geqslant3$
\begin{equation}\label{4.5}
\nabla^s\varphi=\mathfrak{q}^{s-1}(A)\ast\nabla\varphi+\mathfrak{q}^{s-1}(A)\otimes\nu.
\end{equation}

By the remark below Theorem $4.1$ in page 153 of \cite{M}, we can let $\Phi$ depend on the metric $g$. So $\Phi$ also depends on $g^{-1}$. Hence there is a unique local smooth solution to the following flow:
\begin{equation}\label{4.6}
\left\{
\begin{array}{llll}
\frac{\partial\varphi}{\partial t}=[(-1)^{m+1}\Delta^mH+ \Phi(\varphi, g^{-1}, \nu, A)]\nu,\\
\varphi(\cdot,0)=\varphi_0:M^n\longrightarrow N^{n+1}.
\end{array}
\right.
\end{equation}
Here
$$\Phi(\varphi, g^{-1}, \nu, A)\equiv\Phi(\varphi,\nabla\varphi,\cdots,\nabla^{2m+1}\varphi,g^{-1},\nu,\cdots,\nabla^{2m}\nu,A,\nabla A,\cdots,\nabla^{2m-1}A).$$
Therefore $(\ref{4.1})$ has a unique local smooth solution and when $t$ is sufficiently small, $\varphi(t)$ is an immersion.

Let $\varphi_t\equiv \varphi(t)$ be the flow defined in Section \ref{section4} and $g_t$ be the metric on $M$ obtained by pulling back $\langle\cdot|\cdot\rangle$ via $\varphi_t$. $(M,g_t)$ is denoted by $M_t$. As described in the introduction, in order to control $||\nabla^pA||_{L^{\infty}(M_t)}$ by $||A||_{W^{k,2}(M_t)}$ we need to establish some Sobolev type interpolation inequalities. Because $M_t$ is time-varying, we will require that the constants in these inequalities are universal.

By the very definition of the flow
\begin{equation}\label{5.1}
\frac{d}{dt}\mathfrak{F}_m(\varphi_t)=-\int_M[E_m(\varphi_t)]^2\,d\mu_t\leqslant0.
\end{equation}
Hence, as long as the flow remains smooth, we have the uniform estimate
\begin{equation}\label{5.2}
\int_M(1+|\nabla^m\nu|^2)\,d\mu_t\leqslant\mathfrak{F}_m(\varphi_0)
\end{equation}
for every $t\geqslant0$.

Throughout the later proof, the upper bound of $||H||_{n+1}$ is crucial. It is a precondition of many theorems. So we are going to estimate it.

In Proposition \ref{lemma5.2}, taking $(\tilde{M},g)=M_t$, $T=\nu$, $j=1$ and $s=m$, since $|\nu|=1$, we have $||\nu||_{L^{\infty}(M_t)}=1$ and
\begin{equation}\label{5.3}
||\nabla\nu||_{2m}\lesssim||\nabla^m\nu||_2^{\frac{1}{m}}.
\end{equation}
By a simple calculation, it is easy to know
\begin{equation}\label{5.4}
|\nabla\nu|=|A|.
\end{equation}
So
\begin{equation}\label{5.5}
||A||_{2m}\lesssim||\nabla^m\nu||_2^{\frac{1}{m}}\leqslant\mathfrak{F}_m(\varphi_t)^{\frac{1}{2m}}\leqslant\mathfrak{F}_m(\varphi_0)^{\frac{1}{2m}}.
\end{equation}
Hence by H\"{o}lder inequality,
\begin{equation}\label{5.6}
||A||_{n+1}\leqslant||A||_{2m}\cdot(\mbox{vol}(M_t))^{\frac{2m-n-1}{2m(n+1)}}\lesssim\mathfrak{F}_m(\varphi_0)^{\frac{1}{n+1}}
\end{equation}
Choosing an orthogonal frames, by definition, we get
\[|A|^2=\sum\limits_{i=1}^n\sum\limits_{k=1}^nh_{ik}^2\]
and
\[|H|^2=(\sum\limits_{i=1}^nh_{ii})^2\leqslant n\sum\limits_{i=1}^nh_{ii}^2\leqslant n|A|^2.\]
So, it follows that
\[|H|\leqslant\sqrt{n}|A|\]
and
\[||H||_{n+1}\lesssim||A||_{n+1}\lesssim(\mathfrak{F}_m(\varphi_0))^{\frac{1}{n+1}}.\]

\section{Some interpolation inequalities independent of time}\label{section5}

In the sequel our arguments are based on the following universal interpolation inequalities for covariant tensors on $\varphi_t$. To establish uniform estimates for a section $T$ with respect to time $t$, first one need to establish some inequalities for a function $u$. For this purpose we establish the following Lemma \ref{lemma5.4}, Remark \ref{remark5.5} and Lemma \ref{lemma5.8}.

\begin{lem}\label{lemma5.4} Suppose that $\varphi_t$ and $M_t$ are as defined at the beginning of section \ref{section4}. Let
$\bar{K}_{\pi}$ denotes the sectional curvature of $N$, $\bar{R}$ the injectivity radius of $N$, $\vec{H}_t$ the mean curvature vector field of the immersion $\varphi_t$, $\omega_n$ the volume of the unit ball of $\mathbb{R}^n$ and $b$ is a positive real number or a pure imaginary one. Assume that $\bar{R}>0$, $\bar{K}_{\pi}\leqslant b^2\leqslant1$, and
\begin{equation}\label{5.8}
\left\{
\begin{array}{llll}
\aligned
&\mathfrak{F}_m(\varphi_0)\leqslant\min\left\{\frac{\omega_n}{|b|^n(n+1)},\,\Big(\frac{b\bar{R}}{\pi}\Big)^n
\frac{\omega_n}{n+1}\right\}\s\s \mbox{for}\s b \s \mbox{real},\\
&\mathfrak{F}_m(\varphi_0)\leqslant\min\left\{\frac{\omega_n}{|b|^n(n+1)},\, \frac{\bar{R}^n}{(n+1)2^n}\omega_n\right\}
\s\s \mbox{for}\s b\s \mbox{imaginary}.
\endaligned
\end{array}
\right.
\end{equation}
Then, for $h\in C^1(M)$ with $h\geqslant0$, we have
\begin{equation}\label{5.9}
\Big(\int_Mh^{\frac{n}{n-1}}\,d\mu_t\Big)^{\frac{n-1}{n}}\leqslant C(n)\int_M\Big[|\nabla h|+h|\vec{H}_t|\Big]\,d\mu_t
\end{equation}
where
\begin{equation}\label{5.10}
C(n)=\frac{\pi}{2}\cdot2^n\frac{n+1}{n-1}(n+1)^{\frac{1}{n}}\omega_n^{-\frac{1}{n}}.
\end{equation}
In the case $b$ is an imaginary number we may omit the factor $\frac{\pi}{2}$ in the definition of $C(n)$.
\end{lem}

\begin{proof} Obviously, we know that $\mbox{vol}(M_t)\leqslant\mathfrak{F}(\varphi_t)\leqslant\mathfrak{F}(\varphi_0)$. In Theorem \ref{theorem2.1}, we take $m=n$, $\tilde{n}=n+1$, $\vec{H}=\vec{H}_t$ and $\alpha=\frac{n}{n+1}$(the optimal choice of $\alpha$ to minimize $c(n,\alpha)$ is $\alpha=\frac{n}{n+1}$).

Clearly,
\begin{eqnarray*}
\mbox{vol}(\mbox{supp}\,h)\leqslant \mbox{vol}(M_t)\leqslant\mathfrak{F}(\varphi_0)
\end{eqnarray*}
and $b^2\leqslant|b|^2$. Noting
\begin{equation}\label{5.7}
\mathfrak{F}_m(\varphi_0)\leqslant\frac{\omega_n}{|b|^n(n+1)},
\end{equation}
by a simple computation one can see that $(\ref{5.7})$ implies the condition $(\ref{2.27})$ holds true in Theorem \ref{theorem2.1}. Since $\arcsin(b)\leqslant\frac{\pi}{2}$ for $b$ real, $(\ref{5.8})$ yields $(\ref{2.28})$. This completes the proof.\end{proof}

\begin{rem}\label{remark5.5}
For any $p\in[1,n)$, replacing $h$ by $h^{\frac{np-p}{n-p}}$ in Lemma \ref{lemma5.4} and using H$\ddot{o}$lder inequality, we obtain
\begin{equation}\label{5.11}
||h||_{L^{\frac{np}{n-p}}(\mu_t)}\leqslant C(n,p)(||\nabla h||_{L^p(\mu_t)}+||h\cdot\vec{H}_t||_{L^p(\mu_t)}).
\end{equation}
For any function $u\in C^1(M)$, by taking $h=\sqrt{u^2+\varepsilon^2}$, we can see
\begin{equation}\label{5.12}
|\nabla h|=\Big|\frac{u\cdot\nabla u}{\sqrt{u^2+\varepsilon^2}}\Big|\leqslant|\nabla u|.
\end{equation}
Substituting $(\ref{5.12})$ into $(\ref{5.11})$, we have
\begin{equation}\label{5.13}
||\sqrt{u^2+\varepsilon^2}||_{L^{\frac{np}{n-p}}(\mu_t)}\leqslant C(n,p)(||\nabla u||_{L^p(\mu_t)}+||\sqrt{u^2+\varepsilon^2}\vec{H}_t||_{L^p(\mu_t)}).
\end{equation}
Letting $\varepsilon$ tend to 0, one can easily get
\begin{equation}\label{5.14}
||u||_{L^{\frac{np}{n-p}}(\mu_t)}\leqslant C(n,p)(||\nabla u||_{L^p(\mu_t)}+||u\cdot\vec{H}_t||_{L^p(\mu_t)}).
\end{equation}
\end{rem}

\begin{rem}
In the next, we are going to give a lower bound of $\mbox{vol}(M_t)$. Indeed, taking $h\equiv1$ in (\ref{5.9}) and applying the same trick below Proposition 5.2 of \cite{M}, we will get
\begin{eqnarray*}
\mathfrak{F}(\varphi_t)^{-n}\lesssim vol(M_t)\leqslant\mathfrak{F}(\varphi_t).
\end{eqnarray*}
On the other hand, from the inequality $\mathfrak{F}(\varphi_t)\leqslant\mathfrak{F}(\varphi_0)$ the lower bound of $\mbox{vol}(M_t)$ follows.
\end{rem}

Now let us return to prove the following universal inequality.
\begin{lem}\label{lemma5.8}
Let $\varphi$ be the solution to $(\ref{4.1})$ and for each $t$ which is in the maximal existence interval of $\varphi$, $\varphi(t)$ is an immersion from $M$ into $N$. Suppose $\dim M=n$ and $\dim N=n+1$. Let $\bar{R}$ be $N$'s injectivity radius and $\bar{K}_{\pi}$ be the sectional curvature of $N$. Assume that $\bar{R}$ is positive, $\bar{K}_{\pi}\leqslant b^2\leqslant1$ where $b$ is a positive real number or a pure imaginary one. Then, for any $p\in(n,\infty)$ there exists a constant $C$ which depends only on $n$, $p$, $\mathfrak{F}_m(\varphi_0)$ and $\bar{R}$ (maybe also depend upon $|b|$) such that, for any $u\in C^1(M)$, there holds true
\begin{equation}\label{5.49}
\max\limits_M|u|\leqslant C(||\nabla u||_p+||u||_p).
\end{equation}
\end{lem}

\begin{proof} We have known that
\begin{equation}\label{5.50}
\mbox{vol}(M_t)\leqslant\mathfrak{F}_m(\varphi_0)
\end{equation}
and
\begin{equation}\label{5.51}
||H||_{L^{n+1}(M_t)}\lesssim(\mathfrak{F}_m(\varphi_0))^{\frac{1}{n+1}}.
\end{equation}
Firstly, we consider the case that $p\in(n,n+1]$ and $u\geqslant0$.

For any $\xi\in M$, recall that
\[B_{\sigma}(\varphi_t(\xi))=\{x\in M|d^N(\varphi_t(x),\varphi_t(\xi))<\sigma\},\]
and let
\[S_{\sigma}(\xi):=\{x\in M|d^{M_t}(x,\xi)<\sigma\}.\]

When $b$ is real, taking $\rho=\rho_0:=\frac{1}{2}\min\{\bar{R},\frac{\pi}{b}\}$ in Proposition \ref{lemma5.6} we have
\begin{equation}\label{5.52}
\Big[\frac{\mu_t(B_{\sigma}(\varphi_t(\xi)))}{(\sin b\sigma)^n}\Big]^{\frac{1}{p}}\lesssim\Big[\frac{\mu_t(B_{\rho_0}(\varphi_t(\xi)))}{(\sin b\rho_0)^n}\Big]^{\frac{1}{p}}+\int_0^{\rho_0}(\sin b\tau)^{-\frac{n}{p}}\,d\tau.
\end{equation}
Hence
\begin{equation}\label{5.53}
\mu_t(B_{\sigma}(\varphi_t(\xi)))\leqslant C(\bar{R},b,n,p,\mathfrak{F}_m(\varphi_0))\cdot(\sin b\sigma)^n.
\end{equation}
In the next, taking $\rho=\rho_0:=\frac{1}{2}\min\{\bar{R},\frac{\pi}{b}\}$ in Proposition \ref{lemma5.7} we have
\begin{equation}\label{5.54}
\aligned
&\frac{\int_{B_{\sigma}(\varphi_t(\xi))}u\,d\mu_t}{(\sin b\sigma)^n}
\leqslant\frac{\int_Mu\,d\mu_t}{(\sin b\rho_0)^n}+\int_{\sigma}^{\rho_0}d\tau(\sin b\tau)^{-n}\int_{B_{\tau}(\varphi_t(\xi))}(|\nabla u|+u|H|)\,d\mu_t\\
\leqslant&\int_{\sigma}^{\rho_0}d\tau(\sin b\tau)^{-n}(||\nabla u||_{L^p(M_t)}+||uH||_{L^p(M_t)})\cdot\mu_t(B_{\tau}(\varphi_t(\xi)))^{1-\frac{1}{p}}+\frac{\int_Mu\,d\mu_t}{(\sin b\rho_0)^n}\\
\lesssim&\frac{\int_Mu\,d\mu_t}{(\sin b\rho_0)^n}+\int_{\sigma}^{\rho_0}d\tau(\sin b\tau)^{-\frac{n}{p}}(||\nabla u||_{L^p(M_t)}+||uH||_{L^p(M_t)}).
\endaligned
\end{equation}
Since
\begin{equation}\label{5.55}
S_{\sigma}(\xi)\subseteq B_{\sigma}(\varphi_t(\xi)),
\end{equation}
we have
\begin{equation}\label{5.56}
\frac{\int_{B_{\sigma}(\varphi_t(\xi))}u\,d\mu_t}{(\sin b\sigma)^n}\geqslant\frac{\int_{S_{\sigma}(\xi)}u\,d\mu_t}{(\sin b\sigma)^n}.
\end{equation}
As $\sigma\rightarrow0$,
\begin{equation}\label{5.57}
\frac{\int_{S_{\sigma}(\xi)}u\,d\mu_t}{(\sin b\sigma)^n}\longrightarrow u(\xi)\frac{\omega_n}{b^n}.
\end{equation}
The proof of the above can be found in Exercise 1.117 in page 59 of \cite{CLN}. So,
\begin{equation}\label{5.58}
u(\xi)\frac{\omega_n}{b^n}\lesssim\int_Mu\,d\mu_t+||\nabla u||_{L^p(M_t)}+||uH||_{L^p(M_t)}.
\end{equation}

When $b$ is imaginary, taking $\rho=\rho_0:=\frac{1}{2}\bar{R}$ in Proposition \ref{lemma5.6} and Proposition \ref{lemma5.7} we obtain
\begin{equation}\label{5.59}
\aligned
&\frac{\int_{B_{\sigma}(\varphi_t(\xi))}u\,d\mu_t}{\sigma^n}\leqslant\frac{\int_Mu\,d\mu_t}{\rho_0^n}+\int_{\sigma}^{\rho_0}d\tau\cdot\tau^{-n}\int_{B_{\tau}(\varphi_t(\xi))}(|\nabla u|+u|H|)\,d\mu_t\\
\leqslant&\int_{\sigma}^{\rho_0}d\tau\cdot\tau^{-n}(||\nabla u||_{L^p(M_t)}+||uH||_{L^p(M_t)})\cdot\mu_t(B_{\tau}(\varphi_t(\xi)))^{1-\frac{1}{p}}+\frac{\int_Mu\,d\mu_t}{\rho_0^n}\\
\lesssim&\frac{\int_Mu\,d\mu_t}{\rho_0^n}+\int_{\sigma}^{\rho_0}d\tau\cdot\tau^{-\frac{n}{p}}(||\nabla u||_{L^p(M_t)}+||uH||_{L^p(M_t)}).
\endaligned
\end{equation}
Letting $\sigma\rightarrow0$, one can get
\begin{equation}\label{5.60}
u(\xi)\omega_n\lesssim\int_Mu\,d\mu_t+||\nabla u||_{L^p(M_t)}+||uH||_{L^p(M_t)}.
\end{equation}
In conclusion
\begin{equation}\label{5.61}
\max u\lesssim\int_Mu\,d\mu_t+||\nabla u||_{L^p(M_t)}+||uH||_{L^p(M_t)}.
\end{equation}

For arbitrary $u$, replacing $u$ in $(\ref{5.61})$ by $u^2$ we have
\begin{equation}\label{5.62}
\aligned
(\max|u|)^2\lesssim&\int_M|u|^2\,d\mu_t+||\nabla u\cdot u||_{L^p(M_t)}+||u^2H||_{L^p(M_t)}\\
\leqslant&\max|u|\int_M|u|\,d\mu_t+||\nabla u||_{L^p(M_t)}\max|u|+||uH||_{L^p(M_t)}\max|u|.
\endaligned
\end{equation}
So
\begin{equation}\label{5.63}
\max|u|\lesssim\int_Mu\,d\mu_t+||\nabla u||_{L^p(M_t)}+||uH||_{L^p(M_t)}.
\end{equation}
Furthermore, we can take almost the same arguments as in Proposition 6.2 of \cite{M} to get the required result. This completes the proof.
\end{proof}

Now we are going to extend our formulas $(\ref{5.14})$ and $(\ref{5.49})$ to the case of covariant tensor $T$. The method which we use is almost the same as those employed in the proofs of Proposition 6.3, Corollary 6.4 and Proposition 6.5 in \cite{M}. So we only list the results and omit their proofs.

\begin{lem}\label{lemma5.9}
Let $\varphi$ be the solution to $(\ref{4.1})$ and $\varphi(t)$ be an immersion for every $t$. Assume $\bar{R}>0$ where $\bar{R}$ is the $N$'s injectivity radius, the sectional curvature of $N$ satisfies that $\bar{K}_{\pi}\leqslant b^2$ where $b$ is a positive real number belonging to $(0, 1]$ or a pure imaginary one. Suppose the initial immersion map $\varphi_0$ satisfies respectively
\begin{equation}\label{5.65}
\left\{
\begin{array}{llll}
\aligned
&\mathfrak{F}_m(\varphi_0)\leqslant\min\left\{\frac{\omega_n}{|b|^n(n+1)},\, \Big(\frac{b\bar{R}}{\pi}\Big)^n\frac{\omega_n}{n+1}\right\},\,\,\,\,for\,\,\,\,b\,\,\,\,real;\\
&\mathfrak{F}_m(\varphi_0)\leqslant\min\left\{\frac{\omega_n}{|b|^n(n+1)},\, \frac{\bar{R}^n}{(n+1)2^n}\omega_n\right\},\,\,\,\,for\,\,b\,\,\,\,imaginary.
\endaligned
\end{array}
\right.
\end{equation}
Then, there is a constant $C$ which depends only on $n$, $p$ and $\bar{R}$(maybe also depend upon $|b|$) such that for every covariant tensor $T$ there hold
\begin{equation}\label{5.66}
\left\{
\begin{array}{llll}
\aligned
&||T||_{L^{\frac{np}{n-p}}(M_t)}\leqslant C(||\nabla T||_{L^p(M_t)}+||T||_{L^p(M_t)})&\,\,\,\,for\,\,\,\,p\in[1,n);\\
&||T||_{L^{\infty}(M_t)}\leqslant C(||\nabla T||_{L^p(M_t)}+||T||_{L^p(M_t)})&\,\,\,\,for\,\,\,\,p\in(n,\infty).
\endaligned
\end{array}
\right.
\end{equation}
\end{lem}

\begin{lem}\label{lemma5.10}
Under the same hypotheses as in Lemma \ref{lemma5.9}, there is a constant $C$ which only depends on $n$, $p$, $\bar{R}$, $j$, $s$, $q$, $p$ (may be also depend on $|b|$) such that for any covariant tensor $T$, there hold true
\begin{equation}\label{5.67}
\left\{
\begin{array}{llll}
\aligned
&||\nabla^jT||_{L^p(M_t)}\leqslant C||T||_{W^{s,q}(M_t)}\,\,\,\,\quad \mbox{with}\,\,\,\,\frac{1}{p}=\frac{1}{q}-\frac{s-j}{n}>0;\\
&||\nabla^jT||_{L^{\infty}(M_t)}\leqslant C||T||_{W^{s,q}(M_t)}\,\,\,\,\quad \mbox{with}\,\,\,\,\frac{1}{q}-\frac{s-j}{n}<0.
\endaligned
\end{array}
\right.
\end{equation}
\end{lem}

\begin{lem}\label{lemma5.11}
Under the same hypotheses as in Lemma \ref{lemma5.9}, there is a constant $C$ which only depends on $n,p,\bar{R},j,s,q,p,r$(may be also depend on $|b|$) such that for any covariant tensor and for all $0\leqslant j\leqslant s$, $p,q,r\in[1,\infty)$ and $a\in[j/s,1]$, there holds
\begin{equation}\label{5.68}
||\nabla^jT||_{L^p(M_t)}\leqslant C||T||^a_{W^{s,q}(M_t)}||T||^{1-a}_{L^r(M_t)}
\end{equation}
with
\begin{equation}\label{5.69}
\frac{1}{p}=\frac{j}{n}+a(\frac{1}{q}-\frac{s}{n})+\frac{1-a}{r}.
\end{equation}
If $$\frac{j}{n}+a(\frac{1}{q}-\frac{s}{n})+\frac{1-a}{r}<0,$$ then $(\ref{5.68})$ holds for all $p\in[1,\infty)$.
\end{lem}

\section{The uniform bound for the covariant derivatives of the second fundamental form}
In this section we will provide the proof of Theorem 1.1. For this goal we need to take a contradiction argument. If the conclusions of Theorem 1.1 are false, then the evolving hypersurface will blow up at some time. Suppose that at a certain time $T>0$ the evolving hypersurface develops a singularity. Then, for the family $\{M_t\}_{t\in[0,T)}$, we are going to use the time-independent inequality $(\ref{5.68})$ to show the following uniform estimates
\[||\nabla^pA||_{L^{\infty}(M_t)}\leqslant C_p<\infty\]
for any $t\in[0,T)$ and all $p\in \mathbb{N}$. For this purpose, Our strategy is to compute $$\frac{d}{dt}||\nabla^pA||^2_{L^2(M_t)}$$
and derive an ordinary differential inequality with respect to $||\nabla^pA||^2_{L^2(M_t)}$ by the interpolation inequalities discussed in the above sections. Then, by the Gronwall inequality, one are able to obtain an upper bound of $||\nabla^pA||^2_{L^2(M_t)}$.\medskip

First we derive the evolution equations for $g$, $g^{-1}$, $\nu$, $\Gamma^i_{jk}$ and $A$. Recalling the definition of $X$ and substituting $(\ref{4.1})$ into $(\ref{3.7}),(\ref{3.10}),(\ref{3.11})$ and $(\ref{3.12})$ respectively, we get
\begin{equation}\label{6.1}
\frac{\partial g_{ij}}{\partial t}=-2E_m(\varphi)h_{ij},
\end{equation}
\begin{equation}\label{6.2}
\frac{\partial g^{ij}}{\partial t}=2g^{is}h_{sl}g^{lj}E_m(\varphi),
\end{equation}
\begin{equation}\label{6.3}
\nabla_t\nu=\mbox{grad}^ME_m(\varphi)
\end{equation}
and
\begin{equation}\label{6.4}
\frac{\partial\Gamma^i_{jk}}{\partial t}=-3(\nabla E_m\ast A+E_m\ast\nabla A).
\end{equation}

For the convenience of later calculation, we need to analyzing the specific expression of $E_m(\varphi)$ as follows
\begin{equation}\label{6.5}
\aligned
E_m(\varphi)=&\,\mathfrak{q}^{2m+1}(A,\nabla\nu)+\mathfrak{q}^1(A)+R^{2m}_1(\nabla\varphi,\nu)+2(-1)^m\Delta^mH\\
&+R_2^{2m+2}(\nabla\varphi,\nu)+\sum\limits_{i+j+k=2m-2}C_{ijk}\langle\nabla^{i+1}\varphi|\nabla^{j+1}\nu\rangle\ast\nabla^k\mathfrak{R}\\
&+\sum\limits_{i+j=2m-2}C_{ij}\langle\nabla^{i+1}\nu|\nabla\varphi\rangle\ast\nabla^j(\mathfrak{R}+A\otimes A).
\endaligned
\end{equation}
As in \cite{M} we hope that the last two terms on the right-hand side of $(\ref{6.5})$ do not contain $\nabla^{i+1}\varphi(i\in\mathbb{N})$.

Firstly, we consider $\langle\nabla^{i+1}\nu|\nabla\varphi\rangle$. Our strategy is to carry derivatives from $\nu$ to $\varphi$.
Since
\begin{equation}\label{6.6}
\langle\nu|\nabla\varphi\rangle=0,
\end{equation}
differentiating the two sides of $(\ref{6.6})$ $i+1$ times, we get
\begin{equation}\label{6.7}
\nabla^{i+1}\langle\nu|\nabla\varphi\rangle=0.
\end{equation}
So, recalling $(\ref{4.3})$, we have
\begin{equation}\label{6.8}
\aligned
\langle\nabla^{i+1}\nu|\nabla\varphi\rangle
=&-\sum\limits_{p=1}^{i+1}\binom{i+1}{p}\langle\nabla^{p+1}\varphi|\nabla^{i+1-p}\nu\rangle\\
=&\sum\limits_{p=1}^{i+1}\binom{i+1}{p}\langle\nabla^{p-1}(A\otimes\nu)|\nabla^{i+1-p}\nu\rangle\\
=&\sum\limits_{p=1}^{i+1}\sum\limits_{a=0}^{p-1}\binom{i+1}{p}\binom{p-1}{a}\nabla^{p-1-a}A\otimes\langle\nabla^a\nu|\nabla^{i+1-p}\nu\rangle.
\endaligned
\end{equation}
For the case $i=0$,
\begin{equation}\label{6.9}
\langle\nabla\nu|\nabla\varphi\rangle=A=\mathfrak{q}^1(A);
\end{equation}
For $i=1$,
\begin{equation}\label{6.10}
\langle\nabla^2\nu|\nabla\varphi\rangle=\nabla A=\mathfrak{q}^2(A);
\end{equation}
For $i\geqslant2$,
\begin{equation}\label{6.11}
\aligned
\langle\nabla^{i+1}\nu|\nabla\varphi\rangle
=&\sum\limits_{p=2}^i\sum\limits_{a=0}^{p-1}\binom{i+1}{p}\binom{p-1}{a}\nabla^{p-1-a}A\otimes\langle\nabla^a\nu|\nabla^{i+1-p}\nu\rangle\\
&+(i+1)A\otimes\langle\nu|\nabla^i\nu\rangle+\sum\limits_{a=0}^i\binom{i}{a}\nabla^{i-a}A\otimes\langle\nabla^a\nu|\nu\rangle\\
=&\sum\limits_{p=2}^i\sum\limits_{a=1}^{p-1}\binom{i+1}{p}\binom{p-1}{a}\nabla^{p-1-a}A\otimes\langle\nabla^a\nu|\nabla^{i+1-p}\nu\rangle\\
&+\sum\limits_{p=2}^i\binom{i+1}{p}\nabla^{p-1}A\otimes\langle\nu|\nabla^{i+1-p}\nu\rangle\\
&+(i+1)A\otimes\langle\nu|\nabla^i\nu\rangle+\sum\limits_{a=0}^i\binom{i}{a}\nabla^{i-a}A\otimes\langle\nabla^a\nu|\nu\rangle.
\endaligned
\end{equation}
For $i=2$, recalling $(\ref{4.2})$ we have
\begin{equation}\label{6.12}
\aligned
\langle\nabla^3\nu|\nabla\varphi\rangle
=&3A\otimes\langle\nabla\nu|\nabla\nu\rangle+3A\otimes\langle\nu|\nabla^2\nu\rangle+\sum\limits_{a=0}^2\binom{2}{a}\nabla^{2-a}A\otimes\langle\nabla^a\nu|\nu\rangle\\
=&\nabla^2A-A\otimes\langle\nabla\nu|\nabla\nu\rangle=\nabla^2A-A\ast A\ast A=\mathfrak{q}^3(A).
\endaligned
\end{equation}
For $i\geqslant3$,
\begin{equation*}\label{6.13}
\aligned
\langle\nabla^{i+1}\nu|\nabla\varphi\rangle
=&\sum\limits_{p=2}^i\sum\limits_{a=1}^{p-1}\binom{i+1}{p}\binom{p-1}{a}\nabla^{p-1-a}A\otimes\langle\nabla^a\nu|\nabla^{i+1-p}\nu\rangle\\
&+(i+1)A\otimes\langle\nu|\nabla^i\nu\rangle+\sum\limits_{a=2}^i\binom{i}{a}\nabla^{i-a}A\otimes\langle\nabla^a\nu|\nu\rangle\\
&+\nabla^iA+\sum\limits_{p=2}^{i-1}\binom{i+1}{p}\nabla^{p-1}A\otimes\langle\nu|\nabla^{i+1-p}\nu\rangle.
\endaligned
\end{equation*}
Furthermore, we compute
\begin{equation}\label{6.13}
\aligned
\langle\nabla^{i+1}\nu|\nabla\varphi\rangle
=&\sum\limits_{p=2}^i\sum\limits_{a=1}^{p-1}\binom{i+1}{p}\binom{p-1}{a}\nabla^{p-1-a}A\otimes\langle\nabla^a\nu|\nabla^{i+1-p}\nu\rangle\\
&-\frac{1}{2}\sum\limits_{p=2}^{i-1}\sum\limits_{a=1}^{i-p}\binom{i+1}{p}\nabla^{p-1}A\otimes\langle\nabla^a\nu|\nabla^{i+1-p-a}\nu\rangle\\
&-\frac{i+1}{2}\sum\limits_{p=1}^{i-1}\binom{i}{p}A\otimes\langle\nabla^p\nu|\nabla^{i-p}\nu\rangle+\nabla^iA\\
&-\frac{1}{2}\sum\limits_{a=2}^i\sum\limits_{b=1}^{a-1}\binom{i}{a}\binom{a}{b}\nabla^{i-a}A\otimes\langle\nabla^{a-b}\nu|\nabla^b\nu\rangle\\
=&\mathfrak{q}^{i+1}(A,\nabla\nu).
\endaligned
\end{equation}
Secondly, for $j\geqslant1$, we consider $\langle\nabla^{i+1}\nu|\nabla^{j+1}\varphi\rangle$ which equals to
\begin{equation}\label{6.14}
-\langle\nabla^{i+1}\nu|\nabla^{j-1}(A\otimes\nu)\rangle=-\sum\limits_{p=0}^{j-1}\binom{j-1}{p}\nabla^{j-1-p}A\otimes\langle\nabla^{i+1}\nu|\nabla^p\nu\rangle.
\end{equation}
While $j=1$, $i=0$,
\begin{equation}\label{6.15}
\langle\nabla\nu|\nabla^2\varphi\rangle=0.
\end{equation}
While $j=1$, $i\geqslant1$,
\begin{equation}\label{6.16}
\aligned
\langle\nabla^{i+1}\nu|\nabla^2\varphi\rangle
=&-\langle\nabla^{i+1}\nu|A\otimes\nu\rangle
=-A\otimes\langle\nabla^{i+1}\nu|\nu\rangle\\
=&\frac{1}{2}\sum\limits_{p=1}^i\binom{i+1}{p}A\otimes\langle\nabla^p\nu|\nabla^{i+1-p}\nu\rangle\\
=&\mathfrak{q}^{i+2}(A,\nabla\nu).
\endaligned
\end{equation}
While $j\geqslant2$, $i=0$,
\begin{equation}\label{6.17}
\aligned
\langle\nabla\nu|\nabla^{j+1}\varphi\rangle
=&-\sum\limits_{p=0}^{j-1}\binom{j-1}{p}\nabla^{j-1-p}A\otimes\langle\nabla\nu|\nabla^p\nu\rangle\\
=&-\sum\limits_{p=1}^{j-1}\binom{j-1}{p}\nabla^{j-1-p}A\otimes\langle\nabla\nu|\nabla^p\nu\rangle\\
=&\mathfrak{q}^{j+1}(A,\nabla\nu).
\endaligned
\end{equation}
While $j\geqslant2$, $i\geqslant1$,
\begin{equation}\label{6.18}
\aligned
\langle\nabla^{i+1}\nu|\nabla^{j+1}\varphi\rangle
=&-\sum\limits_{p=1}^{j-1}\binom{j-1}{p}\nabla^{j-1-p}A\otimes\langle\nabla^{i+1}\nu|\nabla^p\nu\rangle\\
&-\nabla^{j-1}A\otimes\langle\nabla^{i+1}\nu|\nu\rangle\\
=&-\sum\limits_{p=1}^{j-1}\binom{j-1}{p}\nabla^{j-1-p}A\otimes\langle\nabla^{i+1}\nu|\nabla^p\nu\rangle\\
&+\frac{1}{2}\sum\limits_{p=1}^i\binom{i+1}{p}\nabla^{j-1}A\otimes\langle\nabla^p\nu|\nabla^{i+1-p}\nu\rangle\\
=&\mathfrak{q}^{i+j+1}(A,\nabla\nu).
\endaligned
\end{equation}
In conclusion, we have
\begin{equation}\label{6.19}
\langle\nabla^{i+1}\nu|\nabla^{j+1}\varphi\rangle=
\left\{
\begin{array}{llll}
\aligned
&\mathfrak{q}^{i+1}(A),\s & j=0,\s 0\leqslant i\leqslant2\\
&0, &j=1,\s i=0\\
&\mathfrak{q}^{i+j+1}(A,\nabla\nu), &\mbox{otherwise.}
\endaligned
\end{array}
\right.
\end{equation}
So, we have
\begin{equation}\label{6.20}
\aligned
E_m(\varphi)=&\mathfrak{q}^{2m+1}(A,\nabla\nu)+\mathfrak{q}^1(A)+R^{2m}_1(\nabla\varphi,\nu)+2(-1)^m\Delta^mH\\
&+R_2^{2m+2}(\nabla\varphi,\nu)+\mathfrak{q}^{2m+1}(A,\mathfrak{R})+\mathfrak{q}^{2m+1}(A,\nabla\nu,\mathfrak{R})\\
&+\mathfrak{q}^{2m+1}(A,\mathfrak{R}+A\otimes A)+\mathfrak{q}^{2m+1}(A,\nabla\nu,\mathfrak{R}+A\otimes A).
\endaligned
\end{equation}

Let us return to compute the evolution equation of $A$. By the same way as in \cite{M}, we know
\begin{equation}\label{6.21}
\aligned
\frac{\partial h_{ij}}{\partial t}=&E_m(\varphi)\langle R^N(\nu,\nabla_i\varphi)\nabla_j\varphi|\nu\rangle+\nabla_{ij}E_m(\varphi)-E_m(\varphi)h_{is}g^{sl}h_{lj}\\
=&E_m(\varphi)(R^2_4(\nabla\varphi,\nu)-A\ast A)+2(-1)^m\nabla_{ij}(\Delta^mH)\\
 &+\mathfrak{q}^{2m+3}(A,\nabla\nu)+\mathfrak{q}^3(A)+R_1^{2m+2}(\nabla\varphi,\nu)+R_2^{2m+4}(\nabla\varphi,\nu)\\
 &+\mathfrak{q}^{2m+3}(A,\mathfrak{R})+\mathfrak{q}^{2m+3}(A,\nabla\nu,\mathfrak{R})+\mathfrak{q}^{2m+3}(A,\mathfrak{R}+A\otimes A)\\
 &+\mathfrak{q}^{2m+3}(A,\nabla\nu,\mathfrak{R}+A\otimes A).
\endaligned
\end{equation}
Recalling $(\ref{2.11})$, we have
\begin{equation}\label{6.22}
\nabla_{ij}H=\Delta h_{ij}+\mathfrak{q}^3(A,A,A)+\mathfrak{q}^3(A,\mathfrak{R})+R^4_3(\varphi,\nu).
\end{equation}
Since
\begin{equation}\label{6.23}
\nabla_{ij}(\Delta^mH)=\Delta^m\nabla_{ij}H+\mathfrak{q}^{2m+3}(A,A),
\end{equation}
we conclude
\begin{equation}\label{6.24}
\aligned
\frac{\partial h_{ij}}{\partial t}=&\,E_m(\varphi)(R^2_4(\nabla\varphi,\nu)-A\ast A)+2(-1)^m\Delta^{m+1}h_{ij}\\
&+\mathfrak{q}^{2m+3}(A,A,A)+\mathfrak{q}^{2m+3}(A,\mathfrak{R})\\
&+R^{2m+4}_3(\nabla\varphi,\nu)+\mathfrak{q}^{2m+3}(A,A)\\
&+\mathfrak{q}^{2m+3}(A,\nabla\nu)+\mathfrak{q}^3(A)+R_1^{2m+2}(\nabla\varphi,\nu)\\
&+R_2^{2m+4}(\nabla\varphi,\nu)+\mathfrak{q}^{2m+3}(A,\nabla\nu,\mathfrak{R})\\
&+\mathfrak{q}^{2m+3}(A,\mathfrak{R}+A\otimes A)+\mathfrak{q}^{2m+3}(A,\nabla\nu,\mathfrak{R}+A\otimes A).
\endaligned
\end{equation}
\begin{lem}\label{lemma6.1}
The covariant derivatives of $A$ satisfy following evolution equation
\begin{eqnarray}\label{6.25}
\begin{array}{ll}
\frac{\partial}{\partial t}\nabla^kh_{ij}=&2(-1)^m\Delta^{m+1}\nabla^kh_{ij}+\mathfrak{q}^{2m+k+3}(A,A)+\mathfrak{q}^{2m+k+3}(A,A,A)\\
&+\mathfrak{q}^{2m+k+3}(A,\mathfrak{R})+R^{2m+k+4}_3(\nabla\varphi,\nu)\\
&+\mathfrak{q}^{2m+k+3}(A,\nabla\nu)+\mathfrak{q}^{k+3}(A)\\
&+R_1^{2m+k+2}(\nabla\varphi,\nu)+R_2^{2m+k+4}(\nabla\varphi,\nu)+\mathfrak{q}^{2m+k+3}(A,\nabla\nu,\mathfrak{R})\\
&+\mathfrak{q}^{2m+k+3}(A,\mathfrak{R}+A\otimes A)\\
&+\mathfrak{q}^{2m+k+3}(A,\nabla\nu,\mathfrak{R}+A\otimes A)+\mathfrak{q}^{2m+k+3}(A,A,A,\nabla\nu)\\
&+\mathfrak{q}^{k+3}(A,A,A)+\sum\limits_{l=0}^k\mathfrak{q}^{l+2}(A,A)\ast R_1^{2m+k-l}(\nabla\varphi,\nu)\\
&+\sum\limits_{l=0}^k\mathfrak{q}^{l+2}(A,A)\ast R_2^{2m+k+2-l}(\nabla\varphi,\nu)\\
&+\mathfrak{q}^{2m+k+3}(A,A,A,\mathfrak{R})+\mathfrak{q}^{2m+k+3}(A,A,A,\nabla\nu,\mathfrak{R})\\
&+\mathfrak{q}^{2m+k+3}(A,A,\mathfrak{R}+A\otimes A)+\mathfrak{q}^{2m+k+3}(A,A,A,\nabla\nu,\mathfrak{R}+A\otimes A)\\
&+\sum\limits_{a=0}^kR_4^{a+2}(\nabla\varphi,\nu)\ast[\mathfrak{q}^{2m+k-a+1}(A)+\mathfrak{q}^{2m+k-a+1}(A,\nabla\nu)\\
&+\mathfrak{q}^{k-a+1}(A)+R_1^{2m+k-a}(\nabla\varphi,\nu)+R_2^{2m+k-a+2}(\nabla\varphi,\nu)\\
&+\mathfrak{q}^{2m+k+1-a}(A,\mathfrak{R})+\mathfrak{q}^{2m+k+1-a}(A,\nabla\nu,\mathfrak{R})\\
&+\mathfrak{q}^{2m+k+3-a}(A,\mathfrak{R}+A\otimes A)+\mathfrak{q}^{2m+k+3-a}(A,\nabla\nu,\mathfrak{R}+A\otimes A)]\\
&+\sum\limits_{a=0}^k\mathfrak{q}^{a+2}(A,A)\ast R_1^{2m+k-a}(\nabla\varphi,\nu)\\
&+\sum\limits_{a=0}^k\mathfrak{q}^{a+2}(A,A)\ast R_2^{2m+k-a+2}(\nabla\varphi,\nu).
\end{array}
\end{eqnarray}
\end{lem}

\begin{proof} In Lemma 3.1 of \cite{M}, we take $T=A$. Since $a(X)=-2E_m(\varphi)A$, we get
\begin{eqnarray}\label{6.26}
\frac{\partial}{\partial t}\nabla^kh_{ij}=\nabla^k\frac{\partial h_{ij}}{\partial t}+\mathfrak{p}_k(A,A,E_m(\varphi)).
\end{eqnarray}
By Lemma 7.3 of \cite{M}, one can obtain
\begin{equation}\label{6.27}
\nabla^k\Delta^{m+1}h_{ij}=\Delta^{m+1}\nabla^kh_{ij}+\mathfrak{q}^{2m+k+3}(A,A).
\end{equation}
Then substituting $(\ref{6.24})$ into $(\ref{6.26})$, we take a direct calculation to get the required equality. This completes the proof.
\end{proof}

\begin{lem}\label{lemma6.2}
The following formula holds,
\begin{eqnarray}\label{6.28}
&&\frac{d}{dt}\int_M|\nabla^kA|^2\,d\mu_t\nonumber\\
&=&-4\int_M|\nabla^{k+m+1}A|^2\,d\mu_t+\int_M\mathfrak{q}^{2m+2k+4}(A,A,A)\,d\mu_t\nonumber\\
&&+\int_M\mathfrak{q}^{2m+2k+4}(A,A,A,A)\,d\mu_t+\int_M\mathfrak{q}^{2m+2k+4}(A,A,\mathfrak{R})\,d\mu_t\nonumber\\
&&+\int_MR_3^{2m+k+4}(\nabla\varphi,\nu)\ast\nabla^kA\,d\mu_t+\int_M\mathfrak{q}^{2m+2k+4}(A,A,\nabla\nu)\,d\mu_t\nonumber\\
&&+\int_M\mathfrak{q}^{2k+4}(A,A)\,d\mu_t+\int_MR^{2m+k+2}_1(\nabla\varphi,\nu)\ast\nabla^kA\,d\mu_t\nonumber\\
&&+\int_MR^{2m+k+4}_2(\nabla\varphi,\nu)\ast\nabla^kA\,d\mu_t+\int_M\mathfrak{q}^{2m+2k+4}(A,A,\nabla\nu,\mathfrak{R})\,d\mu_t\nonumber\\
&&+\int_M\mathfrak{q}^{2m+2k+4}(A,A,\mathfrak{R}+A\otimes A)\,d\mu_t+\int_M\mathfrak{q}^{2m+2k+4}(A,A,\nabla\nu,\mathfrak{R}+A\otimes A)\,d\mu_t\nonumber\\
&&+\int_M\mathfrak{q}^{2m+2k+4}(A,A,A,A,\nabla\nu)\,d\mu_t+\int_M\mathfrak{q}^{2k+4}(A,A,A,A)\,d\mu_t\nonumber\\
&&+\sum\limits_{l=0}^k\int_M\mathfrak{q}^{l+k+3}(A,A,A)\ast R_1^{2m+k-l}(\nabla\varphi,\nu)\,d\mu_t\nonumber\\
&&+\sum\limits_{l=0}^k\int_M\mathfrak{q}^{l+k+3}(A,A,A)\ast R_2^{2m+k-l+2}(\nabla\varphi,\nu)\,d\mu_t\nonumber\\
&&+\int_M\mathfrak{q}^{2m+2k+4}(A,A,A,A,\mathfrak{R})\,d\mu_t+\int_M\mathfrak{q}^{2m+2k+4}(A,A,A,A,\nabla\nu,\mathfrak{R})\,d\mu_t\nonumber\\
&&+\int_M\mathfrak{q}^{2m+2k+4}(A,A,A,A,\mathfrak{R}+A\otimes A)\,d\mu_t\\
&&+\int_M\mathfrak{q}^{2m+2k+4}(A,A,A,A,\nabla\nu,\mathfrak{R}+A\otimes A)\,d\mu_t\nonumber\\
&&+\sum\limits_{a=0}^k\int_MR_4^{a+2}(\nabla\varphi,\nu)\ast[\mathfrak{q}^{2m+2k-a+2}(A,A)+\mathfrak{q}^{2m+2k-a+2}(A,A,\nabla\nu)\nonumber\\
&&+\mathfrak{q}^{2k-a+2}(A,A)+R_1^{2m+k-a}(\nabla\varphi,\nu)\ast\nabla^kA+R_2^{2m+k-a+2}(\nabla\varphi,\nu)\ast\nabla^kA\nonumber\\
&&+\mathfrak{q}^{2m+2k-a+2}(A,A,\mathfrak{R})+\mathfrak{q}^{2m+2k-a+2}(A,A,\nabla\nu,\mathfrak{R})\nonumber\\
&&+\mathfrak{q}^{2m+2k-a+4}(A,A,\mathfrak{R}+A\otimes A)+\mathfrak{q}^{2m+2k-a+4}(A,A,\nabla\nu,\mathfrak{R}+A\otimes A)]\,d\mu_t\nonumber\\
&&+\sum\limits_{a=0}^k\int_MR_1^{2m+k-a}(\nabla\varphi,\nu)\ast\mathfrak{q}^{k+a+3}(A,A,A)\,d\mu_t\nonumber\\
&&+\int_MR_1^{2m}(\nabla\varphi,\nu)\ast\mathfrak{q}^{2k+3}(A,A,A)\,d\mu_t\nonumber\\
&&+\int_MR_2^{2m+2}(\nabla\varphi,\nu)\ast\mathfrak{q}^{2k+3}(A,A,A)\,d\mu_t.\nonumber
\end{eqnarray}
\end{lem}

\begin{proof} Recalling $(\ref{3.1})$, we get
\begin{equation}\label{6.29}
\aligned
\frac{d}{dt}\int_M|\nabla^kA|^2\,d\mu_t=&2\int_Mg^{i_1j_1}\cdots g^{i_kj_k}g^{is}g^{jz}\frac{\partial h_{ij,i_1\cdots i_k}}{\partial t}h_{sz,j_1\cdots j_k}\,d\mu_t\\
&+\sum\limits_{l=1}^k\int_Mg^{i_1j_1}\cdots\frac{\partial g^{i_lj_l}}{\partial t}\cdots g^{i_kj_k}g^{is}g^{jz}h_{ij,i_1\cdots i_k}h_{sz,j_1\cdots j_k}\,d\mu_t\\
&+\int_Mg^{i_1j_1}\cdots g^{i_kj_k}\frac{\partial g^{is}}{\partial t}g^{jz}h_{ij,i_1\cdots i_k}h_{sz,j_1\cdots j_k}\,d\mu_t\\
&+\int_Mg^{i_1j_1}\cdots g^{i_kj_k}\frac{\partial g^{jz}}{\partial t}g^{is}h_{ij,i_1\cdots i_k}h_{sz,j_1\cdots j_k}\,d\mu_t\\
&+\int_M|\nabla^kA|^2div^N(X)\,d\mu_t.\\
\endaligned
\end{equation}
Since there holds by the definiton of $X$,
\begin{equation}\label{6.30}
\aligned
div^N(X)&=g^{ij}\langle\nabla_i\varphi|\nabla_j\nabla_t\varphi\rangle=g^{ij}\langle\nabla_i\varphi|\nabla_j(-E_m(\varphi)\nu)\rangle\\
&=g^{ij}\langle\nabla_i\varphi|-E_m(\varphi)\nabla_j\nu\rangle=-E_m(\varphi)div^N\nu\\
&=-E_m(\varphi)H,
\endaligned
\end{equation}
in view of  $(\ref{6.2})$ we have
\begin{equation}\label{6.31}
\frac{d}{dt}\int_M|\nabla^kA|^2\,d\mu_t=2\int_M\frac{\partial}{\partial t}\nabla^kA\ast\nabla^kA\,d\mu_t+(2k+5)\int_M\nabla^kA\ast\nabla^kA\ast A\ast E_m(\varphi)\,d\mu_t.
\end{equation}
By Lemma \ref{lemma6.1},
\begin{equation}\label{6.32}
2\int_M\frac{\partial}{\partial t}\nabla^kA\ast\nabla^kA\,d\mu_t=4(-1)^m\int_M\Delta^{m+1}\nabla^kA\ast\nabla^kA\,d\mu_t+remainder\,\,terms.
\end{equation}
Using divergence theorem, we get
\begin{equation}\label{6.33}
4(-1)^m\int_M\Delta^{m+1}\nabla^kA\ast\nabla^kA\,d\mu_t=-4\int_M|\nabla^{m+k+1}A|^2\,d\mu_t.
\end{equation}
By a direct calculation, we get the needed result. This completes the proof.
\end{proof}

In order to apply Gronwall inequality, we need to use $-\int_M|\nabla^{k+m+1}A|^2\,d\mu_t$ to control the other terms on the right-hand side of $(\ref{6.28})$. However, these terms maybe contain derivatives of orders which are higher than $k+m+1$. So, first of all, we should use divergence theorem to lower their orders.

For the following
$$\int_M\mathfrak{q}^{2m+2k+4}(A,A,A)\,d\mu_t,\s\s \int_M\mathfrak{q}^{2m+2k+4}(A,A,A,A)\,d\mu_t, \s\s \int_M\mathfrak{q}^{2m+2k+4}(A,A,\nabla\nu)\,d\mu_t$$
and $$\int_M\mathfrak{q}^{2m+2k+4}(A,A,A,A,\nabla\nu)\,d\mu_t,$$
by the analysis below and Proposition 7.4 of \cite{M}, we can use the divergence theorem to lower the highest derivatives, and then get the integrals of new polynomials which do not contain derivatives of orders higher than $k+m+1$. Moreover, if there is a derivative of order $k+m+1$ in an additive term, then the orders of all the other derivatives in this term must be lower than or equal to $k+m$.\medskip

For $\int_M\mathfrak{q}^{2k+4}(A,A)\,d\mu_t$ and $\int_M\mathfrak{q}^{2k+4}(A,A,A,A)\,d\mu_t$, by the same argument as above, we can transform them into the sums of integrals of polynomials whose terms do not contain derivatives of orders higher than $k+1$.

Now we want to treat with integrals of $(\ref{6.28})$ containing curvature tensors. For the following term $$\int_MR_3^{2m+k+4}(\nabla\varphi,\nu)\ast\nabla^kA\,d\mu_t,$$ recalling the definition of $R^s_3(\nabla\varphi,\nu)$ in Quantity 5 of Section \ref{section2} and the formulas $(\ref{4.2})$, $(\ref{4.3})$, $(\ref{4.4})$ and $(\ref{4.5})$, we take $\mathfrak{q}^s(A)$ out from curvature tensor. That is to say,
\begin{equation}\label{6.34}
\aligned
R_3^{2m+k+4}(\nabla\varphi,\nu)
=&\sum C_{a_1\cdots a_{\alpha}bcde}\mathfrak{q}^{a_1}(A)\ast\cdots\ast\mathfrak{q}^{a_{\alpha}}(A)\ast\mathfrak{q}^b(A)\ast\mathfrak{q}^c(A)\\
&\ast\mathfrak{q}^d(A)\ast\mathfrak{q}^e(A)\ast\langle(\nabla^{\alpha}R^N)(\omega_1,\cdots,\omega_{\alpha})(\psi,\rho)\delta|\eta\rangle
\endaligned
\end{equation}
where
\begin{equation}\label{6.35}
(a_1+1)+\cdots+(a_{\alpha}+1)+(b+1)+(c+1)+(d+1)+e=2m+k+4,
\end{equation}
and $\omega_i,\psi,\rho,\delta,\eta$ are either $\nabla\varphi$ or $\nu$. So
\begin{equation}\label{6.36}
\aligned
\int_MR_3^{2m+k+4}(\nabla\varphi,\nu)\ast\nabla^kA\,d\mu_t
=&\sum C_{a_1\cdots a_{\alpha}bcde}\int_M\mathfrak{q}^{a_1}(A)\ast\cdots\ast\mathfrak{q}^{a_{\alpha}}(A)\\
&\ast\mathfrak{q}^b(A)\ast\mathfrak{q}^c(A)\ast\mathfrak{q}^d(A)\ast\mathfrak{q}^e(A)\ast\nabla^kA\\
&\ast\langle(\nabla^{\alpha}R^N)(\omega_1,\cdots,\omega_{\alpha})(\psi,\rho)\delta|\eta\rangle\,d\mu_t.
\endaligned
\end{equation}
Since $\alpha$ may be zero, we have
\begin{equation}\label{6.37}
b+c+d+e\leqslant 2m+k+1.
\end{equation}
Because the polynomial on the right-hand side of $(\ref{6.36})$ has at least two terms, using divergence theorem to lower the order of derivatives in $(\ref{6.36})$, we get an integral of a new polynomial which does not contain derivatives of order higher than $k+m$. That is to say,
\begin{equation}\label{6.38}
\aligned
&\int_MR_3^{2m+k+4}(\nabla\varphi,\nu)\ast\nabla^kA\,d\mu_t\\
=&\sum\tilde{C}_{l_1\cdots l_{\theta}\lambda}\int_M\mathfrak{q}^{l_1}(A)\ast\cdots\ast\mathfrak{q}^{l_{\theta}}(A)\ast\langle(\nabla^{\lambda}R^N)(\omega_1,\cdots,\omega_{\lambda})(\psi,\rho)\delta|\eta\rangle\,d\mu_t
\endaligned
\end{equation}
where $\theta\geqslant2$, for $1\leqslant i\leqslant\theta$,
\begin{equation}\label{6.39}
1\leqslant l_i\leqslant k+m+1
\end{equation}
and
\begin{equation}\label{6.40}
l_1+\cdots+l_{\theta}\leqslant2m+2k+2
\end{equation}
and $\omega_i$, $\psi$, $\rho$, $\delta$, $\eta$ are either $\nabla\varphi$ or $\nu$.

Since there hold true  that $(\ref{G:2})$ and
\begin{equation}\label{6.41}
|\nabla\varphi|=\sqrt{n},\s\,\,\,\s |\nu|=1,
\end{equation}
it follows that
\begin{equation}\label{6.42}
|\langle(\nabla^{\lambda}R^N)(\omega_1,\cdots,\omega_{\lambda})(\psi,\rho)\delta|\eta\rangle|\lesssim 1.
\end{equation}
So we obtain that
\begin{equation}\label{6.43}
|\int_MR_3^{2m+k+4}(\nabla\varphi,\nu)\ast\nabla^kA\,d\mu_t|\lesssim\sum\limits_j\int_M\prod\limits_{i=0}^{k+m}|\nabla^iA|^{\alpha_{ij}}\,d\mu_t
\end{equation}
with
\begin{equation}\label{6.44}
\sum\limits_i(i+1)\alpha_{ij}\leqslant2m+2k+2.
\end{equation}

The same method also works for:
\[\int_MR_1^{2m+k+2}(\nabla\varphi,\nu)\ast\nabla^kA\,d\mu_t,\quad\quad \int_MR_1^{2m+k-l}(\nabla\varphi,\nu)\ast\mathfrak{q}^{l+k+3}(A,A,A)\,d\mu_t,\]
 \[\int_MR_1^{2m}(\nabla\varphi,\nu)\ast\mathfrak{q}^{2k+3}(A,A,A)\,d\mu_t,\quad\quad \int_MR_1^{2m+k-a}(\nabla\varphi,\nu)\ast\mathfrak{q}^{a+k+3}(A,A,A)\,d\mu_t,\]

\[\int_MR_2^{2m+k+4}(\nabla\varphi,\nu)\ast\nabla^kA\,d\mu_t,\quad\quad
\int_MR_2^{2m+k-l+2}(\nabla\varphi,\nu)\ast\mathfrak{q}^{l+k+3}(A,A,A)\,d\mu_t,\]
\[\int_MR_2^{2m+2}(\nabla\varphi,\nu)\ast\mathfrak{q}^{2k+3}(A,A,A)\,d\mu_t.\]

\[\int_MR_4^{a+2}(\nabla\varphi,\nu)\ast\mathfrak{q}^{2m+2k-a+2}(A,A)\,d\mu_t,\] \[\int_MR_4^{a+2}(\nabla\varphi,\nu)\ast\mathfrak{q}^{2m+2k-a+2}(A,A,\nabla\nu)\,d\mu_t,\]
\[\int_MR_4^{a+2}(\nabla\varphi,\nu)\ast\mathfrak{q}^{2k-a+2}(A,A)\,d\mu_t,\]

\[\int_MR_4^{a+2}(\nabla\varphi,\nu)\ast R_1^{2m+k-a}(\nabla\varphi,\nu)\ast\nabla^kA\,d\mu_t,\]
and
\[\int_MR_4^{a+2}(\nabla\varphi,\nu)\ast R_2^{2m+k-a+2}(\nabla\varphi,\nu)\ast\nabla^kA\,d\mu_t.\]
They are all bounded (up to a universal constant) by
\[\sum\limits_j\int_M\prod\limits_{i=0}^{k+m}|\nabla^iA|^{\alpha_{ij}}\,d\mu_t\]
where
\[\sum\limits_i(i+1)\alpha_{ij}\leqslant2m+2k+2.\]

Now we are going to deal with $$\int_M\mathfrak{q}^{2m+2k+4}(A,A,\mathfrak{R})\,d\mu_t.$$
Note that if the polynomial contains a derivative (for example $\nabla^pA$ or $\nabla^p\mathfrak{R}$) of order $p\geqslant k+m+1$, then all the other derivatives must be of order lower than or equal to $k+m-1$, since the order of the polynomial is $2m+2k+4$ and there are at least three factors in every additive term. In this case, using divergence theorem to lower the highest derivative, we get the integral of a new polynomial which contain derivatives of order not higher than $k+m$.

The same approach also works for:
\[\int_M\mathfrak{q}^{2m+2k+4}(A,A,\nabla\nu,\mathfrak{R})\,d\mu_t,\quad\quad\int_M\mathfrak{q}^{2m+2k+4}(A,A,\mathfrak{R}+A\otimes A)\,d\mu_t,\]
\[\int_M\mathfrak{q}^{2m+2k+4}(A,A,\nabla\nu,\mathfrak{R}+A\otimes A)\,d\mu_t,\quad\quad\int_M\mathfrak{q}^{2m+2k+4}(A,A,A,A,\mathfrak{R})\,d\mu_t,\]
\[\int_M\mathfrak{q}^{2m+2k+4}(A,A,A,A,\nabla\nu,\mathfrak{R})\,d\mu_t,\quad\quad\int_M\mathfrak{q}^{2m+2k+4}(A,A,A,A,\mathfrak{R}+A\otimes A)\,d\mu_t,\]
\[\int_M\mathfrak{q}^{2m+2k+4}(A,A,A,A,\nabla\nu,\mathfrak{R}+A\otimes A)\,d\mu_t.\]
It means that we can use divergence theorem to lower the highest derivatives and get integrals of new polynomials whose orders of derivatives are not higher than $k+m$.

Now we are going to deal with $$\int_MR_4^{2+a}(\nabla\varphi,\nu)\ast\mathfrak{q}^{2m+2k-a+2}(A,A,\mathfrak{R})\,d\mu_t.$$
Recalling the definition of $R^s_4(\nabla\varphi,\nu)$ in Quantity 6 of Section 2 and our formulas $(\ref{4.2})$, $(\ref{4.3})$, $(\ref{4.4})$ and $(\ref{4.5})$ and taking $\mathfrak{q}^s(A)$ out from curvature tensor, we see that $R_4^{2+a}(\nabla\varphi,\nu)$ can be transformed into
\begin{eqnarray*}
\sum C_{a_1\cdots a_{\alpha}bcde}\mathfrak{q}^{a_1}(A)\ast\cdots\ast\mathfrak{q}^{a_{\alpha}}(A)\ast\mathfrak{q}^b(A)\ast\mathfrak{q}^c(A)\ast\mathfrak{q}^d(A)\ast\mathfrak{q}^e(A)\\
\ast\langle(\nabla^{\alpha}R^N)(\omega_1,\cdots,\omega_{\alpha})(\psi,\rho)\delta|\eta\rangle
\end{eqnarray*}
where
\begin{eqnarray*}
(a_1+1)+\cdots+(a_{\alpha}+1)+b+c+d+e=a,
\end{eqnarray*}
and $\omega_i$, $\psi$, $\rho$, $\delta$, $\eta$ are either $\nabla\varphi$ or $\nu$. Therefore
\begin{eqnarray*}
\aligned
&\int_MR_4^{2+a}(\nabla\varphi,\nu)\ast\mathfrak{q}^{2m+2k-a+2}(A,A,\mathfrak{R})\,d\mu_t\\
=&\sum C_{a_1\cdots a_{\alpha}bcde}\int_M\mathfrak{q}^{a_1}(A)\ast\cdots\ast\mathfrak{q}^{a_{\alpha}}(A)\ast\mathfrak{q}^b(A)\ast\mathfrak{q}^c(A)\ast\mathfrak{q}^d(A)\ast\mathfrak{q}^e(A)\\
&\ast\mathfrak{q}^{2m+2k+2-a}(A,A,\mathfrak{R})
\ast\langle(\nabla^{\alpha}R^N)(\omega_1,\cdots,\omega_{\alpha})(\psi,\rho)\delta|\eta\rangle\,d\mu_t.
\endaligned
\end{eqnarray*}
Note that if the above polynomial on the right-hand side contains a derivative (for example $\nabla^pA$ or $\nabla^p\mathfrak{R}$) of order $p\geqslant k+m+1$ in an additive term, then all the other derivatives in this term must be of order lower than or equal to $k+m-4$, since the order of the polynomial($=a_1+\cdots+a_{\alpha}+b+c+d+e+2m+2k+2-a$) is not larger than $2m+2k+2$ and there are at least four factors in every additive term. In this case, using divergence theorem to lower the highest derivative, we get the integral of a new polynomial which contains derivatives of order not higher than $k+m-1$.

The same method also works for:
\begin{eqnarray*}
\int_MR_4^{2+a}(\nabla\varphi,\nu)\ast\mathfrak{q}^{2m+2k-a+2}(A,A,\nabla\nu,\mathfrak{R})\,d\mu_t,
\end{eqnarray*}
\begin{eqnarray*}
\int_MR_4^{2+a}(\nabla\varphi,\nu)\ast\mathfrak{q}^{2m+2k-a+4}(A,A,\mathfrak{R}+A\otimes A)\,d\mu_t,
\end{eqnarray*}
and
\begin{eqnarray*}
\int_MR_4^{2+a}(\nabla\varphi,\nu)\ast\mathfrak{q}^{2m+2k-a+4}(A,A,\nabla\nu,\mathfrak{R}+A\otimes A)\,d\mu_t.
\end{eqnarray*}
It means that we can use the divergence theorem to lower the highest derivatives and get integrals of new polynomials whose orders of derivatives are not higher than $k+m$.

In conclusion, we can transform all the terms (except for $-4\int_M|\nabla^{k+m+1}A|^2\,d\mu_t$) on the right-hand side of $(\ref{6.28})$ into new integrals of polynomials whose additive terms contain derivatives of order not higher than $k+m+1$. Furthermore, if there is a derivative of order $k+m+1$ in an additive term of a polynomial, then the order of all the other derivatives in this term must be lower than or equal to $k+m$.

Now we are going to estimate the above integrals. Recalling (\ref{2.17}), it is easy to know that
\begin{equation}\label{6.45}
\mathfrak{R}=\langle R^N(\nabla\varphi,\nabla\varphi)\nabla\varphi|\nabla\varphi\rangle.
\end{equation}
It follows from $(\ref{G:2})$ that
\begin{equation}\label{6.46}
\aligned
|\nabla^k\mathfrak{R}|
=&|\sum C_{a_1\cdots a_{\alpha}bcde}\langle(\nabla^{\alpha}R^N)(\nabla^{a_1+1}\varphi,\cdots,\nabla^{a_{\alpha}+1}\varphi)(\nabla^{b+1}\varphi,\nabla^{c+1}\varphi)\nabla^{d+1}\varphi|\nabla^{e+1}\varphi\rangle|\\
\leqslant&\sum C_{a_1\cdots a_{\alpha}bcde}|\langle(\nabla^{\alpha}R^N)(\nabla^{a_1+1}\varphi,\cdots,\nabla^{a_{\alpha}+1}\varphi)(\nabla^{b+1}\varphi,\nabla^{c+1}\varphi)\nabla^{d+1}\varphi|\nabla^{e+1}\varphi\rangle|\\
\leqslant&\sum C_{a_1\cdots a_{\alpha}bcde}\bar{K}_{\alpha}|\nabla^{a_1+1}\varphi|\cdots|\nabla^{a_{\alpha}+1}\varphi|\cdot|\nabla^{b+1}\varphi|\cdot|\nabla^{c+1}\varphi|\cdot|\nabla^{d+1}\varphi|\cdot|\nabla^{e+1}\varphi|.
\endaligned
\end{equation}
where
\begin{equation}\label{6.47}
(a_1+1)+\cdots+(a_{\alpha}+1)+b+c+d+e=k,
\end{equation}
and $\alpha$ may be zero.
Since
\[|\nabla^s\varphi|\lesssim|\mathfrak{q}^{s-1}(A)|\]
and $\mathfrak{q}^0(A)$ is just a universal constant, we have
\begin{equation}\label{6.48}
\aligned
|\nabla^k\mathfrak{R}|
\leqslant\sum C_{a_1\cdots a_{\alpha}bcde}\bar{K}_{\alpha}|\mathfrak{q}^{a_1}(A)|\cdots|\mathfrak{q}^{a_{\alpha}}(A)|\cdot|\mathfrak{q}^b(A)|\cdot|\mathfrak{q}^c(A)|
\cdot|\mathfrak{q}^d(A)|\cdot|\mathfrak{q}^e(A)|.
\endaligned
\end{equation}

Firstly, we estimate $$\int_M\mathfrak{q}^{2m+2k+4}(A,A,\mathfrak{R})\,d\mu_t.$$
Since it can be transformed into
\begin{equation}\label{6.49}
\sum C_{i_1\cdots i_aj_1\cdots j_b}\int_M\nabla^{i_1}A\ast\cdots\ast\nabla^{i_a}A\ast\nabla^{j_1}\mathfrak{R}\ast\cdots\ast\nabla^{j_b}\mathfrak{R}\,d\mu_t
\end{equation}
where $b\geqslant1$, $a\geqslant2$, and
\begin{equation}\label{6.50}
(i_1+1)+\cdots+(i_a+1)+(j_1+2)+\cdots+(j_b+2)=2m+2k+4
\end{equation}
with
\[0\leqslant i_t\leqslant m+k \s\s\mbox{and}\s\s 0\leqslant j_s\leqslant m+k,\]
we have
\begin{equation}\label{6.51}
\Big|\int_M\mathfrak{q}^{2m+2k+4}(A,A,\mathfrak{R})\,d\mu_t\Big|\lesssim\sum\int_M|\nabla^{i_1}A|\cdots|\nabla^{i_a}A|
\cdot|\nabla^{j_1}\mathfrak{R}|\cdots|\nabla^{j_b}\mathfrak{R}|\,d\mu_t.
\end{equation}
Noting $(\ref{6.48})$ and substituting the upper bound of $\nabla^{j_s}\mathfrak{R}$ into $(\ref{6.51})$, one can obtain
\begin{equation}\label{6.52}
\Big|\int_M\mathfrak{q}^{2m+2k+4}(A,A,\mathfrak{R})\,d\mu_t\Big|\lesssim\sum\int_M|\mathfrak{q}^{l_1}A|\cdots|\mathfrak{q}^{l_{\theta}}A|\cdot|\nabla^{i_1}A|
\cdots|\nabla^{i_a}A|\,d\mu_t,
\end{equation}
where
\begin{equation}\label{6.53}
(l_1+2)+\cdots+(l_{\theta}+2)+(i_1+1)+\cdots+(i_a+1)\leqslant2m+2k+4,
\end{equation}
and
\[\theta\geqslant1,\,\,\,\,\s a\geqslant2,\,\,\,\,\s 0\leqslant l_t\leqslant m+k.\]
By the definition of $\mathfrak{q}^{l_t}(A)$, the right-hand side of $(\ref{6.52})$ is bounded by (up to a universal constant)
\[\sum\limits_j\int_M\prod\limits_{i=0}^{k+m}|\nabla^iA|^{\alpha_{ij}}\,d\mu_t,\]
where
\[\sum\limits_i(i+1)\alpha_{ij}\leqslant2m+2k+2.\]

Because
\[|\nabla^s\nu|\lesssim|\mathfrak{q}^s(A)|\]
and
\[|\nabla^s(\mathfrak{R}+A\otimes A)|\leqslant|\nabla^s\mathfrak{R}|+|\mathfrak{q}^{s+2}(A,A)|,\]
the above trick also works for
\[\int_M\mathfrak{q}^{2m+2k+4}(A,A,\nabla\nu,\mathfrak{R})\,d\mu_t,\quad\quad\int_M\mathfrak{q}^{2m+2k+4}(A,A,\mathfrak{R}+A\otimes A)\,d\mu_t,\]
\[\int_M\mathfrak{q}^{2m+2k+4}(A,A,\nabla\nu,\mathfrak{R}+A\otimes A)\,d\mu_t,\quad\quad\int_M\mathfrak{q}^{2m+2k+4}(A,A,A,A,\mathfrak{R})\,d\mu_t,\]
\[\int_M\mathfrak{q}^{2m+2k+4}(A,A,A,A,\nabla\nu,\mathfrak{R})\,d\mu_t,\quad\quad\int_M\mathfrak{q}^{2m+2k+4}(A,A,A,A,\mathfrak{R}+A\otimes A)\,d\mu_t,\]
\[\int_M\mathfrak{q}^{2k+4}(A,A,A,A)\,d\mu_t,\quad\quad\int_M\mathfrak{q}^{2m+2k+4}(A,A,A,A,\nabla\nu,\mathfrak{R}+A\otimes A)\,d\mu_t,\]
\[\int_MR_4^{2+a}(\nabla\varphi,\nu)\ast\mathfrak{q}^{2m+2k-a+2}(A,A,\mathfrak{R})\,d\mu_t,\]
\begin{eqnarray*}
\int_MR_4^{2+a}(\nabla\varphi,\nu)\ast\mathfrak{q}^{2m+2k-a+2}(A,A,\nabla\nu,\mathfrak{R})\,d\mu_t,
\end{eqnarray*}
\begin{eqnarray*}
\int_MR_4^{2+a}(\nabla\varphi,\nu)\ast\mathfrak{q}^{2m+2k-a+4}(A,A,\mathfrak{R}+A\otimes A)\,d\mu_t,
\end{eqnarray*}
\begin{eqnarray*}
\int_MR_4^{2+a}(\nabla\varphi,\nu)\ast\mathfrak{q}^{2m+2k-a+4}(A,A,\nabla\nu,\mathfrak{R}+A\otimes A)\,d\mu_t.
\end{eqnarray*}
And they are all bounded by (up to a universal constant)
\[\sum\limits_j\int_M\prod\limits_{i=0}^{k+m}|\nabla^iA|^{\alpha_{ij}}\,d\mu_t,\]
where
\[\sum\limits_i(i+1)\alpha_{ij}\leqslant2m+2k+4.\]

As for
\[\int_M\mathfrak{q}^{2m+2k+4}(A,A,A)\,d\mu_t,\quad\quad \int_M\mathfrak{q}^{2m+2k+4}(A,A,A,A)\,d\mu_t,\]
\[\int_M\mathfrak{q}^{2m+2k+4}(A,A,\nabla\nu)\,d\mu_t,\quad
\quad\int_M\mathfrak{q}^{2m+2k+4}(A,A,A,A,\nabla\nu)\,d\mu_t,\]
by the analysis in Proposition 7.4 of \cite{M}, it is easy to see that they are all bounded by (up to a universal constant)
\[\sum\limits_j\int_M\prod\limits_{i=0}^{k+m}|\nabla^iA|^{\alpha_{ij}}|\nabla^{k+m+1}A|^{\theta_j}\,d\mu_t,\]
where
\[\sum\limits_i(i+1)\alpha_{ij}+(k+m+2)\theta_j=2m+2k+4,\s\,\,\,\theta_j=0\s \mbox{or}\,\,1.\]
To sum up, we obtain
\begin{equation}\label{6.54}
\aligned
\frac{d}{dt}\int_M|\nabla^kA|^2\,d\mu_t\lesssim&-\int_M|\nabla^{k+m+1}A|^2\,d\mu_t+\sum\limits_j\int_M\prod\limits_{i=0}^{k+m}|\nabla^iA|^{\alpha_{ij}}\,d\mu_t\\
&+\int_M|\mathfrak{q}^{2k+4}(A,A)|\,d\mu_t+\sum\limits_j\int_M\prod\limits_{i=0}^{k+m}|\nabla^iA|^{\beta_{ij}}\,d\mu_t\\
&+\sum\limits_j\int_M\prod\limits_{i=0}^{k+m}|\nabla^iA|^{\gamma_{ij}}|\nabla^{k+m+1}A|^{\theta_j}\,d\mu_t,
\endaligned
\end{equation}
where
\begin{equation}\label{6.55}
\sum\limits_i(i+1)\alpha_{ij}\leqslant2m+2k+2,
\end{equation}
\begin{equation}\label{6.56}
\sum\limits_i(i+1)\beta_{ij}\leqslant2m+2k+4,
\end{equation}
and
\begin{equation}\label{6.57}
\sum\limits_i(i+1)\gamma_{ij}+(k+m+2)\theta_j=2m+2k+4,\s\,\,\s \theta_j=0\,\,\,\mbox{or}\,\,\,1.
\end{equation}

Note that the last term on the right-hand side of $(\ref{6.54})$ may contain derivatives of order $i=k+m+1$. To obtain Gronwall inequality, we hope that they vanish.

If $\theta_j=0$, we do nothing. If $\theta_j=1$, using Young inequality, we have
\begin{equation}\label{6.58}
\aligned
&\int_M\prod\limits_{i=0}^{k+m}|\nabla^iA|^{\gamma_{ij}}|\nabla^{k+m+1}A|\,d\mu_t\\
\leqslant&\frac{1}{2\varepsilon_j}\int_M\prod\limits_{i=0}^{k+m}|\nabla^iA|^{2\gamma_{ij}}\,d\mu_t+\varepsilon_j\int_M|\nabla^{k+m+1}A|^2\,d\mu_t.
\endaligned
\end{equation}
At this time,
\[\sum\limits_{i=0}^{k+m}(i+1)\cdot2\gamma_{ij}=2m+2k+4.\]
Choose $\{\varepsilon_j\}$ such that $\sum\limits_j\varepsilon_j$ is sufficiently small. So
\begin{equation}\label{6.59}
\aligned
\frac{d}{dt}\int_M|\nabla^kA|^2\,d\mu_t\lesssim &-\int_M|\nabla^{k+m+1}A|^2\,d\mu_t +\sum\limits_j\int_M\prod\limits_{i=0}^{k+m}|\nabla^iA|^{\beta_{ij}}\,d\mu_t\\
&+\int_M|\mathfrak{q}^{2k+4}(A,A)|\,d\mu_t,
\endaligned
\end{equation}
where $\beta_{ij}$ satisfies $(\ref{6.56})$. It is easy to see that
\begin{equation}\label{6.60}
\int_M|\mathfrak{q}^{2k+4}(A,A)|\,d\mu_t\lesssim\sum\limits_j\int_M\prod\limits_{i=0}^{k+1}|\nabla^iA|^{\eta_{ij}}\,d\mu_t,
\end{equation}
where
\begin{equation}\label{6.61}
\sum\limits_{i=0}^{k+1}(i+1)\eta_{ij}=2k+4.
\end{equation}

Set
\begin{equation}\label{6.62}
\rho_j:=\sum\limits_{i=0}^{k+m}(i+1)\beta_{ij}.
\end{equation}
By H\"{o}lder inequality, we have
\begin{equation}\label{6.63}
\int_M\prod\limits_{i=0}^{k+m}|\nabla^iA|^{\beta_{ij}}\,d\mu_t\leqslant\prod\limits_{i=0}^{k+m}||\nabla^iA||^{\beta_{ij}}_{L^{\frac{\rho_j}{i+1}}(\mu_t)}.
\end{equation}
In view of $(\ref{5.50})$, $(\ref{6.56})$ and $(\ref{6.62})$, one can easily get that
\begin{equation}\label{6.64}
||\nabla^iA||_{L^{\frac{\rho_j}{i+1}}(\mu_t)}\lesssim||\nabla^iA||_{L^{\frac{2m+2k+4}{i+1}}(\mu_t)}.
\end{equation}
By taking $q=2$, $s=k+m+1$, $r=n+1$, $j=i$ and $T=A$ in Lemma \ref{lemma5.11}, we have
\begin{equation}\label{6.65}
||\nabla^iA||_{L^{\frac{2m+2k+4}{i+1}}(\mu_t)}\lesssim||A||^{a_{ij}}_{W^{k+m+1,2}(\mu_t)}||A||^{1-a_{ij}}_{L^{n+1}(\mu_t)}
\end{equation}
with
\begin{equation}\label{6.66}
a_{ij}:=\frac{\frac{i+1}{2m+2k+4}-\frac{i}{n}-\frac{1}{n+1}}{\frac{1}{2}-\frac{k+m+1}{n}-\frac{1}{n+1}}\in\Big[\frac{i}{k+m+1},\,1\Big].
\end{equation}
By the same argument as the author dealt with the formula $(7.3)$ in \cite{M}, we know that $(\ref{6.66})$ is true for all $i\in\{0,1,\cdots,k+m\}$. Therefore, by Proposition \ref{lemma5.1} we obtain
\begin{eqnarray}\label{6.67}
\aligned
||A||_{W^{k+m+1,2}(\mu_t)}\lesssim&\sum\limits_{s=1}^{k+m+1}||\nabla^{k+m+1}A||_{L^2(\mu_t)}^{\frac{s}{k+m+1}}||A||_{L^2(\mu_t)}^{1-\frac{s}{k+m+1}}+||A||_{L^2(\mu_t)}\\
\lesssim&\sum\limits_{s=1}^{k+m+1}||\nabla^{k+m+1}A||_{L^2(\mu_t)}^{\frac{s}{k+m+1}}||A||_{L^{n+1}(\mu_t)}^{1-\frac{s}{k+m+1}}+||A||_{L^{n+1}(\mu_t)}\\
\lesssim&\sum\limits_{s=1}^{k+m+1}||\nabla^{k+m+1}A||_{L^2(\mu_t)}^{\frac{s}{k+m+1}}+1\\
\lesssim&\sum\limits_{s=1}^{k+m+1}||\nabla^{k+m+1}A||_{L^2(\mu_t)}+1\\
\lesssim&||\nabla^{k+m+1}A||_{L^2(\mu_t)}+1.
\endaligned
\end{eqnarray}
Here we have used Young inequality and $(\ref{5.6})$. So, it follows that
\begin{equation}\label{6.68}
\int_M\prod\limits_{i=0}^{k+m}|\nabla^iA|^{\beta_{ij}}\,d\mu_t\lesssim(1+||\nabla^{k+m+1}A||_{L^2(\mu_t)})^{\sum\limits_{i=0}^{k+m}a_{ij}\beta_{ij}}.
\end{equation}
Now, we need to prove that
\[\sum\limits_{i=0}^{k+m}a_{ij}\beta_{ij}<2.\]
Indeed,
\begin{equation}\label{6.69}
\sum\limits_{i=0}^{k+m}a_{ij}\beta_{ij}=\frac{\frac{\rho_j}{2m+2k+4}-\frac{\rho_j}{n}+\frac{\sum\limits_{i=0}^{k+m}\beta_{ij}}{n(n+1)}}{\frac{1}{2}-\frac{k+m+1}{n}-\frac{1}{n+1}}.
\end{equation}
Clearly,
\begin{equation}\label{6.70}
\sum\limits_{i=0}^{k+m}\beta_{ij}\geqslant\sum\limits_{i=0}^{k+m}\beta_{ij}\frac{i+1}{k+m+1}=\frac{\rho_j}{k+m+1}
\end{equation}
and the denominator of $(\ref{6.69})$ is negative. So
\begin{equation}\label{6.71}
\aligned
&\sum\limits_{i=0}^{k+m}a_{ij}\beta_{ij}
\leqslant\rho_j\frac{\frac{1}{2m+2k+4}-\frac{1}{n}+\frac{1}{n(n+1)(k+m+1)}}{\frac{1}{2}-\frac{k+m+1}{n}-\frac{1}{n+1}}\\
=&\frac{\rho_j}{2m+2k+4}\Big\{2-\frac{4}{(k+m+1)[2(k+m+1)(n+1)-n(n-1)]}\Big\}<2.
\endaligned
\end{equation}

For the term $\int_M|\mathfrak{q}^{2k+4}(A,A)|\,d\mu_t$, also it can be dealt with as Mantegazza did in \cite{M}.

In conclusion, we derive
\begin{equation}\label{6.72}
\frac{d}{dt}\int_M|\nabla^kA|^2\,d\mu_t\lesssim-\int_M|\nabla^{k+m+1}A|^2\,d\mu_t+(||\nabla^{k+m+1}A||^2_{L^2(\mu_t)}+1)^{1-\delta_0}.
\end{equation}
Young Inequality yields
\begin{equation}\label{6.73}
\frac{d}{dt}\int_M|\nabla^kA|^2\,d\mu_t\lesssim-\int_M|\nabla^{k+m+1}A|^2\,d\mu_t+1.
\end{equation}
Using Proposition \ref{lemma5.1} and Young inequality, we obtain
\begin{equation}\label{6.74}
||\nabla^kA||^2_{L^2(\mu_t)}\lesssim||\nabla^{k+m+1}A||^2_{L^2(\mu_t)}+1.
\end{equation}
Substituting $(\ref{6.74})$ into $(\ref{6.73})$, we get
\begin{equation}\label{6.75}
\frac{d}{dt}\int_M|\nabla^kA|^2\,d\mu_t\leqslant C(-\int_M|\nabla^kA|^2\,d\mu_t+1),
\end{equation}
where $C$ is a universal constant. It follows from the Gronwall's inequality that
\begin{equation}\label{6.76}
\aligned
||\nabla^kA||^2_{L^2(\mu_t)}\leqslant&||\nabla^kA||^2_{L^2(\mu_0)}\cdot\exp\{-Ct\}+1-\exp\{-Ct\}\\
\leqslant&||\nabla^kA||^2_{L^2(\mu_0)}+1.
\endaligned
\end{equation}
Furthermore, from $(\ref{5.67})$ we have
\begin{equation}\label{6.77}
||\nabla^kA||_{L^{\infty}(M_t)}\lesssim||A||_{W^{k+[\frac{n}{2}]+1,2}(M_t)}\leqslant C_k.
\end{equation}

\section{Long Time Existence}
Now let us focus on how to extend $\varphi(t)$ to $\varphi(T)$ smoothly. Recall that we have the following isometric embedding
\[\Xi:N\longrightarrow\mathbb{R}^{n+1+L}.\]
$\nabla$ is the connection induced by $\varphi(t)$ and $D$ denotes the connection induced by $\Xi\circ\varphi(t)$. As is described in the introduction, our aim is to prove that
\begin{equation}\label{6.78}
\max\limits_{M_t}\Big|D_{i_1}D_{i_2}\cdots D_{i_k}D_t^s\varphi\Big|\leqslant C_{s,k}.
\end{equation}

Firstly, we claim that, for all $k\geqslant0$, there holds
\begin{equation}\label{6.79}
\max\limits_{M_t}\Big|\nabla_{i_1}\nabla_{i_2}\cdots\nabla_{i_k}\varphi\Big|\leqslant C_k.
\end{equation}

\begin{proof} For the case $k=0$, from \[D_t\varphi=-E_m(\varphi)\nu,\] we have
\[\max\limits_{M_t}|D_t\varphi|=\max\limits_{M_t}|E_m(\varphi)|\leqslant C_m.\]
It follows that for any $t\in[0,T)$
\begin{equation}\label{6.80}
\max\limits_{M_t}|\varphi(t)|\leqslant\max\limits_{M_0}|\varphi_0|+C_m\cdot T
\end{equation}
where $T<\infty$.

By an induction argument, we can prove the following equality:
\begin{equation}\label{6.81}
\nabla_{i_1}\nabla_{i_2}\cdots\nabla_{i_k}\varphi=\nabla_{i_ki_{k-1}\cdots i_1}\varphi+\sum\limits_{p=1}^k\sum\limits_{\substack{j_1+\cdots+j_p+l\leqslant k-1,\\l\geqslant1}}\partial^{j_1}\Gamma\cdots\partial^{j_p}\Gamma\nabla^l\varphi,
\end{equation}
where $k\geqslant2$ and the second term on the right-hand side of $(\ref{6.81})$ contains contraction. Using the same approach as in page 173-174 of \cite{M}, we have
\begin{equation}\label{6.82}
\Big|\Big|\partial^k\frac{\partial\Gamma}{\partial t}\Big|\Big|_{L^{\infty}(M_t)}\leqslant C_k
\end{equation}
and
\begin{equation}\label{6.83}
||\partial^k\Gamma||_{L^{\infty}(M_t)}\leqslant C_k,
\end{equation}
and for any $k$ and any $s\in\mathbb{N}$
\begin{equation}\label{6.84}
\Big|\Big|\nabla^k\frac{\partial g}{\partial t}\Big|\Big|_{L^{\infty}(M_t)}\leqslant C_k
\end{equation}
and
\begin{equation}\label{6.85}
||\partial^k\nabla^sA||_{L^{\infty}(M_t)}\leqslant C_{k,s}.
\end{equation}
By Lemma 7.6 of \cite{M}, there exists a positive universal constant $\tilde{C}_1$ which depends on $T$ such that for all $t\in[0,T)$,
\begin{equation}\label{6.86}
\frac{1}{\tilde{C}_1}\leqslant g(t)\leqslant\tilde{C}_1.
\end{equation}
It is easy to see that
\begin{equation}\label{6.87}
\frac{1}{\tilde{C}_1}\leqslant g^{-1}(t)\leqslant\tilde{C}_1.
\end{equation}
So we have
\begin{equation}\label{6.88}
\sum\limits_{i_k,i_{k-1},\cdots,i_1}|\nabla_{i_ki_{k-1}\cdots i_1}\varphi|^2\leqslant\tilde{C}_1^k|\nabla^k\varphi|^2.
\end{equation}
Combining $(\ref{4.5})$, $(\ref{6.81})$, $(\ref{6.83})$, $(\ref{6.85})$, $(\ref{6.87})$ and $(\ref{6.88})$, we derive $(\ref{6.79})$. This completes the proof.
\end{proof}

From $(\ref{6.79})$, $(\ref{6.83})$ and the definition of covariant derivative, one can easily get
\begin{equation}\label{6.89}
\max\limits_{M_t}|\nabla_{i_1}\nabla_{i_2}\cdots\nabla_{i_k}\nabla_{j_1j_2\cdots j_l}\varphi|\leqslant C_{kl}.
\end{equation}

Since
\begin{equation}\label{6.90}
D_t\nabla_t^s\varphi=D_t(P(\varphi)\nabla_t^s\varphi)=\nabla_t^{s+1}\varphi+DP(\varphi)(\nabla_t\varphi)\nabla_t^s\varphi,
\end{equation}
By an induction argument, it is easy to see that there exist multi-linear forms $B_{a_1\cdots a_k}(y)$ on $TN(y\in N)$ such that
\begin{equation}\label{6.91}
D_t^s\varphi=\nabla_t^s\varphi+\sum B_{a_1\cdots a_k}(\varphi)(\nabla_t^{a_1}\varphi,\cdots,\nabla_t^{a_{k-1}}\varphi)(\nabla_t^{a_k}\varphi)
\end{equation}
where $k\geqslant2$; $1\leqslant a_i\leqslant s-1$ for $1\leqslant i\leqslant k$; and $a_1+\cdots+a_k=s$. Moreover, from $(\ref{6.91})$ we infer that
\begin{equation}\label{6.92}
\aligned
D_{b_1}\cdots D_{b_l}D_t^s\varphi=&\nabla_{b_1}\cdots\nabla_{b_l}\nabla_t^s\varphi\\
&+\sum B_{\vec{e}c_1\cdots c_r}(\varphi)(\tilde{\nabla}_{\vec{e}_1}\nabla_t^{c_1}\varphi,\cdots,\tilde{\nabla}_{\vec{e}_{r-1}}\nabla_t^{c_{r-1}}\varphi)(\tilde{\nabla}_{\vec{e}_r}\nabla_t^{c_r}\varphi).
\endaligned
\end{equation}
Here
\[\vec{e}:=(\vec{e}_1,\cdots,\vec{e}_r)=\sigma(b_1,b_2,\cdots,b_l)\] with $|\vec{e}_1|\geqslant0$, $\cdots$, $|\vec{e}_r|\geqslant0$, where $\sigma$ denotes a permutation; and $$c_1+\cdots+c_r=s$$ with $c_1\geqslant0$, $\cdots$, $c_{r-1}\geqslant0$ and $c_r\geqslant1$.

So, in order to prove $(\ref{6.78})$, we only have to estimate $\max\limits_{M_t}|\tilde{\nabla}_{\vec{e}_i}\nabla_t^{c_i}\varphi|$. For this goal, we need following two lemmas.
\begin{lem}\label{lemma6.3}
For $k\geqslant2$, there exists $C_k$ such that
\begin{equation}\label{6.93}
\max\limits_{M_t}|\nabla_{i_1}\nabla_{i_2}\cdots\nabla_{i_k}\nu|\leqslant C_k.
\end{equation}
Obviously, $(\ref{6.93})$ and $(\ref{6.83})$ imply the following estimate
\begin{equation}\label{6.94}
\max\limits_{M_t}|\nabla_{i_1}\nabla_{i_2}\cdots\nabla_{i_k}\nabla_{j_1j_2\cdots j_l}\nu|\leqslant C_{kl}.
\end{equation}
\end{lem}

\begin{proof}
Inductively, we can get
\begin{equation}\label{6.95}
\nabla_{i_1}\nabla_{i_2}\cdots\nabla_{i_k}\nu=\nabla_{i_ki_{k-1}\cdots i_1}\nu+\sum\limits_{p=1}^k\sum\limits_{\substack{j_1+\cdots+j_p+l\leqslant k-1\\l\geqslant1}}\partial^{j_1}\Gamma\cdots\partial^{j_p}\Gamma\cdot\nabla^l\nu,
\end{equation}
where $k\geqslant2$ and the second term on the right-hand side of $(\ref{6.95})$ concerns the tensor contraction with respect to the metric. From $(\ref{6.87})$ it is easy to know that
\begin{equation}\label{6.96}
\sum\limits_{i_k,i_{k-1},\cdots,i_1}|\nabla_{i_ki_{k-1}\cdots i_1}\nu|^2\leqslant\tilde{C}_1^k\cdot|\nabla^k\nu|^2.
\end{equation}
Combining $(\ref{4.4})$, $(\ref{6.83})$, $(\ref{6.85})$, $(\ref{6.95})$ and $(\ref{6.96})$, we get $(\ref{6.93})$. This completes the proof.
\end{proof}

\begin{lem}\label{lemma6.4}
There holds true
\begin{equation}\label{6.97}
\max\limits_{M_t}\Big|\frac{\partial^kE_m(\varphi)}{\partial x^{i_1}\cdots\partial x^{i_k}}\Big|\leqslant C_k.
\end{equation}
\end{lem}

\begin{proof}
We need to take a complicated computation to estimate all terms on the right hand side of $(\ref{6.20})$. Since these terms can be treated by almost the same method, for simplicity we only pick the following term to make estimation
\begin{equation}\label{6.98}
\max\limits_{M_t}\Big|\frac{\partial^kR_1^{2m}(\nabla\varphi,\nu)}{\partial x^{i_1}\cdots\partial x^{i_k}}\Big|.
\end{equation}
By the definition,
\begin{equation}\label{6.99}
\aligned
R_1^{2m}(\nabla\varphi,\nu)
=\sum C_{a_1\cdots a_qbcde}\langle(\nabla^qR^N)(\nabla^{a_1+1}\varphi,\cdots,\nabla^{a_q+1}\varphi)(\nabla^b\nu,\nabla^{c+1}\varphi)\nabla^d\nu|\nabla^e\nu\rangle.
\endaligned
\end{equation}
Because the right-hand side of $(\ref{6.99})$ concerns some contractions with respect to $g^{-1}$, we need to estimate
\[
||\partial^k(g^{-1})||_{L^{\infty}(M_t)}.
\]
Using $(7.7)$ of \cite{M} inductively, from our formula $(\ref{6.84})$ we get
\begin{equation}\label{6.100}
\Big|\Big|\partial^k\Big(\frac{\partial g}{\partial t}\Big)\Big|\Big|_{L^{\infty}(M_t)}\leqslant C_k
\end{equation}
which implies
\begin{equation}\label{6.101}
||\partial^kg||_{L^{\infty}(M_t)}\leqslant C_k(T).
\end{equation}
Noting
\begin{equation}\label{6.102}
\partial^k(g^{-1}\circ g)=0,
\end{equation}
one can obtain inductively that
\begin{equation}\label{6.103}
||\partial^k(g^{-1})||_{L^{\infty}(M_t)}\leqslant C_k(T).
\end{equation}
Substituting $(\ref{6.89})$, $(\ref{6.94})$ and $(\ref{6.99})$ into $(\ref{6.98})$, we get an upper bound of $(\ref{6.98})$ which depends upon $k$. This completes the proof.
\end{proof}

Now we turn to estimate $\max\limits_{M_t}|\tilde{\nabla}_{\vec{e}_i}\nabla_t^{c_i}\varphi|$. Indeed, we can get a stronger conclusion. It is the following theorem.

\begin{thm}\label{theorem6.5}
(1). For any $s\in\mathbb{N}$ and any $l\in\mathbb{N}^+$, we have
\begin{equation}\label{6.104}
\max\limits_{M_t}|\nabla_t^{a_1}\tilde{\nabla}_{\vec{i}_1}\nabla_t^{a_2}\tilde{\nabla}_{\vec{i}_2}\cdots\nabla_t^{a_k}\tilde{\nabla}_{\vec{i}_k}\varphi|\leqslant C_{s,l},
\end{equation}
\begin{equation}\label{6.105}
\max\limits_{M_t}|\nabla_t^{a_1}\tilde{\nabla}_{\vec{i}_1}\nabla_t^{a_2}\tilde{\nabla}_{\vec{i}_2}\cdots\nabla_t^{a_k}\tilde{\nabla}_{\vec{i}_k}(E_m(\varphi))|\leqslant C_{s,l},
\end{equation}
\begin{equation}\label{6.106}
\max\limits_{M_t}|\nabla_t^{a_1}\tilde{\nabla}_{\vec{i}_1}\nabla_t^{a_2}\tilde{\nabla}_{\vec{i}_2}\cdots\nabla_t^{a_k}\tilde{\nabla}_{\vec{i}_k}\nu|\leqslant C_{s,l},
\end{equation}
where $(a_1,a_2,\cdots,a_k)$ and $(\vec{i}_1,\vec{i}_2,\cdots,\vec{i}_k)$ satisfy
\[a_1+\cdots+a_k=s,\,\,\,\,\s a_1\geqslant0,\cdots,a_k\geqslant0\]
and
\[|\vec{i}_1|+\cdots+|\vec{i}_k|=l,\,\,\,\,\s |\vec{i}_1|\geqslant0,\cdots,|\vec{i}_k|\geqslant0.\]

(2). For any $s\in\mathbb{N}$, any $l\in\mathbb{N}^+$ and any $p\in\mathbb{N}$, there hold true
\begin{equation}\label{6.107}
||\partial_t^s\partial^l\nabla^pA||_{L^{\infty}(M_t)}\leqslant C_{s,l,p}
\end{equation}
and
\begin{equation}\label{6.108}
||\partial_t^s\partial^lg^{-1}||_{L^{\infty}(M_t)}\leqslant C_{s,l}.
\end{equation}
\end{thm}

\begin{proof} We intend to prove the theorem by an inductive argument for $s$.

For the case $s=0$, the above results have been shown. Assume that for all the numbers which are smaller than or equal to $s$, Theorem \ref{theorem6.5} is true.
For $s+1$, without loss of generality we suppose that $a_k\geqslant1$. We need to discuss the following two cases:\medskip

\noindent\textbf{Case \uppercase\expandafter{\romannumeral1}:} $|\vec{i}_k|=0$.

\textbf{(1).} For $\varphi$, we have
\begin{equation}\label{6.109}
\aligned
&\nabla_t^{a_1}\tilde{\nabla}_{\vec{i}_1}\cdots\nabla_t^{a_k}\varphi
=\nabla_t^{a_1}\tilde{\nabla}_{\vec{i}_1}\cdots\nabla_t^{a_{k-1}}\tilde{\nabla}_{\vec{i}_{k-1}}\nabla_t^{a_k-1}(E_m(\varphi)\nu)\\
\thickapprox&\sum\nabla_t^{c_1}\tilde{\nabla}_{\vec{p}_1}\cdots\nabla_t^{c_r}\tilde{\nabla}_{\vec{p}_r}(E_m(\varphi))\cdot\nabla_t^{b_1}
\tilde{\nabla}_{\vec{j}_1}\cdots\nabla_t^{b_l}\tilde{\nabla}_{\vec{j}_l}\nu
\endaligned
\end{equation}
with
\[c_1+\cdots+c_r+b_1+\cdots+b_l=s\]
and
\[(\vec{p}_1,\cdots,\vec{p}_r,\vec{j}_1,\cdots,\vec{j}_l)=\sigma(\vec{i}_1,\cdots,\vec{i}_{k-1}).\]
Using the induction hypotheses $(\ref{6.105})$ and $(\ref{6.106})$, we get
\[
\max\limits_{M_t}|\nabla_t^{a_1}\tilde{\nabla}_{\vec{i}_1}\nabla_t^{a_2}\tilde{\nabla}_{\vec{i}_2}\cdots\nabla_t^{a_k}\varphi|\leqslant C_{s+1,l}.
\]

\textbf{(2).} For $E_m(\varphi)$ we have
\begin{equation}\label{6.110}
\nabla_t^{a_1}\tilde{\nabla}_{\vec{i}_1}\nabla_t^{a_2}\tilde{\nabla}_{\vec{i}_2}\cdots\nabla_t^{a_k}(E_m(\varphi))=\nabla_t^{a_1}
\tilde{\nabla}_{\vec{i}_1}\nabla_t^{a_2}\tilde{\nabla}_{\vec{i}_2}\cdots\nabla_t^{a_k-1}(\frac{\partial E_m(\varphi)}{\partial t}).
\end{equation}
For simplicity, we only calculate a term of the right hand side of  $(\ref{6.20})$. From the definition of $R_2^s(\nabla\varphi,\nu)$, $(\ref{4.4})$ and $(\ref{4.5})$, we have
\begin{equation}\label{6.111}
\frac{\partial}{\partial t}R^{2m+2}_2(\nabla\varphi,\nu)\thickapprox\sum\frac{\partial}{\partial t}[\mathfrak{q}^{l_1}(A)\ast\cdots\ast\mathfrak{q}^{l_{\theta}}(A)\langle(\nabla^{\alpha}R^N)(\omega_1,\cdots,\omega_{\alpha})(\psi,\rho)\delta|\eta\rangle]
\end{equation}
where $\omega_i$, $\psi$, $\rho$, $\delta$, $\eta$ are either $\nabla\varphi$ or $\nu$. We rewrite $(\ref{6.2})$ as
\begin{equation}\label{6.112}
\frac{\partial g^{-1}}{\partial t}=2g^{-1}Ag^{-1}E_m(\varphi)
\end{equation}
and note that
\begin{equation}\label{6.113}
\aligned
&\frac{\partial}{\partial t}\langle(\nabla^{\alpha}R^N)(\omega_1,\cdots,\omega_{\alpha})(\psi,\rho)\delta|\eta\rangle\\
=&\sum\limits_{i=1}^{\alpha}\langle(\nabla^{\alpha}R^N)(\omega_1,\cdots,\nabla_t\omega_i,\cdots,\omega_{\alpha})(\psi,\rho)\delta|\eta\rangle\\
&+\langle(\nabla^{\alpha}R^N)(\omega_1,\cdots,\omega_{\alpha})(\nabla_t\psi,\rho)\delta|\eta\rangle\\
&+\langle(\nabla^{\alpha}R^N)(\omega_1,\cdots,\omega_{\alpha})(\psi,\nabla_t\rho)\delta|\eta\rangle\\
&+\langle(\nabla^{\alpha}R^N)(\omega_1,\cdots,\omega_{\alpha})(\psi,\rho)\nabla_t\delta|\eta\rangle\\
&+\langle(\nabla^{\alpha}R^N)(\omega_1,\cdots,\omega_{\alpha})(\psi,\rho)\delta|\nabla_t\eta\rangle.
\endaligned
\end{equation}
We rewrite $(\ref{6.3})$ as
\begin{equation}\label{6.114}
\nabla_t\nu=\partial E_m(\varphi)g^{-1}\nabla\varphi
\end{equation}
and note that
\begin{equation}\label{6.115}
\nabla_t\nabla\varphi=\nabla\nabla_t\varphi=\nabla(E_m(\varphi)\nu).
\end{equation}
Then, we can transform the derivative with respect to time variable $t$ of
$$\frac{\partial}{\partial t}\langle(\nabla^{\alpha}R^N)(\omega_1,\cdots,\omega_{\alpha})(\psi,\rho)\delta|\eta\rangle$$
into the derivatives with respect to space variables $x^i$. Besides, noting $$\mathfrak{q}^{l_i}(A)\thickapprox\sum\circledast_p\nabla^{j_p}A$$ and using Lemma \ref{lemma6.1}, we can transform $\frac{\partial}{\partial t}\nabla^{j_p}A$ into an expression which only contains the derivatives with respect to space variables $x^i$. For the other remaining terms of $E_m(\varphi)$, we take the same procedure.
In conclusion, $\frac{\partial}{\partial t}E_m(\varphi)$ can be transformed into an expression which contains some derivatives with respect to space variables $x^i$. Hence, the orders of the derivatives with respect to $t$ of $\nabla_t^{a_1}\tilde{\nabla}_{\vec{i}_1}\nabla_t^{a_2}\tilde{\nabla}_{\vec{i}_2}\cdots\nabla_t^{a_k}(E_m(\varphi))$
can be lowered to $s$. It follows from $(\ref{G:2})$ and the induction hypotheses $(\ref{6.104})-(\ref{6.108})$ that
\[\max\limits_{M_t}|\nabla_t^{a_1}\tilde{\nabla}_{\vec{i}_1}\nabla_t^{a_2}\tilde{\nabla}_{\vec{i}_2}\cdots\nabla_t^{a_k}(E_m(\varphi))|\leqslant C_{s+1,l}.\]

\textbf{(3).} For $\nu$, we have
\begin{equation}\label{6.116}
\nabla_t^{a_1}\tilde{\nabla}_{\vec{i}_1}\nabla_t^{a_2}\tilde{\nabla}_{\vec{i}_2}\cdots\nabla_t^{a_k-1}(\nabla_t\nu)
=\nabla_t^{a_1}\tilde{\nabla}_{\vec{i}_1}\nabla_t^{a_2}\tilde{\nabla}_{\vec{i}_2}\cdots\nabla_t^{a_k-1}(\partial E_m(\varphi)g^{-1}\nabla\varphi).
\end{equation}
Using the induction hypotheses $(\ref{6.104})$, $(\ref{6.105})$ and $(\ref{6.108})$, we get
\begin{equation}\label{6.117}
\max\limits_{M_t}|\nabla_t^{a_1}\tilde{\nabla}_{\vec{i}_1}\nabla_t^{a_2}\tilde{\nabla}_{\vec{i}_2}\cdots\nabla_t^{a_k}\nu|\leqslant C_{s+1,l}.
\end{equation}

\noindent\textbf{Case \uppercase\expandafter{\romannumeral2}:} $|\vec{i}_k|\geqslant1$.

Now, we have the following relations of changing the order of the derivatives of $\varphi$
\begin{equation}\label{6.118}
\nabla_t\tilde{\nabla}_{\vec{i}_k}\varphi=
\left\{
\begin{array}{llll}
\aligned
&\tilde{\nabla}_{\vec{i}_k}\nabla_t\varphi,&\,\,\,\,|\vec{i}_k|=1,\\
&\tilde{\nabla}_{\vec{i}_k}\nabla_t\varphi+\sum(\nabla^{\alpha}R^N)(\varphi)(\tilde{\nabla}_{\vec{a}_1}\varphi,\cdots,\tilde{\nabla}_{\vec{a}_{\alpha}}\varphi)
(\tilde{\nabla}_{\vec{c}}\nabla_t\varphi,\tilde{\nabla}_{\vec{b}}\varphi)\tilde{\nabla}_{\vec{e}}\varphi,&\,\,\,\,|\vec{i}_k|\geqslant2.
\endaligned
\end{array}
\right.
\end{equation}
Here $|\vec{b}|\geqslant1$ and $|\vec{e}|\geqslant1$.

For $\nu$ we also have
\begin{equation}\label{6.119}
\nabla_t\tilde{\nabla}_{\vec{i}_k}\nu=\tilde{\nabla}_{\vec{i}_k}\nabla_t\nu+\sum(\nabla^{\alpha}R^N)(\varphi)
(\tilde{\nabla}_{\vec{a}_1}\varphi,\cdots,\tilde{\nabla}_{\vec{a}_{\alpha}}\varphi)(\tilde{\nabla}_{\vec{c}}
\nabla_t\varphi,\tilde{\nabla}_{\vec{b}}\varphi)\tilde{\nabla}_{\vec{e}}\nu
\end{equation}
where $|\vec{b}|\geqslant1$.

Using the induction hypotheses $(\ref{6.104})$ and $(\ref{6.106})$, we infer from $(\ref{6.114})$ and $(\ref{4.1})$ that
\[\max\limits_{M_t}|\nabla_t^{a_1}\tilde{\nabla}_{\vec{i}_1}\nabla_t^{a_2}\tilde{\nabla}_{\vec{i}_2}\cdots\nabla_t^{a_k}\tilde{\nabla}_{\vec{i}_k}\varphi|\leqslant C_{s+1,l}\]
and
\[\max\limits_{M_t}|\nabla_t^{a_1}\tilde{\nabla}_{\vec{i}_1}\nabla_t^{a_2}\tilde{\nabla}_{\vec{i}_2}\cdots\nabla_t^{a_k}\tilde{\nabla}_{\vec{i}_k}\nu|\leqslant C_{s+1,l}.\]

\textbf{(4).} For $A$ we have
\[\partial^{s+1}_t\partial^l\nabla^pA=\partial_t^s\partial^l(\partial_t\nabla^pA).\]
By Lemma \ref{lemma6.1}, $\partial_t\nabla^pA$ can be transformed into an expression without derivative with respect to $t$. So, for $\partial^{s+1}_t\partial^l\nabla^pA$ we can lower the orders of derivatives with respect to $t$ to $s$. Then using induction hypotheses $(\ref{6.104})-(\ref{6.108})$, we get
\[
||\partial_t^{s+1}\partial^l\nabla^pA||_{L^{\infty}(M_t)}\leqslant C_{s+1,l,p}.
\]

\textbf{(5).} For $g^{-1}$ we have
\[\partial_t^{s+1}\partial^lg^{-1}=\partial_t^s\partial^l(\partial_tg^{-1})=\partial_t^s\partial^l(2g^{-1}Ag^{-1}E_m(\varphi)).\]
By the induction hypotheses $(\ref{6.105})$, $(\ref{6.107})$ and $(\ref{6.108})$, it is easy to see that
\[||\partial_t^{s+1}\partial^lg^{-1}||_{L^{\infty}(M_t)}\leqslant C_{s+1,l}.\]
This completes the proof.\end{proof}

Now we are in the position to prove our main theorem (Theorem 1.1).

\begin{proof} By the above discussion we know that the convergence $\varphi(t)\longrightarrow\varphi(T)$ , when $t\rightarrow T$, is in the $C^{\infty}$-topology. Then, using Theorem \ref{theorem 4.1} to restart the flow with $\varphi(T)$ as the initial hypersurface, we get a contradiction to the fact that $[0, T)$ is the maximal interval of existence. This tells us that $\varphi$ is global.
\end{proof}

\vspace{1cm}

Zonglin Jia

\noindent{\small\it Institute of Applied Physics and Computational Mathematics, China Academy of Engineering Physics, Beijing, 100088, P. R. China}

\noindent{\small\it Email: 756693084@qq.com}\\

Youde Wang

\noindent{\small\it College of Mathematics and Information Sciences, Guangzhou University, Guangzhou 510006, People¡¯s Republic of China}\\

\noindent{\small\it Institute of Mathematics, Academy of Mathematics and System Sciences Chinese Academy of Sciences, Beijing 100190, People¡¯s Republic of China}\\

\noindent{\small\it University of Chinese Academy of Sciences, Beijing, People¡¯s Republic of China}\\

\noindent{\small\it Email: wyd@math.ac.cn}

\end{document}